\documentclass[manuscript,screen,timestamp,nonacm]{acmart}

\AtBeginDocument{%
  \providecommand\BibTeX{{%
    \normalfont B\kern-0.5em{\scshape  i\kern-0.25em b}\kern-0.8em\TeX}}}

\usepackage{hyperref}
\usepackage{tabu}
\usepackage{stmaryrd,xcolor}
\usepackage{amsfonts,mathrsfs}
\usepackage[]{bbm}
\usepackage{dsfont}
\usepackage{array}
\usepackage{graphicx,todonotes}
\usepackage{scalerel}
\usepackage{mathtools}
\usepackage{fontawesome}
\usepackage{wasysym}

\usepackage{enumitem}[shortlabels]

\usepackage{url}

\newtheorem{theorem}{Theorem}[section]
\newtheorem{lemma}[theorem]{Lemma}
\newtheorem{corollary}[theorem]{Corollary}
\newtheorem{conjecture}[theorem]{Conjecture}
\newtheorem{proposition}[theorem]{Proposition}

\newtheorem{example}[theorem]{Example}
\newtheorem{remark}[theorem]{Remark}

\newtheorem{definition}{Definition}

\newcommand{\bD}{\mathfrak{D}}

\newcommand{\fs}{\mathfrak{S}}

\newcommand{\bA}{\mathfrak{A}}
\newcommand{\bB}{\mathfrak{B}}
\newcommand{\bC}{\mathfrak{C}}
\newcommand{\bF}{\mathfrak{F}}
\newcommand{\bG}{\mathfrak{G}}

\newcommand{\bR}{\mathfrak{R}}
\newcommand{\bS}{\mathfrak{S}}
\newcommand{\bT}{\mathfrak{T}}

\renewcommand{\nsim}{\bot}
\renewcommand{\precnsim}{
\smash{{}^{\bot}_{\prec}}
}

\newcommand{\red}[1]{\textcolor{red}{#1}}
\newcommand{\id}{\operatorname{id}}
\newcommand{\typ}{\operatorname{typ}}
\newcommand{\atyp}{\operatorname{qf-typ}}

\newcommand{\Pol}{\operatorname{Pol}}

\newcommand{\Aut}{\operatorname{Aut}}

\newcommand{\Age}{\operatorname{Age}}
\newcommand{\Can}{\operatorname{Can}}
\newcommand{\proj}{\operatorname{Proj}}
\newcommand{\Csp}{\operatorname{CSP}}
\newcommand{\Th}{\operatorname{Th}}

\newcommand{\eq}{\mathtt{EQ}}
\newcommand{\poprec}{\po  \, {\cap} \,{\prec}}
\newcommand{\drprec}{\dr  \, {\cap} \,{\prec}}
\newcommand{\dr}{\mathtt{DR}}
\newcommand{\op}{\mathtt{PO}}
\newcommand{\po}{\mathtt{PO}}
\newcommand{\pp}{\mathtt{PP}}
\newcommand{\ppi}{\mathtt{PPI}}

\begin{document}
%\title[The Unique Interpolation Property]{Canonical Polymorphisms of Ramsey Structures and the Unique Interpolation Property}
%\title{A Ramsey-theoretic approach to Complexity of Spatial Reasoning}
\title[A Complexity Dichotomy in Spatial Reasoning 
via Ramsey Theory]{A Complexity Dichotomy in Spatial Reasoning 
via Ramsey Theory}
\author{Manuel Bodirsky}
\authornote{Funded by the European Union (project POCOCOP, ERC Synergy grant No. 101071674). Views and opinions expressed are however those of the author(s) only and do not necessarily reflect those of the European Union or the European Research Council Executive Agency. Neither the European Union nor the granting authority can be held responsible for them.}
\affiliation{\institution{Institut f\"{u}r Algebra, Fakult\"at f\"ur Mathematik, TU Dresden}\country{Germany}}
\orcid{0000-0001-8228-3611}
\email{manuel.bodirsky@tu-dresden.de}
 \author{Bertalan Bodor}
\authornote{
%The research of the authors was 
Supported by the grants TKP2021-NVA-09 of the Ministry for Innovation and Technology, Hungary, and NKFIH-K138892.
%The research of authors was supported by grant TUDFO/47138-1/2019-ITM of the Ministry for Innovation and Technology, Hungary, the EU-funded Hungarian grant EFOP-3.6.2-16-2017-00015, and the NKFIH grant K138892.
}
% \authornotemark[1]
\affiliation{\institution{
Department of Algebra and Number Theory, University of Szeged}\country{Hungary}}
\email{bodor@server.math.u-szeged.hu}

\begin{abstract}
Constraint satisfaction problems (CSPs) for first-order reducts of finitely bounded homogeneous structures form a large class of computational problems that might exhibit a complexity dichotomy, P versus NP-complete. 
%Many %constraint satisfaction problems over infinite domains 
%of these CSPs can be solved in polynomial-time by 
A powerful method to obtain polynomial-time tractability results for such CSPs is a certain reduction to  polynomial-time tractable finite-domain CSPs defined over $k$-types, for a sufficiently large $k$.
We give sufficient conditions when this method can be applied and apply these conditions to obtain a new complexity dichotomy for CSPs of first-order expansions of the basic relations of the well-studied spatial reasoning formalism RCC5. 
%Let $\bC$ be a reduct of a homogeneous Ramsey structure 
%$\bA$ with finite relational signature. 
%We present characterisations of when the existence of a pseudo-Siggers polymorphism
%of $\bC$
%implies the existence of a pseudo-Siggers polymorphism of $\bC$ which is canonical over $\bA$. % (we need to assume that $\
%This has applications for the complexity of constraint satisfaction: 
%Barto and Pinsker showed that an $\omega$-categorical model-complete core structure $\bC$ 
%which does not have a pseudo-Siggers polymorphism has an NP-hard constraint satisfaction problem (CSP). On the other hand, 
%if $\bC$ is a reduct of a finitely bounded homogeneous structure $\bB$ and $\bC$ has a pseudo-Siggers polymorphism which is 
%canonical with respect to $\bB$,  
%then the CSP for $\bC$ can be solved in polynomial time, 
%by a reduction to a finite-domain CSP of Bodirsky and Mottet 
%and the finite-domain dichotomy theorem of Bulatov and Zhuk. 
%Our results allow to re-derive and generalise some of the existing complexity classifications for infinite-domain CSPs, for example for the class of 
%all countable structures with exponential labelled  growth. We also verify the infinite-domain tractability conjecture for first-order expansions of
%the basic relations of the spatial reasoning formalism RCC5.  
We also classify which of these CSPs can be expressed in Datalog. Our method relies on Ramsey theory; we prove that RCC5 has a Ramsey order expansion. 
\end{abstract}

\maketitle

% TEXT ABSTRACT
%Constraint satisfaction problems for first-order reducts of finitely bounded homogeneous structures form a large class of computational problems that might exhibit a complexity dichotomy, P versus NP-complete. A powerful method to obtain polynomial-time tractability results for such CSPs is a certain reduction to  polynomial-time tractable finite-domain CSPs defined over k-types, for a sufficiently large k. We give sufficient conditions when this method can be applied and illustrate how to use the general results to prove a new complexity dichotomy for first-order expansions of the basic relations of the  spatial reasoning formalism RCC5. We also classify which of these CSPs can be expressed in Datalog. Our method relies on Ramsey theory; we prove that RCC5 has a Ramsey order expansion. 

\tableofcontents




\red{This is a postprint version; the official article appeared online in \emph{ACM Transactions on Computation Theory}, March 2024.}

\section{Introduction}
%Many results in universal algebra only hold
%for algebras over a finite domain, for instance
%the important theorem about the existence of cyclic terms in Taylor algebras~\cite{Cyclic}.
%Every cyclic term is in particular a weak near unanimity term, and the existence of a weak near unanimity term is the starting point for Zhuk's algorithm~\cite{ZhukFVConjecture}, proving in 2017 the famous Feder-Vardi dichotomy conjecture for finite-domain constraint satisfaction problems~\cite{FederVardi} (an independent proof, also based on universal algebra, has been given by Bulatov~\cite{BulatovFVConjecture}). 
%General PR for CSPs, then for finite-domain CSPs,
%then for  reducts of homogeneous structures with finite relational signature. 
The \emph{Constraint satisfaction problem (CSP)}
of a relational structure $\bB$ is the computational problem of deciding whether a given finite structure $\bA$ has a homomorphism to $\bB$. 
If the structure $\bB$ is finite, then $\Csp(\bB)$ is in P or NP-complete; this was conjectured by Feder and Vardi~\cite{FederVardi} and proved by Bulatov~\cite{BulatovFVConjecture} and, independently, by Zhuk~\cite{ZhukFVConjecture}. 
Both approaches use concepts and methods from universal algebra; in particular, they use that
the computational complexity of $\Csp(\bB)$ is
fully determined by the set $\Pol(\bB)$ of \emph{polymorphisms} of $\bB$, which is the set of homomorphisms from $\bB^n$ to $\bB$ for $n \in {\mathbb N}$. 
In fact, $\Csp(\bB)$ is NP-complete if $\Pol(\bB)$ has a minor-preserving map to the clone $\proj$ of projections on a two-element set, 
and is in P otherwise.

Some of the results about finite structures can
be lifted to countably infinite structures whose automorphism group satisfies a certain finiteness condition, called \emph{oligomorphicity}: the requirement is that for every $n \in {\mathbb N}$ 
the componentwise action of the automorphism group on $n$-tuples has only finitely many orbits. 
Examples of such structures arise systematically in model theory: every homogeneous structure $\bB$ with a finite relational signature is of this type;
such structures will be called \emph{finitely homogeneous}.  
An additional finiteness condition is to require that
the class of finite substructures of $\bB$ is described by finitely many forbidden substructures, in which case $\bB$ is called \emph{finitely bounded}. The class of reducts of finitely bounded homogeneous structures is a huge generalisation of the class of all finite structures. 

There is also a generalisation of the Feder-Vardi dichotomy conjecture, which is still open, to reducts of finitely bounded homogeneous structures. In fact, there is a known NP-hardness condition for the CSPs of such structures which is conjectured to be at the 
border between
polynomial-time tractable and NP-complete CSPs~\cite{BPP-projective-homomorphisms,wonderland,BKOPP}.
%, similarly to the situation for finite-domain CSPs before 2017. 
The \emph{infinite-domain tractability conjecture} states that the CSP for 
every structure that does not satisfy the mentioned hardness condition is in P (details can be found in Section~\ref{sect:clones}). 
There are many classes of infinite-domain structures where the infinite-domain tractability  conjecture has been verified. Often, these classes consist of the \emph{first-order reducts} of some fixed underlying structure $\bB$. By a \emph{first-order reduct of $\bB$} we mean a reduct of the expansion of $\bB$ by all relations that are first-order definable in $\bB$. For example, the infinite-domain tractability conjecture
has been verified for 
\begin{enumerate}
\item %the class of 
all CSPs of structures preserved by all permutations~\cite{ecsps}, which is 
precisely the class of first-order reducts of pure sets (structures with no relations); 
\item %the class of 
all CSPs of structures with a highly set-transitive automorphism group~\cite{tcsps-journal} (i.e., for all finite subsets $X,Y$ with $|X|=|Y|$ there exists an automorphism which maps $X$ to $Y$),
which is precisely the class of first-order reducts 
of unbounded dense linear orders;
%\item The class of first-order reducts of the Random Graph~\cite{BodPin-Schaefer-both}.
\item %the class of 
all CSPs of first-order reducts of the homogeneous universal poset~\cite{posetCSP18};
%\item Equivalence relation with infinitely many infinite relations
\item %the class of
 all CSPs of first-order reducts of the binary branching C-relation~\cite{Phylo-Complexity};
\item %the class of 
all CSPs of first-order reducts of homogeneous graphs~\cite{BMPP16}; 
\item all CSPs of first-order reducts of \emph{unary structures}, i.e., structures with a signature that consists of finitely many unary relation symbols~\cite{BodMot-Unary}; 
\item all CSPs expressible in MMSNP~\cite{MMSNP}; MMSNP is a fragment of existential second-order logic introduced by Feder and Vardi~\cite{FederVardi}. 
\end{enumerate}

%While these classifications follow similar patterns, there is so far no general result that would imply them in a uniform way. 
%There have been 
%For some other fundamental classes of structures, for example for the first-order reducts of the  homogeneous universal permutation (see~\cite{LinmanPinsker}), the infinite-domain tractability has not been verified yet, and 
%the current methods appear to be too heavy for tackling it. 

In some, but not in all cases above the polynomial-time tractability results for $\Csp(\bB)$ can be obtained by reducing $\Csp(\bB)$ to $\Csp(\bT)$ where
$\bT$ is a certain finite structure whose domain is the set of complete $k$-types of $\bB$, for a sufficiently large $k$;
this method works if the \emph{clone of canonical polymorphisms of $\bB$} does not have a minor-preserving map to the projections~\cite{BodMot-Unary}. 
%Canonical polymorphisms induce polymorphisms
%of
In this case, $\Pol(\bT)$ does not have a minor-preserving map to $\proj$, and hence 
satisfies the condition 
for the algorithms of Bulatov and of Zhuk. The  polynomial-time tractability of 
 $\Csp(\bT)$ then implies the polynomial-time tractability of $\Csp(\bB)$, by a reduction from~\cite{BodMot-Unary}.
The publications cited above show that the canonical
polymorphisms determine the complexity of  
the CSPs in (1), (3), (5), (6), and (7), 
but not for the CSPs in (2) and (4). 
%\footnote{The dividing line corresponds empirically to the model-theoretic property NSOP (of not having the \emph{strict order property}; see~\cite{Simon}).} 

%the structure $\bB$ 
%has a polymorphism which induces a Siggers polymorphism of $\bT$. Polymorphisms 
%that induce polymorphisms of $\bT$ (in a sense that is made precise in Section~\ref{sect:canonicity}) 
%are called \emph{canonical}. 

%In~\cite{BodMot-Unary}, an important property has been used to prove that
%the canonical polymorphisms determine the complexity of the CSP. 
%We call this property the \emph{unique interpolation property}. We 
%close a gap in the proof from~\cite{BodMot-Unary} and 
We introduce a property, which is implicit in~\cite{BodMot-Unary,MMSNP}, and which we call 
the \emph{unique interpolation property (UIP)}.  
We present equivalent characterisations of the unique interpolation property; in particular, we show that it suffices to verify the unique interpolation property for binary polymorphisms. 
For our characterisation we need an additional 
assumption, which holds in all the mentioned examples that appear in (1)-(7) above, and which makes the connection
between CSPs and 
universal algebra particularly strong:
we require that $\bB$ is a reduct of a finitely homogeneous \emph{Ramsey structure} $\bA$. 
While the assumption to be Ramsey is quite strong, the assumption that a finitely homogeneous structure is a reduct of a finitely homogeneous Ramsey structure is weak: in fact, it might be that  \emph{every} finitely bounded homogeneous structure has a finitely bounded homogeneous Ramsey expansion (see~\cite{BodirskyRamsey,BPT-decidability-of-definability}). 

%Our second application of the general results from Sections~\ref{sect:uip} and~\ref{sect:binary} is a 

We use our general results to prove a 
new complexity dichotomy for qualitative spatial reasoning. A fundamental formalism
in spatial reasoning is RCC5~\cite{Bennett}, 
which involves five binary relations defined on some general set of \emph{regions}; a formal definition can be found in Section~\ref{sect:rcc5}. 
%,DBLP:journals/jair/RenzN01,Nebel95,RCC5JD}. 
Renz and Nebel~\cite{NebelRenz} showed that the CSP for RCC5 
is NP-complete, and Nebel~\cite{Nebel95} showed that the CSP for the five basic relations of RCC5 is in P (via a reduction to 2SAT\footnote{The proof of Nebel was formulated for RCC8 instead of RCC5, but the result for RCC5 can be shown analogously.}). 
Renz and Nebel~\cite{NebelRenz} extended the polynomial-time tractability result to 
a superclass of the basic relations, and they showed that their expansion is \emph{maximal}
in the sense that every larger subset of the RCC5 relations has an NP-hard CSP. 
Drakengren and Jonsson~\cite{RCC5JD} classified the computational complexity of the CSP for all subsets of the RCC5 relations. 
We classify the complexity of the CSP for expansions of the basic relations of RCC5 by first-order definable relations of arbitrary finite arity, 
solving the essential part of Problem 4 in~\cite{Qualitative-Survey}\footnote{For a full solution, it remains to consider the CSPs for relations of RCC5 where some of the basic relations are missing.}.
%\footnote{For a full solution, we also have to classify  
In particular, we verify the tractability conjecture for these CSPs, and show that the tractable cases can be solved 
by reducing them to polynomial-time tractable finite-domain CSPs using the reduction from~\cite{BodMot-Unary}. 
For our method, we need a homogeneous Ramsey expansion of RCC5; to prove the Ramsey property of the expansion that we construct we apply a recent Ramsey transfer result of Mottet and Pinsker~\cite{MottetPinskerCores}, using the well-known fact that a certain ordered version of the atomless Boolean algebra is Ramsey~\cite{Topo-Dynamics}. 

The results mentioned so far have been announced in the proceedings of LICS 2021~\cite{BodirskyBodorUIP}, but most proofs were omitted because of space restrictions. 
In this extended version of the article 
we additionally show that if the condition of the infinite-domain tractability conjecture applies to a first-order expansion $\bB$ of the basic relations of RCC5,
then $\Csp(\bB)$ can be solved by Datalog~\cite{FederVardi} (in polynomial time); otherwise, it cannot be expressed in the much stronger \emph{fixed-point logic with counting (FPC)}~\cite{AtseriasBulatovDawar}.

\subsection*{Related Work}
There exists another powerful approach to CSP complexity classification under similar conditions, the theory of \emph{smooth approximations}~\cite{MottetPinskerSmoothConf,Collapses}. 
The theory of smooth approximations also gives a sufficient condition for the UIP in certain cases (see the discussion at the end of Section 3.2 in~\cite{MottetPinskerSmoothConf}); but the machinery and the general statements there do not seem to directly relate to the statements that we prove here. %Each of the two approaches requires extra efforts to obtain classification results in concrete situations. 
The smooth approximations approach proved to be very efficient in re-deriving existing classifications and proving new classifications 
where a surprisingly large proportion of the work follows from the general results of smooth approximations, and only relatively few arguments are needed to adapt to the classes under consideration. However, we do not know how to apply their approach to derive our spatial reasoning classification. We believe that also the general results presented here will be useful in other settings; one example of a different setting where the present approach has been applied is the class of monadically stable $\omega$-categorical structures, for which the second author confirmed the infinite-domain tractability conjecture in his PhD thesis~\cite{BodorDiss}, using results of this article. 
In contrast to the present article, the smooth approximations approach also provides a systematic approach to Datalog expressibility~\cite{Collapses}; but also here we do not know how to obtain the particular results about Datalog expressibility in spatial reasoning from their general theorems. 

%; in comparison, the results that we obtain concerning Datalog are rather

\subsection*{Outline}
In Section~\ref{sect:omega-cat},
Section~\ref{sect:ramsey}, 
Section~\ref{sect:clones}, and Section~\ref{sect:canonicity} 
we introduce some fundamental notions and results from model theory, Ramsey theory, 
universal algebra, and canonical functions, respectively, 
%and in Section~\ref{sect:clones} some fundamental notions and results from universal algebra 
that we need to state our results. In Section~\ref{sect:ramsey} we also need a consequence of the Ramsey property
of an $\omega$-categorical structure $\bA$ 
for the existence of independent elementary substructures of $\bA$ (Theorem~\ref{thm:indep}); this result might be of independent interest in model theory. 

Our general main results about the UIP with applications to complexity classification are presented in Section~\ref{sect:results}. 
%We first formally introduce the unique interpolation property in Section~\ref{sect:uip} and 
In particular, 
we present two results that facilitate the verification of the UIP in concrete settings (Section~\ref{sect:binary-uip}). The proofs of these results are given in Section~\ref{sect:ext-proofs} and Section~\ref{sect:binary}, respectively. 
%All of these general results can be used to classify the complexity of a large class of CSPs

The applications of the general results to the spatial reasoning formalism RCC5 are in  Section~\ref{sect:rcc5}, and the classification of the respective CSPs in Datalog in Section~\ref{sect:Datalog}.

%The proofs of the main results, Theorem~\ref{thm:main} and  Theorem~\ref{thm:binary}, can be found in a technical appendix. 
%TODO HERE. 
%We also present applications to obtain complexity dichotomies for the class of structures with 
%exponential labelled growth (Section~\ref{sect:appl-unary}), for first-order expansions of the homogeneous universal poset (deferred to Section~\ref{sect:appl-poset} in the appendix), and for 
%first-order expansions of the basic relations of RCC-5 (Section~\ref{sect:appl-rcc5-result}). 
%(Section~\ref{sect:appl-rcc5}). 

\section{Countably Categorical Structures}
\label{sect:omega-cat}
%All the material recalled in this section is standard and can for instance be found in~\cite{HodgesLong}. 
All structures considered in this text are assumed to be relational unless stated otherwise. If $\bA$ is a structure,  then $A$ denotes the domain of $\bA$.
A structure $\bA$ is called \emph{$\omega$-categorical} 
if the set of all first-order sentences that hold in
$\bA$ has exactly one countable model up to isomorphism. This concept has an equivalent
characterisation based on the automorphism group
$\Aut(\bA)$ of $\bA$. For $l \in {\mathbb N}$, 
the \emph{$l$-orbit} of $a = (a_1,\dots,a_l) \in A^l$ in $\Aut(\bA)$ is the set $$\{\alpha(a) \colon \alpha \in \Aut(\bA) \} \text{ where } \alpha(a) \coloneqq (\alpha a_1,\dots, \alpha a_l).$$ 
We write ${\mathcal O}_l(\bA)$ 
for the set of $l$-orbits of $\Aut(\bA)$. 
 Engeler, Svenonius, and Ryll-Nardzewski (see, e.g., Theorem~6.3.1 in~\cite{Hodges}) proved that a countable structure is $\omega$-categorical if and only if 
 $\Aut(\bA)$ is \emph{oligomorphic}, i.e., if 
 ${\mathcal O}_l(\bA)$ is finite for every $l \in {\mathbb N}$. In fact, $a,b \in A^l$ have the same $l$-orbit if and only if 
 they satisfy the same first-order formulas over $\bA$; we write 
 $\typ^{\bA}(a)$ for the set of all first-order formulas satisfied by $a$ over $\bA$, called the \emph{type of $a$ over $\bA$}. For $B \subseteq A$, we write $\typ^{\bA}(a/B)$ for the set of all first-order formulas
 with parameters from $B$ (i.e., formulas over an expanded signature that also contains constant symbols denoting elements from $B$) 
 satisfied by $a$ over $\bA$. 

%\subsection{Homogeneous Structures}
%\label{sect:homog}
{\bf Homogeneous structures.}
%All definitions and facts in this section are standard and can be found e.g.~in~\cite{HodgesLong}. 
Many examples of $\omega$-categorical structures
arise from structures $\bA$ that are \emph{homogeneous}.
A relational structure is \emph{homogeneous} if every isomorphism between finite substructures of $\bA$ can be extended to an automorphism of $\bA$. 
It follows from the above that if $\bA$ is finitely homogeneous, then 
$\bA$ is $\omega$-categorical. 
If $a \in A^n$, we write $\atyp^{\bA}(a)$ for the set of all quantifier-free formulas that hold on $a$ over $\bA$. Clearly, in homogeneous structures $\bA$ the quantifier-free type of $a$ determines the type of $a$. 
Also note that every $\omega$-categorical structure can be turned into a homogeneous $\omega$-categorical structure by expanding it by all first-order definable relations.

Every homogeneous structure $\bA$ is uniquely given (up to isomorphism) by its \emph{age}, 
i.e., by the class $\Age(\bA)$ of all finite structures that embed into $\bA$ (see, e.g., Lemma 6.1.4 in~\cite{Hodges}). Conversely, every class
$\mathcal C$ 
of structures with finite relational signature which 
is an \emph{amalgamation class}, i.e., is closed under isomorphism, substructures, and
which has the \emph{amalgamation property}, is the age of a homogeneous structure, which we call the \emph{Fra{\"i}ss\'e-limit of ${\mathcal C}$}
(see, e.g., Theorem~6.1.2 in~\cite{Hodges}). 
%The Fra{\"i}ss\'e-limit of ${\mathcal C}$ embeds 
%all countable structures whose age equals ${\mathcal C}$. 
We say that $\bB$ is \emph{finitely bounded} if there exists a finite set of finite structures $\mathcal F$ such that the age of $\bB$ equals the class of all finite structures $\bA$ such that no structure from $\mathcal F$ embeds into $\bA$.

%\subsection{Model-complete cores}
%\label{sect:mc-core}
{\bf Model-complete cores.} 
An $\omega$-categorical
structure $\bC$ is \emph{model-complete}
if every self-embedding $e \colon \bC \hookrightarrow \bC$ 
preserves all first-order formulas. 
An $\omega$-categorical structure $\bC$
is called a \emph{core} if every endomorphism 
of $\bC$ is a self-embedding of $\bC$. 
If there is a homomorphism from a structure $\bA$ to a structure $\bB$ and vice versa, then $\bA$ and $\bB$ are called \emph{homomorphically equivalent}. 
Every $\omega$-categorical structure $\bA$ is
homomorphically equivalent to
a model-complete core, which is unique up to isomorphism, and again $\omega$-categorical, and which is called \emph{the} model-complete core of $\bA$~\cite{Cores-Journal,BodHilsMartin-Journal}.

%\begin{lemma}\label{lem:core-orbits}
%For every $n$, the model-complete core $\bC$ of
%an $\omega$-categorical structure $\bB$ 
%has at most as many orbits of $n$-tuples as $\bB$. 
%\end{lemma}

%\subsection{Powers}
%\label{sect:power}

%\subsection{Ramsey structures}
%\label{sect:ramsey}
%{\bf Ramsey structures.} 

\section{Ramsey Theory}
\label{sect:ramsey}
For $\tau$-structures $\bA,\bB$ we
write ${\bB \choose \bA}$ for the set of all embeddings of $\bA$ into $\bB$. 
%Let ${\mathcal C}$ be a class of finite $\tau$-structures. 
If $\bA,\bB,\bC$ are $\tau$-structures and $r \in {\mathbb N}$ then we write
$$\bC \to (\bB)^{\bA}_r$$ if for 
every function $\chi \colon {\bC \choose \bA} \to \{1,\dots,r\}$ (also referred to as a \emph{colouring}, where $\{1,\dots,r\}$ are the different colours) there exists 
$e \in {\bC \choose \bB}$ such that
$\chi$ is constant on $$\left \{e \circ f \colon f \in {\bB \choose \bA} \right \} \subseteq {\bC \choose \bA}.$$ 

\begin{definition}
A class $\mathcal C$ of finite $\tau$-structures 
has the \emph{Ramsey property}
if for all $\bA,\bB \in {\mathcal C}$ and $r \in {\mathbb N}$ 
there exists $\bC \in {\mathcal C}$ such that 
 $\bC \to (\bB)^{\bA}_r$. A homogeneous structure is \emph{Ramsey} if its age has the Ramsey property. 
 \end{definition}

Let $r \in {\mathbb N}$. 
We write 
$\bC \to (\bB)_r$
if for every $\chi \colon C^{|B|} \to \{1,\dots,r\}$ 
there exists 
$e \colon \bB \hookrightarrow \bC$ 
%such that if $\bA$ is a finite structure and 
%$f_1,f_2 \in {\bA \choose \bC}$ then
such that if
$\atyp^{\bB}(s) = \atyp^{\bB}(t)$ for $s,t \in B^{|B|}$, then $\chi(e(s)) = \chi(e(t))$. 
%if $\typ of a tuple $t \in f(C)^k$ only depends on $\typ^{\bA}(t)$. 
The Ramsey property 
%of an homogeneous 
%$\omega$-categorical 
%structure $\bA$ 
%If $C$ is a set and $k \in {\mathbb N}$
%then we write ${C \choose k}$ for the set of all functions from $\{1,\dots,k\}$ into $C$. 
has the following well-known consequence (see, e.g., Proposition 2.21 in~\cite{BodirskyRamsey}).  

\begin{lemma}\label{lem:can}
%Let $\bA$ be a homogeneous 
%$\omega$-categorical 
%Ramsey structure. 
Let $\mathcal C$ be a Ramsey class. 
Then for every $r \in {\mathbb N}$ and $\bB \in \mathcal C$ there exists $\bC \in \mathcal C$ such that 
$\bC \to (\bB)_r$. 
\end{lemma}

Another important consequence of the Ramsey property is presented in Section~\ref{sect:canonisation}. 
%The Ramsey property appears to be quite strong; however, note that 
%every homogeneous structure with a finite relational signature known to the authors has a homogeneous expansion with a finite relational signature which is additionally Ramsey~\cite{BodirskyRamsey}. 
%In this paper,
We also need the following consequence of the Ramsey property (in the proof of Proposition~\ref{lem:diagonal-common-interpolation}), which appears to be new.
% (its proof 
%can be found in Section~\ref{sect:indep}). 
Let $\bD$ be a structure and let $A,B \subseteq D$. We say that $A$ is \emph{independent from $B$ in $\bD$} if for all $\bar a,\bar b \in A^n$, if $\typ^{\bD}(\bar a) = \typ^{\bD}(\bar b)$, then
$\typ^{\bD}(\bar a/B) = \typ^{\bD}(\bar b/B)$. 
Two substructures $\bA,\bB$ of $\bD$ are called \emph{independent} if $A$ is independent from $B$ in $\bD$ and $B$  is independent from $A$ in $\bD$. 
A substructure $\bB$ of a $\tau$-structure $\bA$ is called
\emph{elementary} if the identity mapping from $\bB$ to $\bA$ 
preserves all first-order $\tau$-formulas. 

\begin{theorem}\label{thm:indep}
Let $\bA$ be a countable homogeneous $\omega$-categorical Ramsey structure. Then $\bA$ contains two independent elementary substructures. 
\end{theorem} 

%\section{Independence}
%\label{sect:indep}
%In this section we present a consequence
%of the Ramsey property for $\bA$ which concerns the existence of large independent sets in $\bA$  and is relevant for verifying the UIP. 
%If $A,B \subseteq D$, then we write
%$AB$ instead of $A \cup B$. 
%\begin{definition}
%Let $\bD$ be a structure and let $A,B \subseteq D$. We say that $A$ is \emph{independent from $B$ in $\bD$} if for all $\bar a,\bar b \in A$, if $\typ^{\bD}(\bar a) = \typ^{\bD}(\bar b)$, then
%$\typ^{\bD}(\bar a/B) = \typ^{\bD}(\bar b/B)$. If the reference to $\bD$ is clear we might drop it. 
%If $A$ is independent from $B$ and $B$ is independent from $A$, then we say that \emph{$A$ and $B$ are independent}. 
%If $C=\emptyset$ we might also drop the words \emph{`over $C$'}. 
%\end{definition}
%We are interested in the existence of independent elementary substructures of homogeneous $\omega$-categorical 
%structures. 

In the proof of Theorem~\ref{thm:indep} 
we use the following consequence
of the compactness theorem of first-order logic. 
% argument which we present for completeness. 

\begin{proposition}\label{prop:indep}
Let $\bA$ be a countable $\omega$-categorical  
homogeneous structure.
% with finite relational signature. 
%$\omega$-categorical structure. 
Then the following are equivalent.
\begin{enumerate}
\item $\bA$ contains two independent elementary substructures. 
\item For all $\bB_1,\bB_2 \in \Age(\bA)$ there exists embeddings 
$e_i \colon \bB_i \hookrightarrow \bA$, for $i \in \{1,2\}$, 
such that $e_1(B_1)$ and $e_2(B_2)$ are independent in $\bA$. 
\end{enumerate}
\end{proposition}
\begin{proof}
The forward implication is immediate from the definitions. 
For the converse implication, 
let $a_1,a_2,\dots$ be an enumeration of $A$ and let $B = \{b_1,b_2,\dots\}$ and $C = \{c_1,c_2,\dots\}$ be sets of new constant symbols.  
Let $\Phi$ be the set of all sentences 
that express that for all $n \in {\mathbb N}$, $\bar b,\bar b' \in B^n$, and $\bar c,\bar c' \in C^n$
\begin{itemize}
\item if $\typ^{\bA}(\bar b) = \typ^{\bA}(\bar b')$ 
then $\typ^{\bA}(\bar b,\bar c) 
= \typ^{\bA}(\bar b',\bar c)$. 
\item if $\typ^{\bA}(\bar c) = \typ^{\bA}(\bar c')$ 
then $\typ^{\bA}(\bar c,\bar b) 
= \typ^{\bA}(\bar c',\bar b)$. 
\end{itemize}
For $n \in {\mathbb N}$ let 
$\phi_n(x_1,\dots,x_n)$ be a formula that expresses that
$\typ^{\bA}(x_1,\dots,x_n) = \typ^{\bB}(a_1,\dots,a_n)$. 
By assumption, all finite subsets of 
$$T \coloneqq \Th(\bA) \cup \Phi \cup \{\phi_n(b_1,\dots,b_n) \wedge \phi_n(c_1,\dots,c_n) : n \in {\mathbb N} \}$$
are satisfiable. 
By compactness of first-order logic, it follows that $T$ has a model $\bA'$.  
By the downward L\"owenheim-Skolem theorem, we may assume that $\bA'$ is countably infinite. The reduct of $\bA'$ with the same signature as $\bA$ satisfies $\Th(\bA)$, and since $\bA$ is $\omega$-categorical this reduct 
is isomorphic to $\bA$, so that we may assume that $\bA'$ is an expansion of $\bA$. 
Let $\bB$ be the substructure of $\bA$
induced by the constants
from $B$. 
Since
$\bA' \models \{\phi_n(b_1,\dots,b_n) : n \in {\mathbb N} \}$ and by the homogeneity of $\bA$
we have that $\bB$ is an elementary substructure of $\bA$. Likewise, the 
substructure $\bC$ of $\bA$ 
induced by the constants
from $C$ is an elementary substructure of $\bA$. Since $\bA' \models \Phi$ the
two substructures $\bB$ and $\bC$ are independent. 
\end{proof}

To find independent sets, we  use the Ramsey property via the following lemma. 

\begin{lemma}\label{lem:ramsey-indep}
Let $\bA$ be a countable 
%homogeneous Ramsey structure with finite relational signature, 
homogeneous $\omega$-categorical Ramsey structure, 
$C \subseteq A$ finite,
and $\bB \in \Age(\bA)$. 
Then there exists $\bD \in \Age(\bA)$ such that 
for every $e \colon \bD \hookrightarrow \bA$
there exists $f \colon \bB \hookrightarrow \bA$ such that
%\begin{itemize}
%\item 
$f(B) \subseteq e(D)$ 
%\item 
and for all $\bar a, \bar b \in B^{|B|}$, if $\typ^{\bA}(f(\bar a)) = \typ^{\bA}(f(\bar b))$ then 
$\typ^{\bA}(f(\bar a)/C) = \typ^{\bA}(f(\bar b)/C)$. 
%\end{itemize}
\end{lemma}
\begin{proof}
Let $C = \{c_1,\dots,c_n\}$. 
%and let $m$ be the maximal arity of the relations of $\bA$.
Let $r$ be the number of types
of $|B|+n$-tuples in $\bA$. 
Since $\bA$ is homogeneous 
%$\omega$-categorical 
Ramsey 
there exists
$\bD \in \Age(\bA)$ such that 
$\bD \to (\bB)_r$. We claim that
$\bD$ satisfies the statement of the lemma. 
Let $e \colon \bD \hookrightarrow \bA$ be an embedding.
 We color $e(D)^{|B|}$ as follows: 
the colour of $t \in e(D)^{|B|}$ is the orbit
 of $(t_1,\dots,t_{|B|},c_1,\dots,c_n)$. 
 There are at most $r$ such orbits. 
 Since $\bD \to (\bB)_r$ we
 get that there exists an $f \colon \bB \hookrightarrow \bA$ such that 
 $f(B) \subseteq e(D)$ and 
 if $\bar a,\bar b \in B^{|B|}$ are such that $\typ^{\bA}(f(\bar a)) = \typ^{\bA}(f(\bar b))$, then $\typ^{\bA}(f(\bar a)/C) = \typ^{\bA}(f(\bar b)/C)$. 
\end{proof}

%We are not aware of a reference for the following theorem and believe that it of independent interest in the theory of $\omega$-categorical structures.  

\begin{proof}[Proof of Theorem~\ref{thm:indep}]
We use Proposition~\ref{prop:indep}.
Let $\bB_1,\bB_2 \in \Age(\bA)$;
we may assume that $\bB_1$ and $\bB_2$ are substructures of $\bA$. 
Then Lemma~\ref{lem:ramsey-indep}
implies that there exists $\bD_1 \in \Age(\bA)$ such that for every embedding $e_1 \colon \bD_1 \hookrightarrow \bA$
there exists an embedding $f_1 \colon \bB_1 \hookrightarrow \bA$ 
such that $f_1(B_1) \subseteq e_1(D_1)$
and for all $n \in {\mathbb N}$, $b,b' \in (B_1)^n$
\begin{equation}
\begin{aligned}
 \text{ if }
\typ^{\bA}(f_1(b)) = \typ^{\bA}(f_1(b')) \\
\text{ then } 
%\Rightarrow 
\typ^{\bA}(f_1(b)/B_2) = \typ^{\bA}(f_1(b')/B_2).
\end{aligned}
\label{eq:indep}
\end{equation}
We may assume that $\bD_1$ is a
substructure of $\bA$. 
Another application of Lemma~\ref{lem:ramsey-indep}
gives us 
%a structure $\bD_2 \in \Age(\bA)$ such that for every embedding
%$e_2 \colon \bD_2 \to \bA$ there exists
%$f_2 \colon \bB_2 \to \bA$ such that 
%$f_2(B_2) \subseteq e_2(D_2)$
%and  
an embedding $f_2 \colon \bB_2 \hookrightarrow \bA$ such that 
for all $n \in {\mathbb N}$, $b,b' \in (B_2)^n$, 
if
$\typ^{\bA}(f_2(b)) = \typ^{\bA}(f_2(b'))$
then 
$\typ^{\bA}(f_2(b)/D_1) = \typ^{\bA}(f_2(b')/D_1)$. 
By the homogeneity of $\bA$, there exists 
$\alpha \in \Aut(\bA)$ extending $f_2$. 
The property of $\bD_1$ implies that there
exists an embedding $f_1 \colon \bB_1 \hookrightarrow \bA$ such that
$f_1(B_1) \subseteq \alpha^{-1}(D_1)$ such that
(\ref{eq:indep}) holds 
for all $b,b' \in (B_1)^n$. 

We claim that $f_1(B_1)$ and $B_2$ are independent. Clearly, $f_1(B_1)$ is independent from $B_2$. To show that $B_2$ is independent from $f_1(B_1)$, 
let $n \in {\mathbb N}$ and $b,b' \in (B_2)^n$ 
be such that
$\typ^{\bA}(b) = \typ^{\bA}(b')$. 
Then 
$\typ^{\bA}(\alpha(b)) = \typ^{\bA}(\alpha(b'))$,
and hence 
that $\typ^{\bA}(\alpha(b)/D_1) = \typ^{\bA}(\alpha(b')/D_1)$ by the choice of $f_2$. This in turn implies
that $\typ^{\bA}(b/\alpha^{-1}(D_1)) = \typ^{\bA}(b'/\alpha^{-1}(D_1))$ which
implies the statement since $f_1(B_1) \subseteq \alpha^{-1}(D_1)$. 
\end{proof}

\section{Oligomorphic Clones}
\label{sect:clones}
In this section we introduce fundamental concepts and results about polymorphism clones
of finite and countably infinite $\omega$-categorical structures that are needed for the statement of our results. 

The \emph{$k$-th direct power} of a structure $\bA$ 
(also called the \emph{categorical power}), denoted by
$\bA^k$. It is the $\tau$-structure with domain $A^k$ such that 
$$((a_{1,1},\dots,a_{1,k}),\dots,(a_{n,1},\dots,a_{n,k})) \in R^{\bA^k}$$
if and only if $(a_{1,i},\dots,a_{n,i}) \in R^{\bA}$ for every $i \in \{1,\dots,k\}$. 
%The results in Sections~\ref{sect:siggers},~\ref{sect:pseudo-sig}, \ref{thm:inf-dichotomy}, and~\ref{sect:canon} are standard. The %concepts and 
%results from  %Section~\ref{sect:rich} (Rich Subsets)
%and 
%Section~\ref{sect:invariance} (Invariance)
%are mostly easy or follow from the literature;
%but have not been presented in this form in the literature before. 
A \emph{polymorphism} of $\bA$
is a homomorphism from a finite direct power $\bA^k$ to $\bA$. The set of all polymorphisms of $\bA$, denoted by $\Pol(\bA)$,  forms an \emph{clone (over $A$)}, i.e., it is 
%a
%subset of $\bigcup_{n=1}^\infty (A^n \to A)$
%which is 
closed
under composition and contains for every $k$ and $i \leq k$ the projection $\pi^k_i$ of arity $k$ that always returns the $i$-th argument. 
%all the $k$-ary projections, which will always be denoted by $\pi^k_1,\dots,\pi^k_k$. 
Moreover, $\Pol(\bA)$ is closed with respect to the \emph{topology of pointwise convergence}, i.e., 
if $f \colon A^k \to A$ is such that for every finite $F \subseteq A$ there exists $g \in \Pol(\bA)$ such that $g(a_1,\dots,a_k) = f(a_1,\dots,a_k)$ for all $a_1,\dots,a_k \in F$, then $f \in \Pol(\bA)$. 
If ${\mathscr S}$ is a set of operations over $A$,
then $\overline {\mathscr S}$ denotes the smallest closed set that contains ${\mathscr S}$. 
If ${\mathscr S}$ is a set of operations, we write ${\mathscr S}^{(k)}$ for the set of operation in ${\mathscr S}$ of arity $k$.
The clone of all projections on $\{0,1\}$
is denoted by $\proj$. 
If $f \colon A^n \to A$ is an operation and $\alpha \colon \{1,\dots,n\} \to \{1,\dots,k\}$ is a function, then $f_{\alpha}$ denotes the operation $(x_1,\dots,x_{k}) \mapsto f(x_{\alpha(1)},\dots,x_{\alpha(n)})$, and is called a \emph{minor} of $f$.

\begin{definition}
A function $\zeta \colon {\mathscr C}_1 \to {\mathscr C}_2$ between two clones over a countably infinite set $A$ is 
called 
\begin{itemize}
\item \emph{minor-preserving} 
if for any $f \in {\mathscr C}^{(n)}_1$ and $\alpha \colon \{1,\dots,n\} \to \{1,\dots,k\}$ 
$$\zeta(f_{\alpha}) = \zeta(f)_{\alpha};$$
%$$\zeta \big (f(\pi^k_{i_1},\dots,\pi^k_{i_n}) \big) = \zeta(f)(\pi^k_{i_1},\dots,\pi^k_{i_n})$$
%$$ 
%for all $f \in {\mathscr C}^{(n)}_1$ and $i_1,\dots,i_n \in \{1,\dots,k\}$. 
\item \emph{uniformly continuous} if and only if for any 
finite $G \subseteq A$ 
there is a finite $F \subseteq A$
such that for all $n \in {\mathbb N}$ and 
$f,g \in {\mathscr C}_1^{(n)}$, if $f|_{F^n} = g|_{F^n}$ then $\zeta(f)|_{G^n} = \zeta(g)|_{G^n}$. 
\end{itemize}
\end{definition}
The set 
${\mathscr O}_A \coloneqq \bigcup_{k \in {\mathbb N}} (A^k \to A)$ of all operations on $A$ 
can be equipped with a complete 
metric such that 
the closed clones with respect to this metric are precisely the polymorphism clones of relational
structures with domain $A$, and such that 
the uniformly continuous maps are precisely the maps that are uniformly continuous with respect to this metric in the standard sense, justifying our terminology; see, e.g.,~\cite{Book}.
Now suppose that $\bB$ is a finite structure with a finite relational signature; the following results imply that the existence of a minor-preserving map from $\Pol(\bB)$ to $\proj$ describes the border between containment in P and NP-completeness. 
 
 \begin{theorem}[follows from~\cite{JBK,wonderland}] 
 \label{thm:finite-hard}
 Let $\bB$ be a structure with a finite domain and a finite relational signature. 
If there is a minor-preserving map 
$\Pol(\bB) \to \proj$ 
 then $\Csp(\bB)$ is NP-hard.
 \end{theorem}
 
\begin{theorem}[\cite{MarotiMcKenzie}]
\label{thm:wnu}
Let $\bB$ be a structure with a finite domain. 
If there is no minor-preserving map 
$\Pol(\bB) \to \proj$ 
then $\Pol(\bB)$ contains a \emph{weak near unanimity (WNU) operation}, i.e., an operation $f$ of some arity $k \geq 2$ that satisfies for all $a,b \in B$ that
$$f(a,\dots,a,b) = f(a,\dots,b,a) = \dots = f(b,a,\dots,a).$$
\end{theorem} 
 
\begin{theorem}[Bulatov~\cite{BulatovFVConjecture}, Zhuk~\cite{ZhukFVConjecture}]
\label{thm:fin-tract}
Let $\bB$ be a structure with a finite domain  and a finite relational signature. 
If $\bB$ has a WNU polymorphism, 
% there is no minor-preserving map 
%$\Pol(\bB) \to \proj$, 
then $\Csp(\bB)$ is in P. 
\end{theorem}

The hardness result in Theorem~\ref{thm:finite-hard} can be generalised to a large class of countable structures $\bB$. 

\begin{theorem}[\cite{wonderland}]
\label{thm:hard}
Let $\bB$ be a countable $\omega$-categorical structure with finite relational signature such that $\Pol(\bB)$ has a uniformly continuous minor-preserving map to $\proj$. Then $\Csp(\bB)$ is NP-hard.  
\end{theorem}

The present paper is motivated by the following
conjecture from~\cite{BPP-projective-homomorphisms}, in a reformulation from~\cite{BKOPP}, 
which generalises the Bulatov-Zhuk dichotomy (Theorem~\ref{thm:fin-tract}). 

\begin{conjecture}
\label{conj:inf-tract}
Let $\bB$ be a reduct of a countable finitely bounded homogeneous structure such that there is no 
uniformly continuous minor-preserving map from $\Pol(\bB)$ to $\proj$. Then 
$\Csp(\bB)$ is in P. 
\end{conjecture}

\section{Canonical Functions}
\label{sect:canonicity}
%In full generality, this conjecture is open, but

Some of the strongest partial results towards Conjecture~\ref{conj:inf-tract} are based on the concept of \emph{canonical} polymorphisms. 
Let $\bA$ be an $\omega$-categorical structure. 
An operation 
$f \colon A^k \to A$ is called \emph{canonical over ${\bA}$} if for all $n \in {\mathbb N}$ and $a_1,\dots,a_k,b_1,\dots,b_k \in A^n$, if 
 $\typ^{\bA}(a_i) = \typ^{\bA}(b_i)$ for all $i \in \{1,\dots,k\}$, then
 $\typ^{\bA}(f(a_1,\dots,a_k)) = \typ^{\bA}(f(b_1,\dots,b_k))$ (the function is applied to the tuples componentwise). Note that if $\bA$ is a 
 finitely homogeneous structure, and $m$ is the maximal arity of (the relations of) $\bA$, then $f$ is canonical with respect to $\bA$ if and only if it is satisfies the condition above for $n = m$. 
 The set of all canonical operations over $\bA$ is denoted by $\Can(\bA)$; note that $\Can(\bA)$ is a clone and closed.

\subsection{Behaviours}
To conveniently work with canonical functions
and with functions that locally resemble canonical functions, we need the concept of a \emph{behaviour} of functions $f \colon A^k \to A$ over a structure $\bA$, following~\cite{BP-reductsRamsey,BP-canonical}. 
%Let $\bB,\bC$ be structures. 
%In this section we discuss the concept of a \emph{behaviour} of a function $f \colon B \to C$ with respect to $\bB$ and $\bC$, 
In order to have some flexibility when using
the terminology, we work in the setting of  
\emph{partial functions} from $A^k \to A$, i.e.,  functions that are only defined on a subset of $A^k$. 
%We mostly use the concept of a behaviour
%in the setting where $\bB = \bA^{(k)}$ and $\bC = \bA$,  for $k \in {\mathbb N}$ and some fixed $\omega$-categorical structure $\bA$. 

\begin{definition}[Behaviours]
Let $n,k \in {\mathbb N}$. Then an \emph{$n$-behaviour over $\bA$ (of a $k$-ary partial map from $A^k$ to $A$)}  is a partial function 
%$\mathcal B$ 
from ${\mathcal O_n(\bA)}^k$ to ${\mathcal O}_n(\bA)$. 
A \emph{behaviour over $\bA$} 
is a partial function $\mathcal B$ from $\bigcup_{i=1}^{\infty} {\mathcal O_i(\bA)}^k$ to $\bigcup_{i=1}^{\infty} {\mathcal O}_i(\bA)$ so that for every $i$ we have $\mathcal B({\mathcal O}_i(\bA)^k) \subseteq {\mathcal O}_i(\bA)$. 
An $n$-behaviour (behaviour) is called \emph{complete} if it is defined on all of ${\mathcal O}_n(\bA)^k$ (on all of $\bigcup_{i=1}^{\infty} {\mathcal O_i(\bA)^k}$). 
\end{definition}

Note that every $n$-behaviour is in particular a behaviour. 
%a sequence of $n$-behaviors for every $n \in {\mathbb N}$. 

\begin{definition}[Behaviours of functions]
Let $\bA$ be a structure and $k \in {\mathbb N}$. 
A partial function $f \colon A^k \to A$ \emph{realises} a behaviour $\mathcal B$ over $\bA$ 
%(or $f$ \emph{has the behaviour $\mathcal B$}) 
if for all $a^1_1,\dots,a^k_n \in A$,
if $O_i$ is the orbit of $(a^i_1,\dots,a^i_n)$ for $i \in \{1,\dots,k\}$ in $\Aut(\bA)$ and $f$ is defined on $(a^1_1,\dots,a^k_1), \dots, (a^1_n,\dots,a^k_n)$ and $B$ is defined on $(O_1,\dots,O_k)$, then 
${\mathcal B}(O_1,\dots,O_k)$ 
is the orbit of $(f(a^1_1,\dots,a^k_1),\dots,f(a^1_n,\dots,a^k_n))$. 
\end{definition}

If $S \subseteq A^k$ and $\mathcal B$ is a behaviour over $\bA$, then 
we say that $f \colon A^k \to A$ 
\emph{realises $\mathcal B$ on $S$} if the restriction $f|_S$ realises $\mathcal B$. 
The following proposition can be shown by a
standard compactness argument. 

\begin{proposition}\label{prop:compact} 
Let $\bA$ be an $\omega$-categorical structure
and $k,n \in {\mathbb N}$. 
Then an $n$-behaviour over $\bA$ is realised by some $f \in \Pol(\bA)^{(k)}$ if and only if it is realised by some $f \in \Pol(\bA)^{(k)}$ on $F^k$ for every finite subset $F \subseteq A$. 
\end{proposition}

%We leave the proof of this fact to the reader; 
%a similar compactness argument is presented in full detail for Proposition~\ref{prop:indep} below. 

\subsection{Canonicity}
\label{sect:canon}
%\begin{definition}[Canonicity]
Let $n,k \in {\mathbb N}$. 
A partial function $f \colon A^k \to A$ is called
\emph{$n$-canonical over $\bA$} 
if it realises a complete $n$-behaviour over
$\bA$, 
and \emph{canonical over $\bA$} 
if it realises some complete behaviour 
over $\bA$; note that this is consistent with the definition given in the beginning of Section~\ref{sect:canonicity}. 
If $S \subseteq A$, then 
we say that $f$ is 
\emph{($n$-) canonical on $S$} if the restriction $f|_S$ is ($n$-) canonical. 
%\end{definition}

Suppose that $\bB$ is a homogeneous relational structure with maximal arity $m$. 
If $f \colon B \to C$ realises some complete 
$m'$-behaviour over $(\bB,\bC)$ for $m' = \max(m,2)$, 
then $f$ is canonical over $(\bB,\bC)$
(in other words, in this situation a complete $m$-behaviour uniquely determines a complete behaviour; the requirement $m' \geq 2$ is necessary so that the $m'$-behaviour allows, for instance, to distinguish constant functions from injective functions). 
%also the behaviour with respect to equality 
Hence, if $\bB$ is finitely homogeneous, then there are only finitely many complete behaviours for functions of arity $k$. 

%We can apply the concept of canonical functions to polymorphisms
%of a structure $\bA$ by choosing $\bB \coloneqq \bA^{(k)}$ for $k \in {\mathbb N}$
%and $\bC \coloneqq \bA$. The set of all operations  
%$f \colon A^k \to A$  that are canonical over $(\bA^{(k)},\bA)$ is denoted by $\Can(\bA)$. 
%Note that $\Can(\bA)$ is a clone and closed in ${\mathscr O}_A$. 
%If $f \in \Can(\bA)$ then we also say that
%$f$ is \emph{canonical over $\bA$}
%and correspondingly we also use the terminology of 
%\emph{$n$-behaviour over $\bA$} and \emph{behaviour over $\bA$}. 
Note that %for every $n \in {\mathbb N}$,
the binary equivalence relation $\{(u,\alpha(u)) \colon u \in A^n, \alpha \in \Aut(\bA) \}$ 
is preserved by $\Can(\bA)$.
We consider the clone 
whose domain ${\mathcal O}_n(\bA)$ is the (finite) set of equivalence classes of this equivalence relation, and which contains 
for every $f \in \Can(\bA)$ the operation induced by $f$ on these classes. The map that sends $f$ to this operation 
will be denoted by $\xi^{\bA}_n$. 
Note that $\xi^{\bA}_n$ is a uniformly continuous clone homomorphism  
from $\Can(\bA)$ to $\xi^{\bA}_n(\Can(\bA))$. 
If ${\mathscr C} \subseteq \Can(\bA)$,
then we write $\xi^{\bA}_n({\mathscr C})$ for the
subclone of $\xi^{\bA}_n(\Can(\bA))$ induced by
the image of $\mathscr C$ under $\xi^{\bA}_n$. Note that if $f \in \Can(\bA)$, then $g \colon A^k \to A$ realises the same complete behaviour as $f$ over $\bA$ if and only if $f$ interpolates $g$ and $g$ interpolates $f$ over $\bA$. 

%Another compactness argument can be used to show the following. 

\begin{lemma}[Proposition 6.6 in~\cite{BPP-projective-homomorphisms}]\label{lem:lift}
Let $\bA$ be a finitely homogeneous structure.
% and let $\mathscr C$ be a clone that contains $\Aut(\bA)$. 
Let $m$ be the maximal arity of $\bA$ and 
suppose that $f,g \in \Can(\bA)
$ are such that 
$\xi^{\bA}_m(f) = \xi^{\bA}_m(g)$.
Then there are $e_1,e_2 \in \overline{\Aut(\bA)}$
such that $e_1 \circ f = e_2 \circ g$. 
\end{lemma}

The following proposition 
%did not appear
%in the literature before and 
can be shown 
by a compactness argument.
% which we leave to the reader. 

\begin{lemma}\label{lem:simul-interpol}
Let $\bB$ be a first-order reduct of an $\omega$-categorical structure $\bA$. 
Let $k,n,m \in {\mathbb N}$  and 
let ${\mathcal B}_1,\dots,{\mathcal B}_m$
be $n$-behaviours over $\bA$. 
%$(\bA^{(k)},\bA)$. 
Suppose that for every finite $F \subseteq A$ and $i \in \{1,\dots,m\}$ there exist $\alpha_{i,1},\dots,\alpha_{i,k}$ and $f \in \Pol(\bB)$ 
such that $f(\alpha_{i,1},\dots,\alpha_{i,k})$ realises ${\mathcal B}_i$ on $F^k$. 
Then there exist $e_{1,1},\dots,e_{m,k} \in \overline{\Aut(\bA)}$ and $f \in \Pol(\bB)$ 
such that for every $i \in \{1,\dots,m\}$ the operation  $f(e_{i,1},\dots,e_{i,k})$ realises ${\mathcal B}_i$ on all of $A^k$. 
Moreover, if $I \subseteq \{1,\dots,k\}$ and $\alpha_{i,j} = \id$ for all $j \in I$ and every finite $X \subseteq A$ then we can guarantee $e_{i,j} = \id$ for all $j \in I$. 
\end{lemma}

\subsection{Canonisation}
\label{sect:canonisation}
%For a proof of the following lemma, 
The following is from~\cite{BPT-decidability-of-definability}; 
also see~\cite{BP-canonical}. 
%; it can also be shown using Lemma~\ref{lem:can} and a compactness argument. 

\begin{lemma}[Canonisation lemma]\label{lem:canon}
Let $\bA$ be a countable $\omega$-categorical homogeneous Ramsey structure and 
let $f \colon A^k \to A$. Then there is a
canonical function over $\bA$ which is interpolated by $f$ 
over $\Aut(\bA)$. 
\end{lemma}

In some situations it is practical to use a finite version 
of the canonisation lemma (e.g.~in the proof of Lemma~\ref{lem:int-inv} or Lemma~\ref{lem:checkers}), which can be derived easily from
Lemma~\ref{lem:canon} by a compactness argument, or directly from Lemma~\ref{lem:can}. 
%as in a proof of Lemma~\ref{lem:canon}. 

\begin{lemma}\label{lem:canon2}
Let $\bA$ be a homogeneous $\omega$-categorical Ramsey structure and $B \subseteq A$ finite. Then there exists a finite $C \subseteq A$ such that for every $f \colon A^k \to A$
there are $\alpha_1,\dots,\alpha_k \in \Aut(\bA)$
such that $\alpha_1(B) \subseteq C,\dots,\alpha_k(B) \subseteq C$ and 
$f(\alpha_1,\dots,\alpha_k)$ is canonical on $B^k$. 
\end{lemma}
 
\subsection{Complexity}
% The results in this section are from~\cite{BodMot-Unary} and they are used in all the polynomial-time tractability results in this article. 
%\begin{proposition}[Lemma 4.9 in~\cite{BodMot-Unary}]
%\label{prop:fin-type-struct}
%Let $\bA$ be a reduct of a finitely bounded homogeneous structure $\bB$ of maximal arity $m \geq 2$. 
%There exists a finite-signature structure $\bC$ so that $\Pol(\bC)=\xi_m^\bB(\Pol(\bA) \cap \Can(\bB))$. 
 %\end{proposition} 
% TODO: add explanation. 
 
  The following theorem is essentially from~\cite{BodMot-Unary} 
 and uses~\cite{BulatovFVConjecture,ZhukFVConjecture}; 
 %in a formulation from~\cite{wonderland}; 
 since we are not aware of a formulation of the result in the form below, which is the form that is most useful for the purposes of this article, we added a guide to the literature in the proof. 
 
\begin{theorem}[\cite{BodMot-Unary}]
\label{thm:can-tract}
Let $\bB$ be a reduct of a countable finitely bounded homogeneous structure $\bA$ 
such that there is no 
uniformly continuous minor-preserving map from $\Pol(\bB) \cap \Can(\bA)$ to $\proj$. Then 
$\Csp(\bB)$ is in P. 
\end{theorem}
\begin{proof}
Let $m$ be the maximal arity 
of the relations of $\bB$. 
If $\xi_m^{\bA}$ has a (uniformly continuous) minor-preserving map to $\proj$ then so has $\Pol(\bB)$ by composing the maps, contrary to our assumptions.
Therefore, $\xi_m^{\bA}$ contains a \emph{Siggers operation}~\cite{Siggers}. 
Lemma~\ref{lem:lift} implies that $\Pol(\bB) \cap \Can(\bA)$ has a \emph{Siggers operation modulo $\overline{\Aut(\bA)}$}~\cite{Topo}. 
The polynomial-time tractability of 
$\Csp(\bB)$ then follows from Corollary 4.13 in~\cite{BodMot-Unary}, which uses~\cite{BulatovFVConjecture,ZhukFVConjecture}. 
%Let $\bD$ be a finite-signature 
%structure such that 
%from Proposition~\ref{prop:fin-type-struct} such that 
%$\Pol(\bD)=\xi_2^\bR(\Pol(\bC) \cap \Can(\bR))$; such a structure exists 
%by Lemma 4.9 in~\cite{BodMot-Unary}. 
% Cook up the finite structure D := T_A,m(B). 
% if Pol(D) would have a h1 to proj, then so would Can-Pol of the infinite. 
% Know: D has no h1 to proj -> CSP(D) is in P
% by Bulatov/Zhuk.  
% Use Thm 10.5.10 from the book
% for reducton from CSP to 
% CSP(T_Am(B)) 
\end{proof} 

For many classes of reducts of finitely bounded homogeneous structures, the conditions given in Theorem~\ref{thm:hard}
and in Theorem~\ref{thm:can-tract} cover all the cases. 
The next section presents a condition from~\cite{BodMot-Unary} that turned out to be highly relevant to understand this phenomenon. 

\section{Results}
\label{sect:results}
In this section we state the main results of the paper that we then prove in Section~\ref{sect:ext-proofs} and in Section~\ref{sect:binary}. The applications to RCC5 in spatial reasoning can be found in Section~\ref{sect:rcc5}. 

\subsection{The Unique Interpolation Property}
\label{sect:uip}
Let $\mathscr M$ be a set of function from $A$ to $A$ and let $f,g \colon A^k \to A$ be operations. We say that \emph{$f$ interpolates $g$ over ${\mathscr M}$}  if for
every finite $F \subseteq A$ there are
$u,v_1,\dots,v_k \in {\mathscr M}$ such that
$g(a_1,\dots,a_k) = u(f(v_1(a_1),\dots,v_k(a_k)))$ for all $a_1,\dots,a_k \in F$, i.e., 
$$ g \in \overline{ \{ u(f(v_1,\dots,v_k)) \mid u,v_1,\dots,v_k \in {\mathscr M} \} } .$$

%As mentioned earlier, the unique interpolation property already appears implicitly in~\cite{BodMot-Unary}. 
%Our first result can be seen as a characterisation of when the existence of a pseudo-Siggers polymorphism implies the existence of a canonical pseudo-Siggers polymorphism. 
%\end{definition}

%Let $\bB$ be a reduct of a homogeneous
\begin{definition}\label{def:uip}
Let $\bA$ be a structure and 
let $\mathscr C$ be a clone that contains
$\Aut(\bA)$. 
A function $\zeta$ defined
on $\mathscr D \subseteq \mathscr C$ has 
the \emph{unique interpolation property (UIP) with respect to ${\mathscr C}$ over $\bA$} if $\zeta(g) = \zeta(h)$ for all $g,h \in \mathscr D$ that are interpolated by the same operation $f \in {\mathscr C}$ over $\Aut(\bA)$. 
\end{definition}

The main result of this section, 
Lemma~\ref{lem:uip}, gives a sufficient
condition for the existence of an extension
of a minor-preserving map defined on 
the canonical polymorphisms of a structure 
$\bB$ to a uniformly continuous minor-preserving map defined on all of $\Pol(\bB)$. 
%that if a minor-preserving map from the canonical polymorphisms of $\bB$ 
The statement is inspired by 
Theorem~5.5 in~\cite{BodMot-Unary}, and
similar statements played an important role in~\cite{MMSNP}. 
However, the original proof in~\cite{BodMot-Unary}
contains a mistake (in the proof the mashup
theorem, Theorem 5.5 in~\cite{BodMot-Unary}, more specifically in the claim that $\phi$ is a minor-preserving map)
and fixing this with a general approach is one of the contributions of the present article.

\begin{lemma}[The extension lemma]
\label{lem:uip}
	Let $\bB$ be a reduct of a 
	countable finitely homogeneous Ramsey structure $\bA$  
	and let $\mathscr E$ be a clone. 
	Suppose that $\zeta \colon \Pol(\bB) \cap \Can(\bA) \to \mathscr E$ 
	is minor-preserving and 
	has the UIP with respect to $\Pol(\bB)$ over $\bA$. 
	Then $\zeta$ can be extended to a 
	uniformly continuous minor-preserving map $\tilde{\zeta} \colon \Pol(\bB) \to \mathscr{E}$. 
\end{lemma}

%Recall that if $\bB$ is finitely bounded and $\bC$ has a finite signature, then the non-existence
%of a minor-preserving $\xi \colon \Pol(\bB) \cap \Can(\bA) \to {\mathscr D}$ to $\proj$ 
%implies that $\Csp(\bC)$ is in P by the results from~\cite{Bodirsky-Mottet,BodMot-Unary}. 
%On the other hand, we have already mentioned that a uniformly continuous
%minor-preserving map from $\Pol(\bB) \to {\mathscr D}$ implies the NP-hardness of $\Csp(\bB)$.  Hence, 
The extension lemma motivates the study of the UIP in the context of complexity classification of CSPs, and we obtain the following in Section~\ref{sect:main-proofs}. 

\begin{corollary}\label{cor:dicho}
%Let $\mathcal C$ be a class of countable $\omega$-categorical structures closed under taking model-complete cores such that each structure $\bC \in {\mathcal C}$ is the reduct of a finitely bounded homogeneous structure $\bB$, which is itself a reduct of a finitely homogeneous Ramsey structure $\bA$. 
Let $\bC$ be a
reduct of a finitely bounded homogeneous structure $\bB$, which is itself a reduct of a finitely homogeneous Ramsey structure $\bA$. 
Suppose that whenever there is a uniformly continuous minor-preserving map
from $\Pol(\bC) \cap \Can(\bB)$ to $\proj$ then there is also a minor-preserving map 
from $\Pol(\bC) \cap \Can(\bA)$ to $\proj$ which has the UIP with respect to $\Pol(\bC)$ over $\bA$. 
Then 
%for every $\bC \in {\mathcal C}$, we have that 
$\Csp(\bC)$ is in P or NP-complete. 
\end{corollary}

\subsection{Other Characterisations of the UIP}
\label{sect:binary-uip}
To verify the UIP, the following 
characterisation of the UIP is helpful, since it 
allows us to focus on binary polymorphisms. 
Let ${\mathscr C}$ be a clone 
and let ${\mathscr M} \subseteq {\mathscr C}^{(1)}$. 
If $v_1,\dots,v_k \in {\mathscr M}$ we write
$f(v_1,\dots,v_k)$ for the operation that maps
$(a_1,\dots,a_k)$ to $f(v_1(a_1),\dots,v_k(a_k))$. 
We say that a function $\zeta$ defined on $\mathscr C$ is \emph{$\mathscr M$-invariant} if for all
$f \in {\mathscr C}^{(k)}$ and $u,v_1,\dots,v_k \in \mathscr M$
we have 
\begin{align}
\zeta(f) = \zeta \big(u(f(v_1,\dots,v_k)) \big).
\label{eq:inv}
\end{align}

It follows from known results
(see Proposition~\ref{prop:can-proj})
that if $\bA$ is a finitely homogeneous structure and $\mathscr C \subseteq \Can(\bA)$ is a closed clone, then there exists 
a uniformly continuous minor-preserving map 
from $\mathscr C$ to $\proj$ 
if and only if there exists a 
minor-preserving map 
$\zeta \colon {\mathscr C} \to \proj$ which is $\overline{\Aut(\bA)}$-invariant. 
%that there exists 
%a uniformly continuous minor-preserving map 
%from $\Pol(\bC)$ to $\proj$ 
%if and only if there exists a 
%minor-preserving map 
%$\zeta \colon \Pol(\bC) \to \proj$ which is $\overline{\Aut(\bB)}$-invariant.

%, i.e., $\zeta$ has the property that  
%for all
%$f \in {\mathscr C}^{(k)}$ and 
%$u,v_1,\dots,v_k \in \overline{\Aut(\bB)}$
%we have 
%$$ \zeta(f) = \zeta(u(f(v_1,\dots,v_k))).$$

%For an elegant statement of this characterisation, we first need to introduce a new concept: 

\begin{theorem}[Binary UIP verification]
\label{thm:binary}
Let $\bA$ be a countable finitely homogeneous Ramsey structure
and let $\bC$ be a reduct of $\bA$. Let
$\mathscr C$ be the clone of all polymorphisms of $\bC$ that are canonical over $\bA$
and let 
$\zeta \colon {\mathscr C} \to \proj$
be an $\overline{\Aut(\bA)}$-invariant minor-preserving map. 
%Let $\bB$ be a reduct of a homogeneous Ramsey structure $\bA$ with finite relational signature, and let $\zeta$ be an $\overline{\Aut(\bA)}$-invariant minor-preserving map from ${\mathscr C} \coloneqq \Pol(\bB) \cap \Can(\bA)$ to $\proj$. 
Then the following are equivalent. 
\begin{enumerate}
\item $\zeta$ has the UIP with respect to $\Pol(\bC)$ over $\bA$. 
\item 
%there is a uniformly continuous minor-preserving map 
%$\zeta \colon {\mathscr C} \to \proj$
%such that 
for all $f \in \Pol(\bC)^{(2)}$ and 
$u_1,u_2 \in {\overline{\Aut(\bA)}}$, if  
$f(\id,u_1)$ and $f(\id,u_2)$ are canonical over $\bA$ then 
$$\zeta(f(\id,u_1)) = \zeta(f(\id,u_2)).$$
\end{enumerate}
\end{theorem}
%\begin{proof}
%Note
%that if $\zeta$ has the UIP, it is in particular
%$\overline{\Aut(\bA)}$-invariant, and hence
%uniformly continuous by Proposition~\ref{prop:can-proj}. Conversely, if there exists 
% a uniformly continuous minor-preserving map from ${\mathscr C}$ to $\proj$, then there also
% exists an $\overline{\Aut(\bA)}$-invariant minor-preserving map from ${\mathscr C}$ to $\proj$ by Proposition~\ref{prop:can-proj}. The statement
% now follows from Proposition~\ref{prop:binary-main}. 
%\end{proof}

Theorem~\ref{thm:binary} can be strengthened further. 
%To formulate the strengthening, we introduce a new notion and prove a fact that might be of independent interest in model theory; this fact
%will also we used in the proof of  Theorem~\ref{thm:main}.  
We have already seen that every homogeneous $\omega$-categorical Ramsey structure contains two independent elementary substructures. The idea of the following theorem is that it suffices to verify the UIP on such a pair of substructures.

\begin{theorem}
%[Binary independent UIP verification]
\label{thm:binary-indep}
Let $\bA$ be a countable finitely homogeneous Ramsey structure 
and let $\bC$ be a reduct of $\bA$. 
Let ${\mathscr C}$ be the clone of all polymorphisms of $\bC$ that are canonical over $\bA$ and 
let $\zeta \colon {\mathscr C} \to \proj$
be an $\overline{\Aut(\bA)}$-invariant minor-preserving map. 
Then for every pair $(\bA_1,\bA_2)$ of independent elementary substructures of $\bA$  the following are equivalent. 
\begin{enumerate}
\item $\zeta$ has the UIP with respect to $\Pol(\bC)$ over $\bA$. 
\item for all 
$f \in \Pol(\bC)^{(2)}$,
$u_1,u_2 \in {\overline{\Aut(\bA)}}$ 
with 
$u_1(A) \subseteq A_1, u_2(A) \subseteq A_2$, if 
$f(\id,u_1)$ and $f(\id,u_2)$ are canonical over $\bA$ then 
$$\zeta(f(\id,u_1)) = \zeta(f(\id,u_2)).$$ 
\end{enumerate}
\end{theorem}

%We finally prove a strengthening of
%Theorem~\ref{thm:binary} which allows the
%verification of the UIP to take place in two independent elementary substructures of 
%the underlying Ramsey structure. 
%In Section~\ref{sect:rcc5} we apply these general results to prove a new complexity dichotomy result about the spatial reasoning formalism RCC-5. 

\section{Proof of the Extension Lemma}
\label{sect:ext-proofs}
This section contains the proof of Lemma~\ref{lem:uip};
Corollary~\ref{cor:dicho} about the complexity of CSPs is then an easy consequence. 
%We first need to introduce more concepts that will be needed in the proof. 
%(Section~\ref{sect:canon}, 

\subsection{Diagonal Interpolation} 
In our proofs we need the concept of \emph{diagonal interpolation}, which is a more restricted
form of interpolation, and was already used in~\cite{MMSNP}.

\begin{definition}
Let $\mathscr M$ be a set of functions from $A$ to $A$. We say that
$f \colon A^k \to A$ \emph{diagonally interpolates $g \colon A^k \to A$ 
over ${\mathscr M}$} if for every finite $F \subseteq A$ there are $u,v \in {\mathscr M}$ 
such that $g(a_1,\dots,a_k) = u(f(v(a_1),\dots,v(a_k)))$ for all $a_1,\dots,a_k \in F$. 
%$$g \in \overline{ \{u(f(v,\dots,v)) \colon u,v \in {\mathscr M} \} }.$$
%for every finite $F \subseteq B$. 
\end{definition}

%An operation $f \colon A^k \to A$  is called \emph{diagonally canonical over $\bA$} if it is canonical over $(\bA^{[k]},\bA)$ (i.e., we use a full power of $\bA$ instead of an algebraic power). 
Diagonal interpolation 
 is well-behaved with respect to 
minors. 

\begin{lemma}\label{lem:diag-minor}
Let $\mathscr M$ be a set of functions from $A$ to $A$, let $f \colon A^k \to A$ be an operation, and let $\alpha \colon \{1,\dots,k\} \to \{1,\dots,n\}$. 
%$\pi = (\pi^n_{i_1},\dots,\pi^n_{i_k})$ be a vector of $k$ projections, each of arity $n$. 
Suppose that $f$ diagonally interpolates 
$g$ over $\mathscr M$. Then  
$f_\alpha$ diagonally
interpolates $g_\alpha$  over ${\mathscr M}$. 
\end{lemma}
\begin{proof}
By assumption 
$g \in \overline{ \{ u(f(v,\dots,v)) \colon u,v \in \mathscr M \} }$. 
%Note that for all $b_1,\dots,b_n$ we have
%\begin{align*}
%u(f(v,\dots,v) \circ \pi)(b_1,\dots,b_n) 
%& = 
%u(f(v,\dots,v))(b_{i_1},\dots,b_{i_k}) \\
%& = u(f(v(b_{i_1}),\dots,v(b_{i_k}))) \\
%& = u(f \circ \pi)(v(b_{i_1}),\dots,v(b_{i_k}))
%\end{align*} 
Since composition in clones %with $\pi$ 
is continuous, 
\begin{align*}
g_\alpha & \in \overline{ \{ u(f(v,\dots,v))_\alpha \colon u,v \in \mathscr M \} }  \\
& = \overline{ \{ u(f_\alpha)(v,\dots,v)  \colon u,v \in \mathscr M \} } \end{align*}
and thus $f_\alpha$ diagonally interpolates $g_\alpha$ over ${\mathscr M}$. 
\end{proof}

We want to stress that the innocent-looking Lemma~\ref{lem:diag-minor} fails for interpolation instead of diagonal interpolation, as shown in 
the following example.
For $H_1,\dots,H_k \subseteq A$ define 
$$f(H_1,\dots,H_k)  \coloneqq \{f(a_1,\dots,a_k) \colon a_1 \in H_1,\dots,a_k \in H_k\}.$$

% shows that Lemma~\ref{lem:diag-minor} fails for interpolation instead of diagonal interpolation. 

\begin{example}\label{expl:canonisation-violates-minor}
 Let $\bA$ be a structure with countable domain and two disjoint infinite unary relations $U_1$ and $U_2$. 
 %Then $\bA$ is homogeneous and $\Aut(\bA) = \Sym(U_1) \times \Sym(U_2)$. 
 Let $f \colon A^3 \to A$ be an injective function such that for every $i \in \{1,2\}$ 
 \begin{itemize}
 \item $f(a,b,c) \in U_i$ if $a \in U_i$ and $b \neq c$, 
 \item $f(a,b,c) \in U_i$ if $b \in U_i$ and $b = c$. 
\end{itemize}
%Note that $f$ is diagonally canonical but not canonical. 
%It is easy to see that such a function exists. 
Let
$e_1,e_2$ be two self-embeddings of $\bA$ with disjoint images and let $g \colon A^3 \to A$ be given by
$$g(x,y,z) \coloneqq f(x,e_1(y),e_2(z)).$$
Then $g$ is interpolated by $f$, injective and for each $i \in \{1,2\}$ we have $g(U_i,A,A) \subseteq U_i$. Hence, $g$ is canonical and
$\xi_1(g) = \pi^3_1$. 
Let $\alpha \colon \{1,2,3\} \to \{1,2\}$ be given by $\alpha(1) = 1$, $\alpha(2) = 2$, and $\alpha(3) = 2$.  
We claim that 
$f_\alpha$ does not 
interpolate $g_\alpha$. 
On the one hand, for each $i \in \{1,2\}$ we have $f_\alpha(A,U_i) \subseteq U_i$,  
so $f'$ is canonical and $\xi_1(f_\alpha) = \pi^2_2$. 
On the other hand 
$$\xi_1(g_\alpha) = \xi_1(g)_\alpha = (\pi^3_1)_\alpha = \pi^2_1$$
which proves the claim. 
\end{example}

The next lemma plays an important role 
in the proof of Lemma~\ref{lem:uip}.
%;  the proof can be found in Appendix~\ref{app:diag}. 

\begin{lemma}\label{lem:diagonal-common-interpolation}
Let $\bB$ be a reduct of a countable 
finitely homogeneous Ramsey structure $\bA$. 
Let $f \in \Pol(\bB)$. Then there exists $g \in \Pol(\bB)$ that diagonally interpolates both $f$
and an operation in $\Can(\bA)$ over $\Aut(\bA)$. 
\end{lemma}

\begin{proof}
Let $k$ be the arity of $f$. 
We first show a local version of the statement,
and then derive the statement by a compactness argument. The local version is that for every finite 
$X \subseteq A$ there exists $g \in \Pol(\bB)$
and $e \in \overline{\Aut(\bA)}$ such that $g(e,\dots,e)$ is canonical on $X^k$ over $\bA$ and $g$ agrees with $f$ on $X^k$. 

Let $c = (c_1,\dots,c_n)$ be such that 
$\{c_1,\dots,c_n\} = X$. By Theorem~\ref{thm:indep} there are
two independent elementary substructures $\bC_1$ and $\bC_2$ of $\bA$.
Since $\bA$ is $\omega$-categorical, for
$i \in \{1,2\}$ there exists an isomorphism 
$e_i$ from $\bA$ to $\bC_i$; by the homogeneity of $\bA$ we have $e_i \in \overline{\Aut(\bA)}$. 
The restriction 
of $e_1$ to $X$ can be extended to an automorphism 
$\delta \in \Aut(\bA)$. Then $\delta^{-1}(C_1)$
and $\delta^{-1}(C_2)$ induce independent elementary substructures of $\bA$, and $X \subseteq \delta^{-1}(C_1)$. 
So we may assume without loss of generality that $X \subseteq C_1$. 

Since $\bA$ is Ramsey, $f(e_2,\dots,e_2)$ interpolates a canonical function over $\Aut(\bA)$; so there are $\delta_1,\dots,\delta_k \in \Aut(\bA)$ such that $f(e_2 \delta_1,\dots, e_2 \delta_k)$ is canonical on $X^k$. 
In other words, for $Y_i \coloneqq e_2 (\delta_{i}(X))$,
the map $f$ is canonical on
$Y_1 \times \cdots \times Y_k$ over $\bA$. 
Since $X \subseteq C_1$ and 
$Y_i \subseteq C_2$, the tuples $e_2\delta_1(c),\dots,
e_2\delta_k(c)$ all have the same 
type over $c$. 
Since $\bC_1$ and $\bC_2$ are independent, there exist $\alpha_i  \in \Aut(\bA)$
such that $\alpha_i(Y_1) = Y_i$ and 
$\alpha_{i}(c) = c$. 
Define $g \coloneqq f(\alpha_1,\dots,\alpha_k)$. 
Then $f$ and $g$ agree on $X^k$. 
We claim that $h \coloneqq g(e_2 \delta_1,\dots,e_2 \delta_1)$ is canonical on $X^k$ over $\bA$. 
This follows from $f$ being canonical on 
\begin{align*}
Y_1 \times \cdots \times Y_k & = \alpha_1(Y_1) \times \cdots \times \alpha_k(Y_1) \\
& = \alpha_1 e_2 \delta_1(X) \times 
\cdots \alpha_k e_2 \delta_{1}(X)
\end{align*}
since $h$ equals $f (\alpha_1 e_2 \delta_1,\dots, \alpha_k e_2 \delta_{1})$ on $X^k$. 

\medskip 
To show how the local version implies the statement of the lemma, 
let $\sigma$ be the signature of $\bB$
and let $\bC$ be an expansion of 
$\bB$ by countably many constants 
such that every element of
$\bC$ is named by a constant symbol. 

Let $\rho$ be the signature of $\bC$ together
with a new a $k$-ary function symbol $g$. 
Consider the $\rho$-theory $T$ consisting
of the union of the following first-order sentences: 
\begin{enumerate}
\item $\Th(\bC)$; 
\item A first-order sentence which asserts that $g$ preserves all relations from $\sigma$; 
\item For all $l \in {\mathbb N}$ and constant symbols $c_0,c_1,\dots,c_k \in \rho$ such that $f(c_1^{\bC},\dots,c_k^{\bC}) = c_0^{\bC}$ the sentence $g(c_1,\dots,c_k)=c_0$. 
\item for all constant symbols $c_1,\dots,c_l \in \sigma$ the sentence $\psi_{c_1,\dots,c_l}$ which expresses that there exist $y_1,\dots,y_l$ such that $\typ^{\bA}(y_1,\dots,y_l) = \typ^{\bA}(c_1,\dots,c_l)$ and
$g$ is canonical on $\{y_1,\dots,y_l\}^k$. 
\end{enumerate}
It follows from the basic facts about $\omega$-categorical structures from Section~\ref{sect:omega-cat} that such sentences exist. 

{\bf Claim.} Every finite subset $S$ of $T$ has a model. Let $X$ be the (finite) set of all constants $c$ such that the respective constant symbol appears in a sentence of $S$ or in $\{c_1,\dots,c_l\}$ for some $\psi_{c_1,\dots,c_l} \in S$. Let $d_1,\dots,d_n$ be some enumeration 
of $X$. Then $\psi_{d_1,\dots,d_n}$ implies 
every sentence in $S'$ from item (4). By the 
local version of the statement there exists $g \in \Pol(\bB)$ 
and $e \in \overline{\Aut(\bA)}$ such that
\begin{itemize} 
\item $g(e,\dots,e)$ is canonical on $X^k$ 
over $\bA$, and 
\item $g$ agrees with $f$ on $X^k$. 
\end{itemize}
Let $\bD$ be the $\rho$-expansion of $\bC$ 
where $g^{\bD} \coloneqq g$. Then $\bD$ satisfies $S$. 
This is clear for the sentences from item (1), (2), and (3).  Finally, since $g(e,\dots,e)$ is canonical on $X^k$ over $\bA$ the sentences of $S'$ from item (4) can be satisfied by setting $y_i \coloneqq e(c_i)$. 

\medskip 
By the compactness theorem of first-order logic, 
$T$ has a model $\bD$. By the downward L\"owenheim-Skolem theorem we may also assume that $\bD$ is countable.
Then the $\sigma$-reduct of $\bD$ satisfies $\Th(\bB)$ and hence is isomorphic to $\bB$ 
by the $\omega$-categoricity of $\bB$. We may
therefore identify the element of $\bD$ with the elements of $\bB$ along this isomorphism and henceforth assume that $\bD$ is an expansion of $\bA$. Because of the sentences under (2) we 
have that $g^{\bD}$ is a polymorphism of $\bB$. 
The map $e \colon B \to B$ given by $c^{\bC} \mapsto c^{\bD}$ is an elementary self-embedding. 
The sentences under (3) imply that $g(e,\dots,e)$ 
equals $e(f)$, so $g$ diagonally interpolates $f$. 
Finally, the sentences under (4) imply that $g$ diagonally interpolates some function from $\Can(\bB)$. 
\end{proof}

\subsection{Rich subsets}
\label{sect:rich}
To prove the existence of canonical
functions it is useful to introduce a notion of \emph{rich} substructures of a structure; this will be needed in the proofs for Lemma~\ref{lem:uip},
but also explicitly in Section~\ref{sect:rcc5} (e.g., in~Lemma~\ref{lem:copy}) when we apply the general results to the spatial reasoning formalism RCC5. 

\begin{definition}\label{def:rich}
Let $\bA$ be a countable finitely homogeneous structure and let $m$ be the maximal arity of $\bA$.
Let 
${\mathscr C} \subseteq {\mathscr O}_A$ be a closed clone containing
$\Aut(\bA)$ and let $k \in {\mathbb N}$.  Then we say that $X \subseteq A$ is \emph{$k$-rich with respect to ${\mathscr C}$} if 
\begin{itemize}
\item every $m$-orbit of $\Aut(\bA)$ 
%$\bA^{(k)}$ 
contains a tuple from $X^m$. 
%$(X^k)^m$; 
\item every behaviour over $\bA$ which is realised on $X^k$ by some operation in $\mathscr C$ is also realised on $A^k$ by some operation in 
$\mathscr C$. 
\end{itemize}
%A substructure of 
\end{definition}

Note that if $X$ is $k$-rich then every behaviour
which is realisable on $X^k$ is complete.  
It is clear from the definition that if $X \subseteq A$ is $k$-rich, and $X \subseteq Y$, then $Y$ is also $k$-rich. 

%We present an example that satisfies the first, but not the second condition in Definition~\ref{def:rich}. 
% Where did it go? 

\begin{lemma}\label{lem:rich}
Let $\bA$ be a countable finitely homogeneous structure. Let 
$\mathscr C$ be a closed clone containing $\Aut(\bA)$ and let $k \in {\mathbb N}$.
Then there exists a finite $X \subseteq A$ which is $k$-rich with respect to ${\mathscr C}$. 
\end{lemma} 

\begin{proof}%[Proof of Lemma~\ref{lem:rich}]
Let $m$ be the maximal arity 
of $\bA$. 
%The structure $\bA^{(k)}$ is $\omega$-categorical 
%and hence has finitely many $m$-orbits.
Since ${\mathcal O}_m(\bA)$ is finite,  
we can choose a finite subset $X_0$ of $A$ such that every orbit of $m$-tuples in $\Aut(\bA)$ intersects
$X_0^m$. Let $X_0 \subset X_1 \subset \cdots$ be finite subsets of $A$ such that $\bigcup_{i=1}^{\infty} X_i = A$. Let $\mathcal B_i$ denote the (finite) set of all $m$-behaviours which are realised on $X^k_i$ by some operation of ${\mathscr C}$. Then ${\mathcal B}_0 \supseteq {\mathcal B}_1 \supseteq  \cdots$ and hence there exists an $N$ such that ${\mathcal B}_N = {\mathcal B}_{N+1} = \cdots$.
Then Proposition~\ref{prop:compact} implies 
that every behaviour in ${\mathcal B}_N$ is realised on $A^k$ by some operation in ${\mathscr C}$. 
This proves that $X_N$ is $k$-rich:
the first item is satisfied because $X_0 \subseteq X_N$, and the second item because every
behaviour that is realised in $X^k_N$ is also realised on $A^k$. 
\end{proof}

\begin{definition}\label{def:rich-behave}
Let $\bA$ and $\mathscr C$ be as in Definition~\ref{def:rich}, let $\eta$ be a
function defined on ${\mathscr C} \cap \Can(\bA)$, let $F$ be $k$-rich with respect to ${\mathscr C}$ for some $k \geq 1$, and  
let $f \in \mathscr C^{(k)}$. 
Suppose that $f$ realises on $F^k$ the same 
complete behaviour as $g \in \Can(\bA)$. 
Then we will also write $\eta(f|_{F^k})$ 
for $\eta(g)$. 
\end{definition}

%Another important ingredient for the proof
%of Lemma~\ref{lem:uip} will be introduced in the next section. 

\subsection{Interpolation invariance}
\label{sect:int-inv}
It has been shown in~\cite{Topo-Birk}
(see Theorem 6.4 in~\cite{wonderland})
that if $\mathscr C \subseteq {\mathscr O_B}$ is a closed clone and $\mathscr G  \subseteq {\mathscr C}$ an oligomorphic permutation group over the same base set $B$, 
then every $\mathscr G$-invariant continuous 
map $\zeta$ defined on $\mathscr C$ is 
uniformly continuous. 
In Lemma~\ref{lem:int-inv} we present a sufficient condition for uniform continuity which does not require that the map $\zeta$ we start from is continuous, 
for the special case that $\mathscr G = \Aut(\bC)$ for a finitely homogeneous Ramsey 
structure $\bC$.

%\begin{definition}
%Let $\bA$ be a structure and $\mathscr C$ a closed operation clone over the domain $A$. 
%Then a function $\xi$ defined on ${\mathscr C}$
%is called \emph{interpolation invariant over $\bA$} if $\xi(f) = \xi(g)$ whenever $f \in {\mathscr C}$ interpolates $g \in {\mathscr C}$ over $\bA$. 
%\end{definition}

\begin{definition}
Let $\mathscr C$ be a closed clone on a set $B$ and let $\mathscr M$ be a set of functions from $B$ to $B$. 
A function $\zeta$ defined on ${\mathscr C}$
is called \emph{interpolation invariant over $\mathscr M$} if $\zeta(f) = \zeta(g)$ whenever $f \in {\mathscr C}$ interpolates $g \in {\mathscr C}$ over $\mathscr M$. 
\end{definition}

Clearly, interpolation invariance over $\mathscr M$ implies ${\mathscr M}$-invariance; see~\eqref{eq:inv}. 

\begin{lemma}\label{lem:int-inv}
Let $\bB$ be a first-order reduct of a countable finitely homogeneous 
%$\omega$-categorical 
Ramsey structure $\bA$. 
Let $\zeta$ be a function defined on $\Pol(\bB)$ which is interpolation 
invariant over $\Aut(\bA)$.  
Then $\zeta$ is uniformly continuous. 
\end{lemma}
% This is Lem 6.10 in Bertalan's draft. 

%The proof of Lemma~\ref{lem:int-inv} can be found in Section~\ref{app:int-inv}. 

\begin{proof}%[Proof of Lemma~\ref{lem:int-inv}]
Let $m$ be the maximal arity of $\bA$. 
%It follows from Lemma~\ref{lem:uc-can} that the restriction of $\zeta$ to $\Pol(\bB) \cap \Can(\bA)$
%is uniformly continuous. 
Suppose for contradiction that $\zeta$ is not uniformly continuous. Then there exists a $k \in {\mathbb N}$ so that for every finite $X \subseteq A$ there exist $f_X,g_X \in \Pol(\bB)$ such that 
$(f_X)|_{X^k} = (g_X)|_{X^k}$, but $\zeta(f_X) \neq \zeta(g_X)$. 
By Lemma~\ref{lem:rich}, there exists a finite $A_0 \subseteq A$ which is $k$-rich with respect to $\Pol(\bB)$.
%We define $\zeta_n$. 

{\bf Claim.} For every finite $F \subseteq A$ that contains $A_0$
there exists an $h \in \Pol(\bB)^{(k)}$ and 
$h' \in \Pol(\bB)^{(k)} \cap \Can(\bA)$ 
such that $h$ and $h'$ realise the same complete behaviour on $F^k$ and 
$\zeta(h) \neq \zeta(h')$. 
By Lemma~\ref{lem:canon2} there exists
a finite $C \subseteq A$ 
such that for every $h \colon A^k \to A$
there are $\alpha_1,\dots,\alpha_k \in \Aut(\bA)$
such that $\alpha_1(F) \subseteq C,\dots,\alpha_k(F) \subseteq C$ 
and $h(\alpha_1,\dots,\alpha_k)$ is canonical on $F^k$. 
In particular, this holds for the operations 
$f_C,g_C$ introduced above, so that 
$f' \coloneqq f_C(\alpha_1,\dots,\alpha_k)$ and 
$g' \coloneqq g_C(\alpha_1,\dots,\alpha_k)$ 
are canonical on $F^k$. 
We have $(f_C)|_{C^k} = (g_C)|_{C^k}$ and $\zeta(f_C) \neq \zeta(g_C)$, 
and $\zeta(f_C) = \zeta(f')$ and $\zeta(g_C) = \zeta(g')$
because $f_C$ interpolates $f'$ and $g_C$ interpolates $g'$ over $\Aut(\bA)$.
Since $A_0 \subseteq F$
is $k$-rich 
there exist $h' \in \Pol^{(k)}(\bB) \cap \Can(\bA)$ such that $h'$ realise the same complete behaviour as $f'$ on $F^k$, and therefore also the same complete behaviour as $g'$ on $F^k$. 
% and $g''$ has the same behaviour as $f''$ on $F^k$. 
Hence, we must have
$\zeta(h') \neq \zeta(f')$ or $\zeta(h') \neq \zeta(g')$, and hence either $(f',h')$ or 
$(g',h')$ provide us witnesses for the claim. 

Let $A_0 \subseteq A_1 \subseteq \cdots$ be finite such that
$\bigcup_{l \in {\mathbb N}} A_l = A$. By the claim above, for each $l \in {\mathbb N}$ we can find a function $h \in \Pol(\bB)$ and 
$h' \in \Pol(\bB)^{(k)} \cap \Can(\bA)$ 
such that $h$ and $h'$ have the same $m$-behaviour ${\mathcal H}_l$ on $(A_l)^k$ and $\zeta(h) \neq \zeta(h')$. 
By Lemma~\ref{lem:canon} the
operation $h$ interpolates an operation $g$
with a complete behaviour ${\mathcal B}_l$
% \in \Can(\bA)$ 
over $\bA$, and $\zeta(g) = \zeta(h) \neq \zeta(h')$. 
Since there are finitely many $m$-behaviours over $\bA$, there exists
$({\mathcal H},{\mathcal B})$ such that $({\mathcal H}_l,{\mathcal B}_l) = ({\mathcal H},{\mathcal B})$  for infinitely many $l \in {\mathbb N}$.

%{\bf Claim 2.} There are 
%$h_1,h_2 \in \Pol^{(k)}(\bB) \cap \Can(\bA)$
%such that $\zeta(h_1) \neq \zeta(h_2)$
%and for all finite $F \subseteq A$
%that contain $A_0$ there 
%exists an $f \in \Pol(\bB)$ and
%$\alpha_{i,1},\dots,\alpha_{i,k} \in \Aut(\bA)$ for $i \in \{1,2\}$ such that 
%$f(\alpha_{i,1},\dots,\alpha{i,k})$ 
%has the same behaviour as $h_i$
%on $F^k$. 

%Claim 2 in combination with 

Lemma~\ref{lem:simul-interpol}
shows that there exists 
$f \in \Pol(\bB)$ and $e_{1},e'_{1},\dots,e_{k},e_{k}' \in \overline{\Aut(\bA)}$ such that 
$f \circ (e_{i},\dots,e_{k})$ has behaviour ${\mathcal H}$
and $f \circ (e'_{i},\dots,e'_{k})$ has behaviour ${\mathcal B}$. 
Hence, $f$ interpolates two functions that $\zeta$ maps to different values, contradicting interpolation invariance of $\zeta$. 
\end{proof}

\begin{lemma}\label{lem:uc-can}
Let $\bA$ be a finitely 
homogeneous structure 
and let $\mathscr C$ be a subclone
of $\Can(\bA)$. 
Let $\zeta \colon {\mathscr C} \to \proj$ be $\overline{\Aut(\bA)}$-invariant. 
Then $\zeta$ is uniformly continuous. 
\end{lemma}
\begin{proof}
Let $m$ be the maximal arity of $\bA$, and 
let $X \subseteq A$ be finite so that every 
orbit of $m$-tuples contains a witness in $X^m$. 
Let $k \in {\mathbb N}$ and $f,g \in {\mathscr C}^{(k)}$ be such that
$f|_{X^k} = g|_{X^k}$. Then $\zeta(f) = \zeta(g)$ by the choice of $X$. By Lemma~\ref{lem:lift}, there are 
$e_1,e_2 \in \overline{\Aut(\bA)}$ such that
$e_1 \circ f = e_2 \circ g$. 
The $\overline{\Aut(\bA)}$-invariance of $\zeta$ implies that $\zeta(f) = \zeta(e_1 \circ f) = \zeta(e_2 \circ g) = \zeta(g)$, proving the uniform continuity of $\zeta$. 
\end{proof}

The following can be obtained by combining known results in the literature.  
An operation $f \colon A^k \to A$ is called \emph{cyclic} if $k \geq 2$ and  $f(x_1,\dots,x_k) = f(x_2,\dots,x_k,x_1)$ for all $x_1,\dots,x_k \in A$. 
%except for the addition of item (2). 

\begin{proposition}
\label{prop:can-proj}
Let $\bB$ be finitely 
homogeneous 
with maximal arity $m \in {\mathbb N}$. Let
$\bC$ be a first-order reduct of $\bB$ which is
a model-complete core. Suppose that 
${\mathscr C} \coloneqq \Pol(\bC) \subseteq \Can(\bB)$. Then the following are equivalent. 
\begin{enumerate}
\item There is no uniformly continuous minor-preserving map from ${\mathscr C} \to \proj$. 
\item There is no $\overline{\Aut(\bB)}$-invariant 
minor-preserving map ${\mathscr C} \to \proj$. 
%\item $\xi^{\bA}_m({\mathscr C})$ satisfies some h1 identities
%that do not hold in $\proj$;
\item $\xi^{\bB}_m({\mathscr C})$ has no minor-preserving map to $\proj$. 
\item $\xi^{\bB}_m({\mathscr C})$ has a cyclic  operation. 
%\item $\xi^{\bB}_m({\mathscr C})$ has a Siggers operation. 
%\item ${\mathscr C}$ has a pseudo-Siggers operation. 
\end{enumerate}
\end{proposition}
\begin{proof}
$(1) \Rightarrow (2)$: An
immediate consequence of Lemma~\ref{lem:uc-can}. 

%$(3) \Leftrightarrow (4)$: it is clear that the identities from the definition of Siggers operations do not hold in $\proj$.
%; another example of such 
%identities is $f(x,y) = f(y,x)$ ($f$ is called \emph{symmetric} operation). 
%The equivalence is a 
%well-known fact about finite algebras~\cite{Siggers}. 

$(2) \Rightarrow (3)$: 
If $\xi^{{\mathfrak C}}_m({\mathscr C})$ has a minor-preserving map to $\proj$, then we can compose
it with $\xi^{{\mathfrak C}}_m$ and obtain a 
$\overline{\Aut(\bB)}$-invariant 
 minor-preserving map from ${\mathscr C}$ to $\proj$, and $(2)$ does not hold. 
 
$(3) \Rightarrow (4)$ %and $(4) \Rightarrow (5)$ 
has been shown in~\cite{wonderland,Cyclic}. 
%TODO: expand? 

%Theorem~\ref{thm:fin-dichotomy}.
% implies that 
%$\xi^{{\mathfrak C}}_m({\mathscr C})$ has a Siggers operation. 
%$(5) \Rightarrow (6)$ is a consequence of Lemma~\ref{lem:lift} applied 
%to ${\mathfrak A} \coloneqq {\mathfrak B}$. 

$(4) \Rightarrow (1)$. Since ${\mathfrak C}$ is finitely homogeneous, the number of orbits of $n$-tuples of $\Aut({\mathfrak C})$ is bounded by $2^{p(n)}$ for some polynomial $p$, and hence less than doubly exponential. Therefore, the same applies to $\bB$, and  
the statement follows from results in~\cite{BKOPP-equations}. 
%Theorem~\ref{thm:inf-dichotomy}. 
\end{proof}

\subsection{Proofs of main results}
\label{sect:main-proofs}
We can now prove the Extension Lemma (Lemma~\ref{lem:uip}). 
\begin{proof}[Proof of Lemma~\ref{lem:uip}]
Let $f \in \Pol(\bB)^{(k)}$. 
Since $\bA$ is a finitely homogeneous Ramsey structure,
we can apply the canonisation lemma (Lemma~\ref{lem:canon}), 
% and obtain
%the existence of 
%$$g \in \Can(\bA) \cap \overline{\{u(f(v_1,\dots,v_k)) \colon u,v_1,\dots,v_k \in \Aut(\bA) \} }.$$
which implies that $f$ interpolates a function $g \in \Can(\bA)$ over $\Aut(\bA)$. 
Define $\bar \zeta(f) \coloneqq \zeta(g)$; this is well-defined because $\zeta$ has the UIP with respect to $\Pol(\bB)$. Note that if $f \in \Can(\bA)$, then $\bar \zeta(f) = \zeta(f)$,
so $\bar \zeta$ extends $\zeta$. 

{\bf Claim 1.} $\bar \zeta$ is $\overline{\Aut(\bA)}$-invariant.
%, which implies $\Aut(\bC)$-invariance. Who cares? 
Suppose that $f,g \in \Pol(\bB)$ are such that $f$ interpolates $g$ over $\Aut(\bA)$. 
Lemma~\ref{lem:canon}
implies that $g$ interpolates over $\Aut(\bA)$ an operation $h \in \Can(\bA)$. Then $f$ interpolates $h$ over  $\Aut(\bA)$, too, and we have
$\bar \zeta(f) = \zeta(h) = \bar \zeta(g)$. 

{\bf Claim 2.} 
$\bar \zeta$ is uniformly continuous. This follows
from Claim 1 by Lemma~\ref{lem:int-inv}. 

{\bf Claim 3.} $\bar \zeta$ is minor-preserving. Arbitrarily choose $f \in \Pol(\bB)^{(k)}$ and 
%a vector $\pi = (\pi^n_{i_1},\dots,\pi_{i_k}^n)$ of projections of arity $n \in {\mathbb N}$. 
$\alpha \colon \{1,\dots,n\} \to \{1,\dots,k\}$. 
We have to show that 
$\bar \zeta(f_\alpha) = \bar \zeta(f)_\alpha$. 
By Lemma~\ref{lem:diagonal-common-interpolation}, there exists $g \in \Pol(\bB)^{(k)}$
that diagonally interpolates $f$ and diagonally 
%(even strongly) 
interpolates 
$h \in \Pol(\bB)^{(k)} \cap \Can(\bA)$. 
We obtain
\begin{align*}
\bar \zeta(f_\alpha) & = \bar \zeta(g_\alpha) \quad \quad\quad \quad  \text{(Lemma~\ref{lem:diag-minor} and interpolation-invariance of $\bar \zeta$)} \\
& = \bar \zeta(h_\alpha) \quad \quad\quad \quad  \text{(Lemma~\ref{lem:diag-minor} and interpolation-invariance of $\bar \zeta$)} \\
& = \zeta(h_\alpha) \quad \quad \quad\quad \text{($h_\alpha \in \Pol(\bB) \cap \Can(\bA)$)} \\
& = \zeta(h)_\alpha \quad \quad \quad\quad \text{($\zeta$ is minor-preserving)} \\
& = \bar \zeta(h)_\alpha = \bar \zeta(g)_\alpha = \bar \zeta(f)_\alpha.
\end{align*}
Claim 2 and 3 imply the statement of the lemma. 
\end{proof}

\begin{proof}[Proof of Corollary~\ref{cor:dicho}]
%Let $\bC \in {\mathcal C}$.
%; then $\bC$
%has the same CSP as its model-complete core, which is again in ${\mathcal C}$, so we may assume without loss of generality that $\bC$ is a model-complete core. \red{By assumption, $\bC$ is a reduct of a finitely bounded homogeneous structure $\bB$ which itself is the reduct of a finitely homogeneous Ramsey structure $\bA$.}
First suppose that $\Pol(\bC) \cap \Can(\bB)$ 
does not have a uniformly continuous minor-preserving map to the projections. 
%contains a pseudo-Siggers operation. 
%Since $\bB$ is a finitely bounded homogeneous structure, 
Theorem~\ref{thm:can-tract} implies
that $\Csp(\bC)$ is in P. 
%Otherwise, if $\Pol(\bC) \cap \Can(\bB)$ 
%does not contain a pseudo-Siggers operation, 
%then  
%$\Pol(\bC) \cap \Can(\bB)$ has a uniformly continuous minor-preserving map to $\proj$
%by Proposition~\ref{prop:can-proj}. 
Now suppose that $\Pol(\bC) \cap \Can(\bB)$ 
has a uniformly continuous minor-preserving map to $\proj$. 
Then the assumptions imply that there is also 
a uniformly continuous minor-preserving map 
 $\zeta \colon \Pol(\bC) \cap \Can(\bA) \to \proj$ 
 that has the UIP with respect to 
$\Pol(\bC)$ over $\bA$.  
Lemma~\ref{lem:uip}  
shows that $\zeta$ 
can be extended to a uniformly continuous minor-preserving map $\bar \zeta \colon \Pol(\bC) \to \proj$. 
%Note that $\bA$, and hence also $\bB$ 
%and $\bC$, have 
%less than doubly exponential growth since 
%$\bA$ is finitely homogeneous. 
%Therefore, 
%Theorem~\ref{thm:inf-dichotomy} implies
%that $\bC$ cannot have a pseudo-Siggers polymorphism. 
%Then $\Csp(\bC)$ is NP-complete by the main result of~\cite{Topo-Birk}. 
Then $\Csp(\bC)$ is NP-hard by Theorem~\ref{thm:hard}. Since CSPs of reducts of finitely bounded structures are in NP, this concludes the proof. 
%\red{TODO: unclear where the m.c. core assumption of $\bC$ is used.} 
\end{proof}

Note that the proof of Corollary~\ref{cor:dicho} shows
that Conjecture~\ref{conj:inf-tract} holds for all structures $\bC \in {\mathcal C}$.
 
%\begin{porism}
%\red{Let $\mathcal C$ be as in the statement of Corollary~\ref{cor:dicho}. Then for every $\bC \in {\mathcal C}$
%\begin{itemize}
%\item either there is 
%a uniformly continuous minor-preserving map from $\Pol(\bC)$ to $\proj$ (and $\Csp(\bC)$ is NP-hard), 
%\item or the model-complete core of $\bC$ contains a pseudo weak near unanimity operation  (and $\Csp(\bC)$ is in P).
%\end{itemize}}
% do we really need the mc core assumption?
% is it enough? 
%\end{porism} 
%\begin{proof}
%The two cases are disjoint:  if there is a uniformly continuous minor-preserving map from $\Pol(\bC)$ to $\Proj$, then also from the model-complete core of $\bC$ (see~\cite{wonderland}). 
%Since $\bC$ is a reduct of a finitely homogeneous structure, the implication $(5) \Rightarrow (
%if $\Pol(\bC) \cap \Can(\bB)$ 
%does not have a uniformly continuous minor-preserving map $\proj$
%The proof of Corollary~\ref{cor:dicho} shows that 
%it suffices to prove that 
%\begin{itemize}
%\item , then 
%$\Pol(\bC)$ contains a pseudo weak near unanimity operation, and
%\item if $\Pol(\bC)$ contains a pseudo weak near unanimity operation, then $\Pol(\bC)$ has no uniformly continuous minor-preserving map to $%\proj$. 
%\end{itemize}
%Both statements follows from the literature: TODO. 
%\end{proof} 

\section{Verification of the UIP}
\label{sect:binary}
In this section we show that if a minor-preserving map to the clone of projections does not have
the UIP, then this is witnessed by binary operations of a very special form, proving Theorem~\ref{thm:binary} and its strengthening Theorem~\ref{thm:binary-indep}. 
We need the following 
`higher-dimensional checker board'
canonisation lemma. 

\begin{lemma}\label{lem:checkers}
Let $\bA$ be a countable finitely homogeneous Ramsey structure 
and $f \colon A^k \to A$. Suppose that $f$ interpolates over $\bA$ 
the operations $h_1,\dots,h_m \in \Can(\bA)$. Then for every finite $X \subseteq A$ there exist
$\alpha_{1,1},\dots,\alpha_{m,k} \in \Aut(\bA)$ such that 
\begin{itemize}
\item for every $i \in \{1,\dots,m\}$  the operation $f(\alpha_{i,1},\dots,\alpha_{i,k})$ has the same behaviour as $h_i$ on $X^k$, and 
\item for all $u_1,\dots,u_k \in \{1,\dots,m\}$ the operation
$f(\alpha_{u_1,1},\dots,\alpha_{u_k,k})$ is canonical on $X^k$. 
\end{itemize}
\end{lemma}
\begin{proof}
Let $u^1,\dots,u^p$ be an enumeration of $\{1,\dots,m\}^k$. We show by induction on $q \in \{1,\dots,p\}$ that for every finite $X \subseteq A$ there exist $\alpha_{1,1},\dots,\alpha_{m,k} \in \Aut(\bA)$ 
such that for every $i \in \{1,\dots,m\}$  the operation $f(\alpha_{i,1},\dots,\alpha_{i,k})$ has the same behaviour as $h_i$ on $X^k$, and 
for all $l \in \{1,\dots,q\}$ the operation
$f(\alpha_{u^l_1,1},\dots,\alpha_{u^l_k,k})$ is canonical on $X^k$. For $q = p$ we obtain the statement of the lemma. 

If $q = 0$ the statement of the claim follows from the assumptions of the lemma. For the inductive step, we assume that the statement holds for 
$q-1$
%$q-1 \in \{0,\dots,p-1\}$ 
%smaller values than $q$ 
and prove the statement for $q$. Let $X \subseteq A$ be finite. Lemma~\ref{lem:canon2} asserts the existence of a finite $C \subseteq A$ such that for every $f \colon A^k \to A$ there are $\beta_1,\dots,\beta_k \in \Aut(\bA)$ such that 
$\beta_1(X) \subseteq C,\dots,\beta_k(X) \subseteq C$ and 
$f(\beta_1,\dots,\beta_k)$ is canonical on $X^k$. We apply the induction hypothesis to $C$,
and obtain $\gamma_{1,1},\dots,\gamma_{m,k}$ such that for every $i \in \{1,\dots,m\}$ the operation $f(\gamma_{i,1},\dots,\gamma_{i,k})$
has the same behaviour as $h_i$ on $C^k$ 
and $f(\gamma_{u^l_1,1},\dots,\gamma_{u^l_k,k})$ is canonical on $C^k$ for every $l \in \{1,\dots,q-1\}$. 
Let $f' \coloneqq f(\gamma_{u^q_1,1},\dots,\gamma_{u^q_k,k})$. 
The property of $C$ implies that there are 
$\beta_1,\dots,\beta_k \in \Aut(\bA)$ such that
$\beta_1(X) \subseteq C,\dots,\beta_k(X) \subseteq C$ and 
$f'(\beta_1,\dots,\beta_k)$ is canonical on $X^k$. 
For $i \in \{1,\dots,m\}$ and $j \in \{1,\dots,k\}$ 
define $\alpha_{i,j} \coloneqq \gamma_{i,j} \circ \beta_j$. 
Observe $\alpha_{i,j}(X) \subseteq \gamma_{i,j}(C)$ and hence 
\begin{itemize}
%We first show that 
\item $f(\alpha_{u^l_1,1},\dots,\alpha_{u^l_k,k})$ is canonical on $X^k$ for $l \in \{1,\dots,q-1\}$: 
this follows from the inductive assumption that 
$f(\gamma_{u^l_1,1},\dots,\gamma_{u^l_k,k})$ 
is canonical on $C^k$;
\item $f(\alpha_{u^q_1,1},\dots,\alpha_{u^q_k,k})$ 
is canonical on $X^k$: this follows from the
property of $\beta_1,\dots,\beta_k$ that 
$$f'(\beta_1,\dots,\beta_k) = f(\alpha_{u^q_1,1},\dots,\alpha_{u^q_k,k})$$ is canonical on $X^k$;
\item for every $i \in \{1,\dots,m\}$, the operation $f(\alpha_{i,1},\dots,\alpha_{i,k})$ has the same behaviour as $h_i$ on $X^k$
because 
for every $j \in \{1,\dots,k\}$ and $f(\gamma_{i,1},\dots,\gamma_{i,k})$
has the same behaviour as $h_i$ on $C^k$ by the inductive assumption. 
\end{itemize}
This concludes the proof that $\alpha_{1,1},\dots,\alpha_{m,k}$ 
satisfy the inductive statement. 
\end{proof} 

We introduce useful notation for the proofs of Theorem~\ref{thm:binary} and Theorem~\ref{thm:binary-indep}. If $f \colon B^k \to B$, 
$u \colon B \to B$, and $\ell \in \{1,\dots,k\}$,
we write $f^{u}_\ell$ for the $k$-ary operation
defined by 
$$(x_1,\dots,x_k) \mapsto f(x_1,\dots,x_{\ell-1},u(x_\ell),x_{\ell+1},\dots,x_k).$$

%The central step of the proof is the following
%proposition.  
 %\begin{proposition}\label{prop:binary-main}
%Let $\bB$ be a reduct of a finitely homogeneous Ramsey structure $\bA$,
%let ${\mathscr C} \coloneqq \Pol(\bB) \cap \Can(\bA)$, 
%and let $\zeta \colon {\mathscr C} \to \proj$ be an  $\overline{\Aut(\bA)}$-invariant
%minor-preserving map. 
%Then the following are equivalent. 
%\begin{enumerate}
%\item $\zeta$
%has the UIP with respect to $\Pol(\bB)$ over $\bA$. 
%\item For all $f \in {\mathscr C}^{(2)}$,
%$u,v \in {\overline{\Aut(\bA)}}$, if $f^{u}_2,f^{v}_2
%\in \Can(\bA)$ then 
%$\zeta(f^{u}_2) = \zeta(f^{v}_2)$. 
%\end{enumerate}
%\end{proposition}

%$f(\pi^k_1,\dots,\pi^k_{i-1},\alpha,\pi^k_{i+1}, \dots,\pi^k_k)$ and $f(\pi^k_1,\dots,\pi^k_{i-1},\alpha,\pi^k_{i+1}, \dots,\pi^k_k)$

% which does not have the UIP. 
\begin{proof}[Proof of 
Theorem~\ref{thm:binary}]%Proposition~\ref{prop:binary-main}]
The forward implication is trivial. 
For the converse implication, suppose that
$\zeta$ does not have the UIP with respect to $\Pol(\bC)$ over $\bA$. That is, there are operations $g \in \Pol(\bC)$ and canonical $g_1,g_2 \in {\mathscr C}^{(k)}$ such that $g$ interpolates both $g_1$ and $g_2$ over $\Aut(\bA)$ and
$\zeta(g_1) \neq \zeta(g_2)$. By Lemma~\ref{lem:rich} there exists 
a finite $X_0 \subseteq A$ which is $k$-rich with respect to $\Pol(\bB)$.  

\medskip 
{\bf Claim.} For every finite $X \subseteq A$ that contains $X_0$ there exist $\ell \in \{1,\dots,k\}$, $f \in \Pol(\bB)^{(k)}$, $\alpha,\beta \in \Aut(\bA)$, 
%such that $\zeta(f^{\alpha}_{\ell}|_{X^k}) \neq \zeta(f^{\beta}_{\ell}|_{X^k})$ (recall Definition~\ref{def:rich-behave}). 
%and $h,h' \in \Pol(\bB)^{(k)} \cap \Can(\bA)$ 
such that 
\begin{itemize}
\item $f^\alpha_{\ell}$ has the same behaviour as
$g_1$ on $X^k$, and 
\item  $f^\beta_{\ell}$ has the same behaviour as $g_2$ on $X^k$. 
%\item 
%$\zeta(h) \neq \zeta(h')$. 
\end{itemize}
Indeed, by Lemma~\ref{lem:checkers} 
applied to $g$ and $X_0$ 
there exist $\alpha_{1,1},\dots,\alpha_{2,k} \in \Aut(\bA)$ 
such that 
\begin{itemize}
\item for all $u_1,\dots,u_k \in \{1,2\}$ the operation $g(\alpha_{u_1,1},\dots,\alpha_{u_k,k})$ is canonical on $X_0^k$ over $\bA$, and 
\item for $i \in \{1,2\}$
the operation  
$g(\alpha_{i,1},\dots,\alpha_{i,k})$  has the same behaviour as $g_i$ on $X_0^k$. 
\end{itemize}
Let $\ell \in \{1,\dots,k\}$ be the smallest index 
such that 
\begin{align*}
& \zeta \big(g(\alpha_{1,1},\dots,\alpha_{1,{\ell}-1},\alpha_{1,\ell},\alpha_{2,{\ell+1}},\dots,\alpha_{2,k} ) \big) \\\neq \; & \zeta \big (g(\alpha_{1,1},\dots,\alpha_{1,\ell- 1},\alpha_{2,\ell},\alpha_{2,{\ell+1}},\dots,\alpha_{2,k}) \big).
\end{align*}
We know that such an index exists, because 
$$\zeta \big (g(\alpha_{1,1},\dots,\alpha_{1,k}) \big) = \zeta(g_1) \neq \zeta(g_2)= \zeta \big (g(\alpha_{2,1},\dots,\alpha_{2,k}) \big).$$
Then  
$f \coloneqq g (\alpha_{1,1},\dots,\alpha_{1,\ell-1},\id,\alpha_{2,{\ell+1}},\dots,\alpha_{2,k})$, $\alpha \coloneqq \alpha_{1,\ell}$, and 
$\beta \coloneqq \alpha_{2,\ell}$ have the required properties, because $X_0$ is $k$-rich with respect to $\Pol(\bB)$.

Let $X_0 \subset X_1 \subset \cdots$
be finite subsets of $A$ such that $\bigcup_{i=0}^{\infty} X_i = A$.
Then the claim applied to $X = X_i$, for $i \in {\mathbb N}$, asserts the existence of $f_i$, $\alpha_i$, $\beta_i$, and $\ell_i$ such that
$(f_i)_{\ell_i}^{\alpha_i}$ and $(f_i)_{\ell_i}^{\beta_i}$ are canonical on $X_i^k$ and 
$\zeta((f_i)^{\alpha_i}_{\ell_i}) \neq \zeta((f_i)^{\beta_i}_{\ell_i})$. 
By thinning out the sequences 
$(f_i)_{i \in \mathbb N}$, $(\alpha_i)_{i \in{\mathbb N}}$, $(\beta_i)_{i \in{\mathbb N}}$,
and $(\ell_i)_{i \in{\mathbb N}}$ we can assume
that all $\ell_i$ are equal, say $\ell$, 
and that the complete behaviour ${\mathcal B}_1$ of $f_i^{\alpha_i}$  on $X_0$ and ${\mathcal B}_2$ of $f_i^{\beta_i}$ on $X_0$ does not depend on $i$. Note that we must have  ${\mathcal B}_1 \neq {\mathcal B}_2$. 
By Lemma~\ref{lem:simul-interpol}
(and in particular the statement at the end starting with ``moreover'') 
there exist $g \in \Pol(\bB)$ and $u,v \in \overline{\Aut(\bA)}$ such that
$g^u_\ell$ has behaviour ${\mathcal B}_1$ on all of $A$, and $g^v_\ell$ has behaviour ${\mathcal B}_2$ on all of $A$. So they are canonical over $\bA$
and $\zeta(g^u_{\ell}) \neq \zeta(g^v_{\ell})$.
Suppose that 
$\zeta(g^u_{\ell})=\pi^k_{r}$ and
$\zeta(g^v_{\ell})=\pi^k_{s}$.
%We may assume
%without loss of generality that $r \neq \ell$. 
Let
$$f(x,y) \coloneqq f(\underbrace{y,\dots,y}_{r-1},x,y,\dots,y).$$
Then note that 
\begin{align}
 \zeta \big(f(\pi^2_1,u(\pi^2_2)) \big) 
= \; & 
\zeta\big(g(\underbrace{u(\pi^2_2),\dots,u(\pi^2_2)}_{r-1},\pi^2_1,u(\pi^2_2),\dots,u(\pi^2_2)) \big) 
\label{eq:2}
\\
= \; & 
\zeta \big(g^u_{\ell}(\underbrace{\pi^2_2,\dots,\pi^2_2}_{r-1},\pi^2_1,\pi^2_2,\dots,\pi^2_2) \big)  %&& \text{(} 
\label{eq:3}
\\
= \; & \pi^k_r(\underbrace{\pi^2_2,\dots,\pi^2_2}_{r-1},\pi^2_1,\pi^2_2,\dots,\pi^2_2) 
\label{eq:4} \\
%&& \text{()} \\
%\label{eq:4}
= \; & \pi^2_1 \nonumber 
\end{align}
where equality of (\ref{eq:2}) and (\ref{eq:3}) holds since
$\zeta$ is $\overline{\Aut(\bA)}$-invariant,
and equality of (\ref{eq:3}) and (\ref{eq:4}) holds since $\zeta$ is minor-preserving. 
Similarly, since $r \neq s$
\begin{align*}
\zeta \big (f(\pi^2_1,v(\pi^2_2)) \big) 
%\zeta \big(f(\underbrace{v(y),\dots,v(y)}_{r-1},x,v(y),\dots,v(y)) \big) \\
%& = \zeta \big( f^v_{\ell}(\underbrace{y,\dots,y}_{r-1},x,y,\dots,y) && \text{($\zeta$ is $\overline{\Aut(\bA)}$-invariant)} \\ 
& = \pi^k_s(\underbrace{\pi^2_2,\dots,\pi^2_2}_{r-1},\pi^2_1,\pi^2_2,\dots,\pi^2_2) %&& \text{($\zeta$ is minor-preserving)} \\ 
 = \pi^2_2 
%&& \text{.}
\end{align*}
Therefore, $f,u,v$ show that item $(2)$ from the statement does not hold, concluding the proof that $(2) \Rightarrow (1)$.  
\end{proof}

We finally prove Theorem~\ref{thm:binary-indep},
which shows that it suffices to verify the UIP on independent elementary substructures. 

\begin{proof}[Proof of Theorem~\ref{thm:binary-indep}]
The forward implication holds trivially. 
%First suppose that $\zeta$ has the UIP with respect to $\Pol(\bC)$ over $\bA$. Note that . The statement then holds for all $f \in \Pol(\bC)^{(2)}, u_1,u_2 \in \overline{\Aut(\bA)}$ and in particular for those
%with $u_1(A) \subseteq A_1$ and $u_2(A) \subseteq A_2$. 
We now show the converse, $(2) \Rightarrow (1)$. 
It suffices to verify $(2)$ in Theorem~\ref{thm:binary}. Let $f \in \Pol(\bC)^{(2)}, u_1,u_2 \in \overline{\Aut(\bA)}$
be such that $f^{u_1}_2,f^{u_2}_2 \in \Can(\bA)$. 
We have to show that $\zeta(f^{u_1}_2) = \zeta(f^{u_2}_2)$. 
 Let $F \subseteq A$ be 2-rich with respect to $\Pol(\bC)$. 
Then it suffices to show that $f^{u_1}_2$ and $f^{u_2}_2$ realise the same complete behaviour on $F^2$. 

%Let $(\bA_1,\bA_2)$ be a pair of independent elementary substructures of $\bA$; such a 
%pair exists by Theorem~\ref{thm:indep}. 
For $j \in \{1,2\}$ let $e_j$ be an embedding of 
$\bA$ into $\bA_j$. 
Since $\bA$ is homogeneous
there exists $\epsilon \in \Aut(\bA)$ such that 
$e_1$ and $\epsilon$ agree on $u_1(F) \cup u_2(F)$. 
 By Lemma~\ref{lem:canon} applied to
$f(\id,\epsilon^{-1} e_2)$
there are $\beta_1,\beta_2 \in \Aut(\bA)$ such that $f(\beta_1,\epsilon^{-1} e_2 \beta_2)$ is canonical on $F^2$. 
Let $v_i \coloneqq \epsilon u_i$
% \in \overline{\Aut(\bA)}$
 and $w \coloneqq e_2 \beta_2$. 
 % \in \overline{\Aut(\bA)}$. 
Note that $v_i(F) \subseteq A_1$ 
and $w(F) \subseteq A_2$. 
Let $h \coloneqq f(\beta_1,\epsilon^{-1}) \in \Pol(\bC)$,
and note that 
$$h_2^{v_i} = f(\beta_1,\epsilon^{-1} \epsilon u_i) = f(\beta_1,u_i)$$ is canonical since $f_2^{u_i} = f(\id,u_i)$ is canonical,  
and that $$h_2^{w} = f(\beta_1,\epsilon^{-1} e_2 \beta_2)$$ is canonical on $F^2$ as we have seen above.  
We thus apply (2) two times 
to obtain that $\zeta(h_2^{v_1}) = \zeta(h_2^{w}) = \zeta(h_2^{v_2})$. Finally, note that 
$$\zeta(f(\id,u_i)) = \zeta(f(\beta_1,u_i))  = \zeta(f(\beta_1,\epsilon^{-1} \epsilon u_i))  = \zeta(h_2^{v_i}).$$ 
We conclude that 
$\zeta(f(\id,u_1)) = \zeta(f(\id,u_2))$.  
\end{proof}

\section{First-order expansions of the basic relations of RCC5}
\label{sect:rcc5}

RCC5 is a fundamental formalism for spatial reasoning, and may be viewed as an $\omega$-categorical structure $\bR$ (see, e.g.,~\cite{qe-journal,Qualitative-Survey}). 
There are many equivalent ways of formally
introducing RCC5; we follow the presentation in~\cite{qe-journal} and
 refer to~\cite{Qualitative-Survey} 
for other definitions and their equivalence. 
Let $\fs$ be the structure with domain $S \coloneqq 2^{\mathbb N} \setminus \{\emptyset\}$, i.e., the set of all non-empty subsets of the natural numbers ${\mathbb N}$. The signature of $\fs$ consists of the five binary relation symbols $\eq,\pp,\ppi,\dr,\op$ where for 
$x,y \subseteq \mathbb{N}$ we have
\begin{align}
(x,y)\in \eq & \text{ iff } x=y, & \text{ ``$x$ and $y$ are equal''}
\label{eq:eq} \\
(x,y)\in \pp & \text{ iff } x\subset y, & \text{ ``$x$ is strictly contained in $y$''} 
\label{eq:pp2} \\
(x,y)\in  \ppi & \text{ iff } x\supset y, & \text{ ``$x$ strictly contains $y$''}
\label{eq:ppi} \\
(x,y)\in \dr & \text{ iff } x\cap y=\emptyset, & \text{ ``$x$ and $y$ are disjoint''}
\label{eq:dr} \\
(x,y)\in \op & \text{ iff } 
x\not\subset y\wedge y\not\subset x\wedge x\cap y\neq \emptyset, 
& \text{ ``$x$ and $y$ properly overlap''.}
\label{eq:op} 
\end{align}

	Note that by definition every pair $(x,y) \in S^2$ is contained in exactly one of the relations $\dr^\fs,\op^\fs,\pp^\fs,\ppi^\fs,\eq^\fs$.
Note that the structure $\fs$ is not $\omega$-categorical; however, $\Age(\fs)$ is 
an amalgamation class (Theorem 30 in~\cite{qe-journal}), and hence there exists a countable homogeneous structure $\bR$ with the same age as $\fs$. We refer to the relations $\eq^\bR,\pp^\bR,\ppi^\bR,\dr^\bR,\op^\bR$ as the \emph{basic relations of RCC5}. 

The \emph{composition} of two binary relations $R_1$ and $R_2$ is the binary relation 
$R_1 \circ R_2 \coloneqq \big \{(x,y):\exists z \big(R_1(x,z) \wedge R_2(z,y) \big) \big \}$.
The \emph{converse} (sometimes also called \emph{inverse}) of a relation $R$ is the 
relation $\{(y,x) \colon (x,y) \in R\}$, and denoted by $R^\smile$. The converse of $\pp$ is $\ppi$,
and $\eq^\bR$, $\dr^\bR$, and $\op^\bR$ are their own converse. The full binary relation containing all
pairs of elements of $\bR$ is denoted by $\bf 1$. 
 It is straightforward to verify that 
the relations
$\eq^\bR,\pp^\bR,\ppi^\bR,\dr^\bR,\op^\bR$ 
compose as shown in Table~\ref{table:rcc}. 

\begin{remark}
Note that the respective relations of $\bS$ are \emph{not} closed under composition; e.g, the relation $\pp^{\bS} \circ \pp^{\bS}$ does not contain the pair $(\{0\},\{0,1\})$, and hence is not equal to $\pp^{\bS}$. However, 
$\pp^{\bR} \circ \pp^{\bR} = \pp^{\bR}$ by the homogeneity of $\bR$.
\end{remark} 

\begin{lemma}\label{lem:R-fin-bound}
$\bR$ is finitely bounded.
\end{lemma}
\begin{proof}
A finite structure embeds into $\bR$ if and only if all three-element substructures of the structure embed into $\bR$. For details we refer to~\cite{Qualitative-Survey}, implication $2 \Rightarrow 3$ in Proposition 15. 
\end{proof}

\begin{remark} As a \emph{relation algebra}, RCC5 is given by the composition table and the data about the converses. Our structure $\bR$ introduced above is a \emph{representation} of 
RCC5. We do not introduce relation algebras formally, because they will not be needed in the following, and rather refer to~\cite{Qualitative-Survey}. 
\end{remark}

\begin{figure*}
%\begin{table}[ht]
\begin{center}
%\caption{The composition table of RCC5}
$\begin{tabu}[t]{|c|[1pt]c|c|c|c|c|}
\hline
\circ & \dr & \op & \pp & \ppi & \eq\\[-1pt]
\tabucline[1pt]{1-6}
\dr & \bf  1 & \pp\cup \dr \cup \po & \pp\cup \dr \cup \po & \dr & \dr\\
\hline
\op & \ppi\cup \dr \cup \po & \bf 1 & \pp\cup \op & \ppi\cup \op & \op\\
\hline
\pp & \dr & \pp\cup \dr \cup \po & \pp & \bf 1 & \pp\\
\hline
\ppi & \ppi\cup \dr \cup \po & \ppi \cup \op& 
% \pp\cup \ppi\cup \op\cup \eq 
\bf 1 \setminus \dr
& \ppi &\ppi\\
\hline
\eq & \dr & \op & \pp & \ppi & \eq\\
\hline
\end{tabu}$
\end{center}
\caption{The composition table for the relations of $\bR$.}
\label{table:rcc}
%\end{table}
%\label{fig:rcc5}
\end{figure*}

A \emph{first-order expansion} of a structure $\bA$ is an expansion of $\bA$ by relations that are first-order definable over $\bA$. 

\begin{theorem}\label{thm:rcc5}
For every first-order expansion $\bC$ of 
$\bR$ either 
\begin{itemize}
\item $\Pol(\bC) \cap \Can(\bR)$ does not have a uniformly continuous minor-preserving map to $\proj$,  
%has a pseudo-Siggers operation which is canonical with respect to $\bR$, 
in which case $\Csp(\bC)$ is in P,
or 
\item $\Pol(\bC)$ has a uniformly continuous minor-preserving map to $\proj$, 
%does not have a pseudo-Siggers operation, 
in which case $\Csp(\bC)$ is NP-complete. 
\end{itemize}
\end{theorem}
We prove 
Theorem~\ref{thm:rcc5} using  
Corollary~\ref{cor:dicho}  
and hence first need to introduce a Ramsey expansion $(\bR;\prec)$ of $\bR$ (Section~\ref{sect:rcc5-ramsey}); we use a recent Ramsey transfer result of Mottet and Pinsker~\cite{MottetPinskerCores}. 
Every first-order expansion $\bC$ of $\bR$ is clearly a model-complete core. We verify the remaining assumption
in Corollary~\ref{cor:dicho} in two steps. 
Assuming that there exists a uniformly continuous minor-preserving map from $\Pol(\bC) \cap \Can(\bR)$ to $\proj$,
we construct in Section~\ref{sect:rcc5-uch1} either 
\begin{itemize}
\item a minor-preserving map $\eta \colon \Pol(\bC) \cap \Can(\bR,\prec) \to\proj$
 which arises from the action of
the canonical polymorphisms on the two relations $\pp$ and 
$(\dr \cup \po) \, \cap \prec$,
or 
\item a minor-preserving map $\rho \colon \Pol(\bC) \cap \Can(\bR,\prec) \to\proj$ which arises from the action of the canonical polymorphisms on the two relations $\drprec$ and $\poprec$. 
\end{itemize}
In the second step, we prove that if such a map 
$\eta,\rho \colon \Pol(\bC) \cap \Can(\bR,\prec) \to \proj$ exists, then it has the UIP with respect to $\Pol(\bC)$ over $(\bR,\prec)$  (Section~\ref{sect:rcc5-uip}); here we use Theorem~\ref{thm:binary-indep}. 
The statement then follows from Corollary~\ref{cor:dicho}.

\subsection{A Ramsey order expansion of $\bR$}
\label{sect:rcc5-ramsey}
The structure $\bR$ is not Ramsey, but it has
a homogeneous Ramsey expansion by a linear order. 
Let $\mathcal C$ be the class of all expansions 
of structures from $\Age(\bS)$ with the signature 
$\{\eq,\dr,\op,\pp,\ppi,\prec\}$ 
such that $\prec$ denotes a linear extension 
of $\pp$. 
\begin{proposition}\label{prop:amalg}
The class $\mathcal C$ defined above is an  amalgamation class. 
\end{proposition}

\begin{proof}
It is clear from the definition that $\mathcal C$ is closed under isomorphisms and substructures. It is well-known that in order to prove the amalgamation property, it suffices to verify the 1-point amalgamation property (see, e.g.,~\cite{Book}): for all structures $\bA,\bB_1,\bB_2 \in {\mathcal C}$ such that $A = B_1 \cap B_2$ and $B_i = A \cup \{b_i\}$, for $i \in \{1,2\}$ and $b_1 \neq b_2$, there exists $\bC \in {\mathcal C}$ with $C = B_1 \cup B_2 = A \cup \{b_1,b_2\}$ such that $\bB_1$ and $\bB_2$ are substructures of $\bC$. 
It is easy to see that such a structure $\bC$ can be determined by specifying which relations of $\bC$ contain the pair $(b_1,b_2)$. 

%Since 
%$\Age(\bS)$ has the strong amalgamation property, 
Bodirsky and Chen~\cite{qe-journal} (also see~\cite{Book}) 
%The amalgamation argument for $\Age(\bS)$ from
proved that there exists $R \in  \{\dr,\op,\pp,\ppi\}$ 
such that 
\begin{itemize}
\item if we add the pair $(b_1,b_2)$ to $R^{\bC}$
and the pair $(b_2,b_1)$ to $(R^\bC)^{\smile}$, then 
the $\{\eq,\dr,\op,\pp,\ppi\}$-reduct of the resulting
structure $\bC$ is in $\Age(\bS)$,
and 
\item $(b_1,b_2) \in \pp^{\bC}$
only if there exists $a \in A$ such that $(b_1,a) \in \pp^{\bB_1}$ and $(a,b_2) \in \pp^{\bB_2}$. 
\end{itemize}
If $R$ equals $\pp$ we add 
$(b_1,b_2)$ to $\prec^{\bC}$; 
if $R$ equals $\ppi$ we add $(b_2,b_1)$ to $\prec^{\bC}$. 
If $R$ is from $\{\dr,\op\}$, then we add 
$(b_1,b_2)$ or $(b_2,b_1)$ to $\prec^{\bC}$ according to an order-amalgam 
of the $\{\prec\}$-reducts of $\bB_1$ and $\bB_2$ over $\bA$. 
In each of these cases, $\prec^{\bC}$ is a linear order that
extends $\pp^{\bC}$, and hence $\bC$ is in ${\mathcal C}$. 
\end{proof}

One can check by an easy back-and-forth
argument that the $\{\eq,\dr,\op,\pp,\ppi\}$-reduct
of the Fra{\"i}ss\'e-limit of ${\mathcal C}$ is isomorphic to $\bR$; hence, we may denote this Fra{\"i}ss\'e-limit by 
$(\bR,\prec)$. To show that
$(\bR,\prec)$ is Ramsey,
we use a classical result from Ramsey theory about the atomless Boolean algebra, and the following recent Ramsey transfer result.

\begin{definition}[Mottet and Pinsker~\cite{MottetPinskerCores}]
Let $\mathscr G$ be a permutation group on a set $X$.
A function $g \colon X \to X$ is called \emph{range-rigid with respect to $\mathscr G$} if 
for every $k \in {\mathbb N}$, every orbit of $k$-tuples in the range of $g$ is preserved by $g$. 
%for all $\beta \in G$ we have 
%$$g \in \overline{\{ \alpha \circ g \circ \beta \circ g \mid \alpha \in {\mathscr G}\} }.$$
\end{definition}

Mottet and Pinsker~\cite{MottetPinskerCores} 
proved that if $\bB$ is a homogeneous structure and $g \colon B \to B$ is range-rigid with respect to $\Aut(\bB)$, then the age of the structure induced by
$\bB$ on the image of $g$ has the amalgamation property; we
denote the Fra\"{i}ss\'e-limit of this class by $\bB_g$. 

%\begin{theorem}[]
%Let $\bB$ be a first-order reduct of an $\omega$-categorical homogeneous Ramsey structure $\bA$ and let $\bC$ be the model-complete core of $\bB$. Then 
%\begin{itemize}
%\item 
%$\bC$ is a first-order reduct of an $\omega$-categorical homogeneous Ramsey structure $\bA'$ which embeds into $\bA$.
%\item If $\bB$ is finitely bounded, then $\bB'$ can be chosen to be finitely bounded as well. 
%\end{itemize}
%\end{theorem}

\begin{theorem}[Lemma 10 in~\cite{MottetPinskerCores}]
\label{thm:ramsey-trans}
Let $\bB$ be a homogeneous
structure and let $g \colon B \to B$ be range-rigid with respect to $\Aut(\bB)$. If $\bB$ is Ramsey,
then so is $\bB_g$. 
\end{theorem}

Let 
 $\bA = (A;\cap,\cup,\overline{\cdot},0,1)$ be the countable atomless Boolean algebra~\cite{Hodges}, 
 which we view as a subalgebra of $(2^{\mathbb N};\cap,\cup,\overline{\cdot},0,1)$. 
 We use the usual shortcuts: 
 \begin{itemize}
\item $x \subseteq y$ stands for $x \cap y = x$,
\item $x \subset y$ stands for $x \subseteq y$ and $x \neq y$, and 
\item $x + y$ stands for $(x \cap \overline{y}) \cup (y \cap \overline{x})$. 
\end{itemize} 
Let $\bA'$ be the ($\omega$-categorical) relational structure with
the same domain as $\bA$ carrying all relations that are first-order definable in $\bA$. 
Kechris, Pestov, and Todor\v{c}evi\'c described a  linear order $\prec$ on $A$
so that the expanded structure $(\bA',\prec)$ is a homogeneous $\omega$-categorical
Ramsey structure~\cite{Topo-Dynamics}. 
The age of $(\bA',\prec)$ can be described as follows. If $\bF$ is a finite substructure of
$\bA$, then $\bF$ 
is a finite Boolean algebra and there exists an enumeration $a_1,\dots,a_n$ of the atoms of $\bF$ such that for all $u,v \in F$ 
with $u = \bigcup_{i=1}^n (d_i  \cap a_i)$ and
$v = \bigcup_{i=1}^n (e_i \cap a_i)$ for 
$d,e \in \{0,1\}^n$
we have $u \prec v$ if and only if there exists some $j \in \{1,\dots,n\}$ such that $d_j < e_j$ and $d_i = e_i$ for all $i > j$; such an ordering is called an \emph{antilexicographical ordering}. 

To apply the Ramsey transfer theorem (Theorem~\ref{thm:ramsey-trans}) to the Ramsey structure $(\bA',\prec)$ to infer the Ramsey property of $(\bR,\prec)$, we need to find an appropriate function $g$ which is range-rigid 
with respect to $\Aut(\bA',\prec)$. In the course of the argument, we define a number of relations that are summarised in Figure~\ref{fig:structs}.

\begin{figure}
\begin{center}
\begin{tabular}{llp{3.4cm}l}
Structure & Domain & Signature &  Explanation \\
\hline
$\bA$ & $A$ & $\{\cap, \cup, \overline{\cdot}, 0, 1\}$ & the countable atomless Boolean algebra \\
$(\bA,\prec)$ & A & $\{\cap, \cup, \overline{\cdot}, 0, 1, \prec\}$ & a Ramsey order expansion of $\bA$ \\
$\bB$ & $B := A \setminus \{0\}$ & $\{R_{k,l} \mid k, l \geq 1\}$ $\cup$ $\{O_{k,d,e} \mid k \geq 1, e,d \in \{0,1\}^k\}$ & homogeneous structure to define the range-rigid map $g$ \\
$\bB_g$ & $B_g$ & same as $\bB$ & from Theorem~\ref{thm:ramsey-trans} \\
$(\bD,\prec)$ & $B$ &  $\{\eq,\dr,\op,\pp,\ppi,\prec\}$ & reduct of $\bB$ \\
$(\bC,\prec)$ & $B_g$ & $\{\eq,\dr,\op,\pp,\ppi,\prec\}$ & corresponding reduct of $\bB_g$ 
\end{tabular}  
\end{center}
\caption{Some structures needed to prove the Ramsey property of $(\bR,\prec)$.}
\label{fig:structs}
\end{figure}

%Clearly, $\prec$ is a linear extension of 
%$\pp^{\bA}$. 

%Observe that the structure $(2^{\mathbb N} \setminus \emptyset;\cap,\cup,c,\emptyset,{\mathbb N})$ has the same age
%as $\bA$. 

Let $\bB$ be the relational structure with domain $B \coloneqq A \setminus \{0\}$ %as the countable atomless Boolean algebra $\bA$, 
and 
for all $k,l \geq 1$ the relation
\begin{align*}
R_{k,l} \coloneqq \big \{ (a_1,\dots,a_k,b_1,\dots,b_l)
\mid  \bigcap_{i=1}^k a_i \subseteq \bigcup_{j=1}^l b_j \big \} 
\end{align*}
and for all $k \geq 1$ and $d,e \in \{0,1\}^k$ the relation 
\begin{align*}
O_{k,d,e} \coloneqq \big \{ (a_1,\dots,a_k) \mid \bigcap_{i = 1}^k (a_i + d_i)  \prec \bigcap_{i = 1}^k (a_i + e_i) \big \}. 
\end{align*}
Note that $x \prec y$ has the definition $O_{2,(0,1),(1,0)}(x,y)$ in $\bB$, which holds precisely if 
$x \cap \overline y \prec {\overline x} \cap y$. 

%To see this, let $c \in B$ be such that $c \subseteq a + b$. TODO HERE. 

\begin{lemma}\label{lem:b-hom}
The structure $\bB$ is homogeneous. 
\end{lemma}
\begin{proof}
Let $i$ be an isomorphism between finite induced substructures of $\bB$. 
We claim that $i$ preserves all quantifier-free formulas of $\bA$. To see this, it suffices to show that 
for every Boolean term $f$ of arity $n$ and every Boolean term $g$ of arity $k$,
the function $i$ preserves the formula 
$f(x_1,\dots,x_n) = g(y_1,\dots,y_k)$. 
Since $f(x_1,\dots,x_n) = g(y_1,\dots,y_k)$ if and only if 
$f(x_1,\dots,x_n) + g(y_1,\dots,y_k) = 0$, 
it suffices to show the statement for the special case $f(x_1,\dots,x_n) = 0$. 
Every Boolean term can be rewritten
as a finite union of finite intersections of arguments and complements of arguments. 
%we may even assume that $f$ is of the form 
%$x_{i_1} \cap \cdots \cap x_{i_k} \cap \overline{x_{j_1}} \cap \cdots \cap \overline{x_{j_l}}$. Note that $f = 0$ can then be written as 
Therefore, $f(x_1,\dots,x_n) = 0$ is equivalent to a conjunction of formulas of the form 
$x_{i_1} \cap \cdots \cap x_{i_r} \subseteq x_{j_1}\cup \cdots \cup x_{j_s}$.
These formulas are preserved by $i$, because $i$ preserves $R_{r,s}$.

We claim that $i$ also preserves all quantifier-free formulas of $(\bA,\prec)$. The definition of $\prec$ implies that it suffices to show that $i$ preserves $\prec$ for the atoms of the Boolean algebra generated by the domain $\{a_1,\dots,a_m\}$ of $i$. Let $b$ and $c$ be two such atoms with $b \prec c$. Then $b = \bigcap_{i = 1}^m (a_i + d_i)$ and $c = \bigcap_{i = 1}^m (a_i + e_i)$ for some $d,e \in \{0,1\}^m$,
and hence $i(b) \prec i(c)$, because $i$ preserves the relation $O_{m,d,e}$.  
By the homogeneity of $(\bA,\prec)$, 
the map $i$ has an extension to an automorphism of $(\bA,\prec)$. The restriction of this map to $B$ is an automorphism of $\bB$, concluding the proof. 
\end{proof}

\begin{lemma}
The structure $\bB$ is Ramsey. 
\end{lemma}
\begin{proof}
The structure $(\bA',\prec)$ and its homogeneous substructure $\bA''$ with domain $B$ have topologically isomorphic 
automorphism groups; since the Ramsey property only depends on the topological automorphism group (see, e.g.,~\cite{Topo-Dynamics}), it follows that $\bA''$ is Ramsey. 
%We claim that 
It is easy to see that $\bA''$ and $\bB$ have the same automorphism group, even as permutation groups, which implies the statement. 
%It suffices to show that $\bA''$ and $(\bB,\prec)$ are first-order interdefinable. Each relation of $\bB$ is a relation of $\bA''$; conversely, 
%and $(A,\subset,\prec)$ have the same automorphism group, considered as a permutation group. 
%are first-order interdefinable. 
%have the same automorphism group  as a permutation group. 
%The structure $(\bA',\prec)$
%and the structure $\bB$ have the same automorphism group as a topological group. 
%Since 
\end{proof}

Consider 
the $\{\eq,\dr,\op,\pp,\ppi\}$-structure with domain $A$ 
whose relations
are defined over $\bA$ by the expressions in (\ref{eq:eq}),
(\ref{eq:pp2}), (\ref{eq:ppi}), (\ref{eq:dr}), (\ref{eq:op}); let $\bD$ be the substructure of this structure with domain $B$,
and let $(\bD,\prec)$ be the expansion 
of $\bD$ by the restriction of the order $\prec$ to $B$.

% Lemma 0.5 in Berti's draft. 
\begin{lemma}\label{lem:range-rigid}
There exists a self-embedding 
%$e \colon (\bD,\prec) \hookrightarrow (\bD,\prec)$ 
$g$ of $(\bD,\prec)$ such that 
%such that 
\begin{enumerate}
\item for all $k,l \in {\mathbb N}$ and $a_1,\dots,a_k,b_1,\dots,b_l \in B$ we have that 
$$(g(a_1),\dots,g(a_k),g(b_1),\dots,g(b_l)) \in R_{k,l}$$
if and only if  $(a_i,b_j) \in \pp \cup \eq$ or $(a_i,a_j) \in \dr$ for some
$i,j \leq \max(k,l)$, and 
\item for every $k \in {\mathbb N}$ and 
$d,e \in \{0,1\}^k$
there exists a quantifier-free formula 
$\phi_{k,d,e}$ in the signature 
of $(\bD,\prec)$ 
that defines $O_{k,d,e}$. 
\end{enumerate}
The function $g$ is range-rigid 
with respect to $\Aut(\bB)$. 
\end{lemma}

\begin{proof}
First note that any function $g$ that satisfies the first statement of the lemma must be range-rigid.
 By the homogeneity of $\bB$, orbits of $k$-tuples in $\bB$ can be described by conjunctions of relations of the form $R_{k,l}$ and of the form $O_{k,d,e}$. 
Therefore, every orbit of $k$-tuples in the range of $g$ has by items (1) and (2) of the statement a quantifier-free definition over the relations of $(\bD,\prec)$. 
Since $g$ preserves these relations, 
$g$ preserves all orbits of $k$-tuples in its range. 
 
 To prove the first part of the statement, 
 by a standard compactness argument it suffices
to prove that for every finite substructure
$\bF$ of $\bB$ there exists a finite 
substructure $(\bG,\prec)$ of $(\bA,\prec)$ and a map $f \colon F \to G \setminus \{0\}$ that preserves the relations of $(\bD,\prec)$ 
and satisfies the property of $g$ formulated in 
items (1) and (2) in the statement. Let 
\begin{align*}
I \coloneqq \big \{ & \{a_1,\dots,a_k\}  \mid \; k \in {\mathbb N},  a_1,\dots,a_k \in F,
% \text{ such that  }
%\\ & 
(a_i,a_j) \in \po^{\bF} 
%a_i \cap a_j \neq 0 
\text{ for all } 
1 \leq i < j  \leq k\} \big \} 
\end{align*}
and let $f \colon F \to 2^I$ be the map given by
$$f(a) \coloneqq \{X \in I \mid \exists a' \in X. \, a' \subseteq a \}.$$
Define the relation $\sqsubset$ on $I$ so that 
$X \sqsubset Y$ holds for $X,Y \in I$ if and only if the smallest element of $X + Y = (X \cup Y) \setminus (X \cap Y)$ with respect to $\prec$ is contained in $X$. 

{\bf Claim 1.} $\sqsubset$ defines a linear order on $I$. It is clear from the definition of $\sqsubset$ that for all $X,Y \in I$ exactly one of $X = Y$, $X \sqsubset Y$, and $Y \sqsubset X$ holds. To show that $\sqsubset$ is transitive it suffices to show that $\sqsubset$ does not contain a directed 3-cycle. 
Let $X,Y,Z \in I$ be pairwise distinct. Let $a$ be the smallest element in $(X \cup Y \cup Z) \setminus (X \cap Y \cap Z)$ and 
%We may assume without loss of generality 
suppose that $a \in X \setminus Y$, so that $X \sqsubset Y$.
\begin{itemize}
\item If $a \in Z$ then $a$ is the smallest element in $Z+Y$, and thus $Z \sqsubset Y$.
\item If $a \notin Z$ then $a$ is the smallest element in $X+Z$, and thus $X \sqsubset Z$.
\end{itemize} 
In either case the restriction of $\sqsubset$ to $X,Y,Z$ does not form a directed 3-cycle. By symmetry, the assumption that $x \in X \setminus Y$ can be made without loss of generality; this proves the claim. 

Let $\bG$ be the Boolean algebra on $2^I$, 
and let $\prec$ denote the linear order on $G$ 
such that 
$S \prec T$ if the largest element of the symmetric difference $S + T$ with respect to the linear order $\sqsubset$ is contained in $T$. Note that $\prec$ is an antilexicographic linear order on $G$ and hence $(\bG;\prec)$ embeds into $(\bA;\prec)$. 
Also note that the image of $f$ lies in $G \setminus \{ 0 \}$. 

{\bf Claim 2.} Every $X \in I$ equals the largest element of
$I$ with respect to $\sqsubset$ which is contained in $\bigcap_{a \in X} f(a)$. 
By construction, $X \in f(a)$ for all $a \in X$, that is,  $X \in \bigcap_{a \in X} f(a)$. Assume that there exists $Y \in \bigcap_{a \in X} f(a)$. We have to show that $Y=X$ or $Y \sqsubset X$. If $X \subseteq Y$ then we are done. Otherwise, let $a$ be the smallest element of $X \setminus Y$ with respect to $\prec$. Since $Y \in f(a)$ there exists some $a' \in Y$ such that $a' \subseteq a$ by the definition of $f(a)$. Since $a \notin Y$, it follows that $a' \neq a$. So $a' \subset a$ and thus $a$ and $a'$ cannot both be in $X$. Since $a \in X$ we conclude that $a' \notin X$. 
Hence, $a' \in Y \setminus X$. Since $\prec$ extends $\subset$ it follows that $a' \prec a$. 
Since $a$ was chosen to be minimal in $X \setminus Y$ with respect to $\prec$ this implies that $Y \sqsubset X$. 

To prove that $f$ preserves the relations of $\bD$, it suffices to show that $f$ preserves $\pp$, $\dr$, and $\po$. 
\begin{itemize}
\item Suppose that $a,b \in F$ satisfy $\pp(a,b)$ in $\bF$. 
Let us assume that $X \in f(a)$ for some $X \in I$. 
By definition we know that $a' \in X$ for some $a' \subseteq a$. Then we also have $a' \subseteq b$, and thus by definition $X \in f(b)$. This concludes the proof that $f(a) \subseteq f(b)$. 
On the other hand $\{b\} \in f(b) \setminus f(a)$, thus in fact $\pp(f(a),f(b))$ holds in $\bG$. 
\item 
Suppose that $a,b \in F$ satisfy $\dr(a,b)$ in $\bF$.  
If $X \in f(a)$ then by definition there exists $a' \in X$ such that $a' \subseteq a$. Since $X \in I$ and $a' \cap b = 0$ this implies that 
$b' \notin X$ for every $b' \subseteq b$. Therefore, $X \notin f(b)$.  This implies that $\dr(f(a) \cap f(b))$ holds in $\bG$. 
\item Suppose that $a,b \in F$ satisfy $\po(a,b)$ in $\bF$. Then $\{a\} \in f(a) \setminus f(b)$, 
$\{b\} \in f(b) \setminus f(a)$, and $\{a,b\} \in \gamma(a) \cap \gamma(b)$. Therefore, $\po(f(a),f(b))$ holds in $\bG$. 
\end{itemize}
We now show that $f$ preserves $\prec$. 
Let $a,b \in F$ be such that $a \prec b$. 
Then by definition $\{a\} \sqsubset \{b\}$. 
Claim 2 implies that the largest elements contained in $f(a)$ and $f(b)$ are $\{a\}$ and $\{b\}$, respectively. It follows that $f(a) \prec f(b)$. 

To prove that $f$ satisfies items (1) and (2) of the statement, the following notation will be convenient.
For $X \subseteq F$, we write $M(X)$ for the minimal elements in $S$ with respect to $\subset$, i.e., 
$$M(X) \coloneqq \{a \in X \mid \neg \exists b \in X. \, b \subset a\}.$$
%Note that if $M(X) \in I$ if and only if the elements of $X$ have pairwise nonempty intersection.  
%Also 
Note that for every $a \in X$ there exists an $a' \in M(X)$ such that $a' \subseteq a$, 
because $X\subseteq F$ is finite.

{\bf Claim 3.} Let $a_1,\dots,a_k,b_1,\dots,b_l \in F$. Then the following are equivalent. 
\begin{enumerate}
\item $(f(a_1),\dots,f(a_k),f(b_1),\dots,f(b_l)) \notin R_{k,l}$;
\item $c \coloneqq \bigcap_{i = 1}^k f(a_i) \cap \bigcap_{i = 1}^l \overline{f(b_i)} \neq \emptyset$;
\item 
$a_i \cap a_j \neq 0$ and $a_i \not \subseteq b_j$ for all $1 \leq i < j \leq \max(k,l)$. 
\item The $\sqsubset$-maximal element of $c = \bigcap_{i = 1}^k f(a_i) \cap \bigcap_{i = 1}^k \overline{f(b_i)}$ is $M(\{a_1,\dots,a_k\})$. 
\end{enumerate}
(1) and (2) are equivalent by definition.  \\
%{\bf Claim 3.} Let $a_1,\dots,a_k \in F$ and $X \coloneqq M(\{a_1,\dots,a_k\})$. Then the following are equivalent. 
%\begin{enumerate}
%\item $X \in I$;
%\item $X \in \bigcap_{i = 1}^k f(a_i)$;
%\item $\bigcap_{i = 1}^k f(a_i) \neq \emptyset$; 
%\item $\neg \dr(a_i,a_j)$ for all $1 \leq i < j \leq k$. 
%\end{enumerate}
%1) implies 2): if $X \in I$
%then Claim 2 implies $X \in \bigcap_{i = 1}^k f(a_i)$. \\
(2) implies (3): if there are  $i,j \in \{1,\dots,k\}$ such that 
$a_i \cap a_j = \emptyset$, then $f(a_i) \cap f(a_j) = \emptyset$ since $f$ preserves $\dr$, and hence $c = \emptyset$. 
If there are $i \leq k$ and $j \leq l$ such that $a_i \subseteq b_j$ then $a \cap \overline{b_j} = \emptyset$ and hence $c = \emptyset$. \\
(3) implies (4): Let $X \coloneqq M(\{a_1,\dots,a_k\})$. Then by assumption $X \in I$, and by Claim 2 we know that $X$ is the maximal element of $\bigcap_{i=1}^k f(a_i)$ with respect to $\sqsubset$. Suppose for contradiction that $X \in f(b_j)$ for some $j \in \{1,\dots,l\}$. The definition of $f$ implies
that there exists $b_j' \subseteq b_j$ such that $b_j' \in X$. In other words, $b_j' = a_i$ for some $i \in \{1,\dots,k\}$, and therefore we have $a_i \subseteq b_j$, a contradiction. \\
(4) implies (2): trivial. 

\medskip 
Note that the equivalence of (1) and (3) in Claim 3
immediately implies that $f$ satisfies item (1) of the statement.  
To prove that $f$ satisfies item (2), let $a_1,\dots,a_k \in F$ and $d,e \in \{0,1\}^k$. 
Let $X \coloneqq M(\{a_i \mid d_i = 0\})$ and $Y \coloneqq M(\{a_i \mid e_i = 0\})$. 
We will prove that
$(f(a_1),\dots,f(a_k)) \in O_{k,d,e}$ if and only if 
$\bigcap_{i=1}^k (f(a_i) + e_i) \neq 0$ and either 
\begin{enumerate}
\item $\bigcap_{i=1}^k (f(a_i) + d_i) = 0$, or 
\item $\bigcap_{i=1}^k (f(a_i) + d_i) \neq 0$ and $X \sqsubset Y$. 
\end{enumerate}
If $\bigcap_{i=1}^k (f(a_i) + e_i) = 0$ then clearly 
\begin{align}
\bigcap_{i = 1}^k (f(a_i) + d_i)  \prec \bigcap_{i = 1}^k (f(a_i) + e_i) \label{eq:order}
\end{align} is false and hence $(f(a_1),\dots,f(a_k)) \notin O_{k,d,e}$ and we are done. 
So suppose that $\bigcap_{i=1}^k (f(a_i) + e_i) \neq 0$. 
If $\bigcap_{i=1}^k (f(a_i) + d_i) = 0$
then $(\ref{eq:order})$ holds. In this case $(f(a_1),\dots,f(a_k)) \in O_{k,d,e}$ and we are again done. So suppose that $\bigcap_{i=1}^k (f(a_i) + d_i) \neq 0$.

Then by Claim 3 we know that $X,Y \in I$ and
 $X$ and $Y$ are the $\sqsubset$-maximal elements in 
$\bigcap_{i=1}^k (f(a_i) + d_i)$ 
and $\bigcap_{i=1}^k (f(a_i) + e_i)$, respectively. 
Thus, if $X \sqsubset Y$ then $\bigcap_{i=1}^k (f(a_i) + d_i) \prec \bigcap_{i=1}^k (f(a_i) + e_i)$ and so $(f(a_1),\dots,f(a_k)) \in O_{k,d,e}$ and we are done. 
If $Y \sqsubset X$ then $\bigcap_{i=1}^k (f(a_i) + e_i) \prec \bigcap_{i=1}^k (f(a_i) + d_i)$ and hence $(f(a_1),\dots,f(a_k)) \notin O_{k,d,e}$ and we are also done. 
If $X = Y$, let $i \in \{1,\dots,k\}$. If $d_i = 0$ then by the definition of $X$ there exists an $a_j \in X$ 
such that $a_j \subseteq a_i$. Then $a_j \in Y$ and thus $e_j = 0$. Since $\bigcap_{l=1}^k (f(a_l) + e_l) \neq 0$ it follows that in particular $f(a_j) \cap (f(a_i) + e_i) = (\gamma(a_j)+e_j) \cap (f(a_i) + e_i) \neq \emptyset$. Since $f$ preserves 
$\dr$ this implies that $a_j \cap (a_i \cap e_i) \neq 0$. Since $a_j \subseteq a_i$ this is only possible if $e_i=0$. Similarly one can show that $e_i = 0$ implies $d_i = 0$. Therefore, $d=e$, and thus $O_{k,d,e} = \emptyset$. In particular, $(f(a_1),\dots,f(a_k)) \notin O_{k,d,e}$. 

Claim 3 implies that $\bigcap_{i=1}^k (f(a_j) + e_i) \neq 0$ and $\bigcap_{i=1}^k (f(a_i) + d_i) = 0$ can be defined by quantifier-free $\{\pp,\eq,\dr\}$-formulas. Note that condition (2) above can be expressed (assuming 
for simplicity of notation that 
$X = \{a_i \mid d_i = 0\}$ and $Y = \{a_i \mid e_i = 0\}$; the general case is similar) by the quantifier-free $\{\prec\}$-formula 
$$ \bigvee_{i: d_i = 0, e_i = 1} \bigwedge_{j: d_j \neq e_j} a_i \prec a_j $$
which concludes the proof. 
\end{proof}

\begin{proposition}\label{prop:ages}
%The structure $(\bR,\prec)$ is the model-complete core of $(\bD,\prec)$. 
%The structure $(\bD,\prec)$ has the same age as $(\bR,\prec)$. 
$\Age(\bD,\prec) = \Age(\bR,\prec)$. 
\end{proposition}

\begin{proof}
%Since $(\bR,\prec)$ is homogeneous, it is 
%model-complete. It is easy to see that $(\bR,\prec)$ is a core, because the negation of every relation of $\bR$ can be defined existentially positively as a union of the other relations of $\bR$, and the complement of $\prec$ has the positive quantifier-free definition $x=y \vee y \prec x$. 
%We claim that $(\bD,\prec)$ and 
%$(\bR,\prec)$ have the same age.
Note that $\bD$ has the same age as $\bS$ and hence the same age as $\bR$.
Also, $\prec$ is a linear extension of $\pp^{\bD}$,
and hence every finite substructure 
of $(\bD,\prec)$ is also a substructure
of $(\bR,\prec)$. 

Conversely, let $(\bF,\prec)$ be a finite substructure
of $(\bR,\prec)$. 
Let $g$ be an embedding of $\bF$ into $\bD$. 
Let $u_1,\dots,u_k$ be an enumeration of $F$ such that $u_1 \prec \cdots \prec u_k$. Let 
$v_1,\dots,v_k \in B$ be such that $(v_i,g(u_j)) \in \dr^{\bD}$ for all $i,j \in \{1,\dots,k\}$ and such that $(v_i,v_j) \in \dr^{\bD}$ for $i \neq j$. 
%$v_1,\dots,v_k \in {\mathbb N} \setminus \bigcup_{i \in \{1,\dots,k\}} u_i$ be distinct. 
Then we define $f \colon F \to B$ by
$$b(u) \coloneqq g(u) \cup \bigcup_{(u_i,u) \in \pp^{\bF}} v_i.$$
Note that $v_1,\dots,v_k$ are atoms in the
Boolean algebra generated by the elements 
$g(F) \cup \{v_1,\dots,v_n\}$ in $\bA$, and let $w_1,\dots,w_{\ell}$ be the other atoms. 
We may assume that $\prec$ is defined
on this Boolean algebra according to the enumeration $w_1,\dots,w_{\ell},v_1,\dots,v_k$ of the atoms. 
We prove that $f$ is an embedding of 
$(\bF,\prec)$ into $(\bD,\prec)$. 
Let $u,u' \in F$. 
If $(u,u') \in \pp^{\bF}$, then $g(u) \subset g(v)$
and $\{i \colon  (u_i,u) \in \pp^{\bF} \} \subset 
\{i \colon  (u_i,u') \in \pp^{\bF} \}$ by the transitivity of $\pp^{\bF}$, so $(b(u),b(u')) \in \pp^{\bD}$. 
It is also clear that $b$ preserves $\eq$, $\ppi$, $\dr$. To see that $b$ preserves $\op$,
note that if $(u,u') \in \po^{\bF}$, then $(g(u),g(u')) \in \po^{\bD}$, so 
$g(u) \cap g(u')$, $g(u) \cap \overline{g(u')}$,
and $\overline{g(u)} \cap g(u')$ are non-empty. 
Note that $g(u) \cap g(u') \subseteq b(u) \cap b(u')$, $g(u) \cap \overline{g(u')} \subseteq b(u) \cap \overline{b(u')}$, and $\overline{g(u)} \cap g(u') \subseteq \overline{b(u)} \cap b(u')$, 
so $(b(u),b(u')) \in \po^{\bD}$. 

To prove that $b$ preserves $\prec$, 
let $i,j \in \{1,\dots,k\}$ be such that $i < j$ and $u_i \prec u_j$. Then $\pp^{\bD}(v_j,b(u_j))$, 
$\dr^{\bD}(v_j,b(u_i))$, and for every $m \in \{j+1,\dots,k\}$
we have $\dr^{\bD}(v_m,b(u_i))$ 
and $\dr^{\bD}(v_m,b(u_j))$. Indeed, 
if $\dr^{\bD}(v_m,b(u_i))$ does not hold, 
then $\pp^{\bD}(v_m,b(u_i))$ and hence
$\pp^{\bF}(u_m,u_i)$ by the definition of $b$. 
This in turn implies that $u_m \prec u_i$ 
and hence $m<i$, a contradiction. 
By the definition of $\prec$ on $A$ 
%with
%respect to the enumeration $w_1,\dots,w_{\ell},v_1,\dots,v_k$ 
this implies that $b(u_i) \prec b(u_j)$ and finishes the proof of the claim. 
\end{proof}

Every relation of $(\bD,\prec)$ has a quantifier-free definition in the homogeneous structure $\bB$; let $(\bC,\prec)$ be the structure with the signature $\{\eq,\dr,\op,\pp,\ppi,\prec\}$ 
defined over $\bB_g$ via these quantifier-free definitions.

\begin{proposition}\label{prop:c-homo}
The structures $(\bC,\prec)$ and $\bB_g$ are quantifier-free interdefinable; in particular, 
$(\bC,\prec)$ is homogeneous. 
\end{proposition}
\begin{proof}
By definition, every relation of  $(\bC,\prec)$ has a quantifier-free definition in $\bB_g$. Conversely, let $R$ be a relation
of $\bB_g$. If $R$ is of the form $R_{k,l}$ then item (1) of Lemma~\ref{lem:range-rigid} implies that $R$
has a quantifier-free definition over $(\bC,\prec)$.
If $R$ is of the form $O_{k,d,e}$ then item (2) of  Lemma~\ref{lem:range-rigid} implies that $R$
has a quantifier-free definition over $(\bC,\prec)$.
\end{proof}

\begin{corollary}\label{cor:c-and-r}
$(\bC,\prec)$ and $(\bR,\prec)$ are isomorphic. 
\end{corollary}
\begin{proof}
By Proposition~\ref{prop:c-homo}, 
it suffices to prove that the homogeneous structures $(\bC,\prec)$ and $(\bR,\prec)$ have
the same age. 
By Proposition~\ref{prop:ages}, 
the age of $(\bR,\prec)$ equals the age
of $(\bD,\prec)$, which has the same age as
$(\bC,\prec)$, because $g$ is an embedding 
from $(\bD,\prec)$ to $(\bC,\prec)$. 
\end{proof}

\begin{theorem}\label{thm:ramsey-exp}
$(\bR,\prec)$ is Ramsey. 
\end{theorem}
\begin{proof}
We apply the Ramsey transfer (Theorem~\ref{thm:ramsey-trans})
to the homogeneous Ramsey structure $\bB$ from Lemma~\ref{lem:b-hom} and the map $g$ from Lemma~\ref{lem:range-rigid} which is
range-rigid with respect to $\Aut(\bB)$. 
Theorem~\ref{thm:ramsey-trans} states that
the homogeneous structure $\bB_g$ is Ramsey.  The structure $\bB_g$ and 
$(\bC,\prec)$ have the same automorphism group (Proposition~\ref{prop:c-homo}), and hence $(\bC,\prec)$ is Ramsey as well. This implies the statement, because $(\bR,\prec)$ is isomorphic to 
$(\bC,\prec)$ by Corollary 9.12.
\end{proof}

\subsection{Polymorphisms of $\bR$ that are canonical with respect to $(\bR,\prec)$}
From now on, we identify the symbols
$\eq,\pp,\ppi,\dr,\po$ with the respective relations of $\bR$, and we write $\succ$ for the converse of $\prec$. Note that $\pp \cup \dr \cup \po$ and $\ppi \cup \dr \cup \po$ are primitively positively definable in $\bR$, since they are entries in Table~\ref{table:rcc}. Hence, their intersection $\dr \cup \po$ is primitively positively definable in $\bR$, too. We also write $\bot$ instead of $\dr \cup \po$, and $\precnsim$ for the relation $\bot \, \cap \prec$. 

The composition table for the binary relations with a first-order definition over $(\bR;\prec)$ can be derived from the composition table of $\bR$ (Table~\ref{table:rcc}) using the following lemma. 

\begin{lemma}\label{lem:decomp}
Let $R_1,R_2 \in \{\eq,\pp,\ppi,\dr,\op\}$ 
and let $O_1,O_2$ be two orbits of pairs 
of $(\bR;\prec)$ such that $O_i \subseteq R_i$. Then $$O_1 \circ O_2 = \begin{cases}
(R_1 \circ R_2) \, \cap \prec & \text{ if } O_1,O_2 \subseteq \; \prec \\
(R_1 \circ R_2) \, \cap \succ & \text{ if } O_1,O_2 \subseteq \; \succ \\
R_1 \circ R_2 & \text{ otherwise. }
\end{cases}$$
\end{lemma}

\begin{proof}
It is clear that $O_1 \circ O_2 \subseteq R_1 \circ R_2$ and that if $O_1,O_2 \subseteq \; \prec$ then $$O_1 \circ O_2 \subseteq \, \prec \circ \prec \, = \, \prec.$$ 
So the $\subseteq$-containment in the statement of the lemma holds in the first case, and by similar reasoning also in the other two cases. 
The reverse containment
in the first two cases is also clear since the homogeneity of $(\bR,\prec)$ implies that the expression in the statement on the right describes an orbit of 
pairs in $\Aut(\bR,\prec)$ (we have already seen that it is non-empty). 
To show the equality in the third case,
let $(x,z) \in R_1 \circ R_2$. 
If $R_1 \circ R_2$ equals 
$\pp$ (or $\ppi$) then $x \prec z$ (or $z \succ x$), and again $R_1 \circ R_2$ is an orbit of pairs in $\Aut(\bR,\prec)$ and the equality holds. If $R_1 \circ R_2$ equals $\eq$ then
the statement is clear, too. 
Otherwise, let $x',y',z' \in R$ be such that $(x',y') \in O_1$ 
and $(y',z') \in O_2$. If $(x',z')$ lies in the same orbit as $(x,z)$ in $\Aut(\bR,\prec)$
then $(x,z) \in O_1 \circ O_2$. Otherwise, 
consider the structure induced by $(\bR,\prec)$
on $\{x',y',z'\}$; we claim that if we replace the tuple $(x',z')$ in $\prec$ by the tuple $(z',x')$,
the resulting structure still embeds into $(\bR;\prec)$. 
By assumption, $O_1 \subseteq \, \prec$ and $O_2 \subseteq \, \succ$, or  $O_1 \subseteq \, \succ$ and $O_2 \subseteq \, \prec$, so the modified relation $\prec$ is still acyclic. Moreover, 
since $R_1 \circ R_2 \in \{\dr,\po\}$, 
the modified relation $\prec$ is still a linear extension of
$\pp$ and hence in $\Age(\bR;\prec)$. 
The homogeneity of $(\bR;\prec)$ implies
that $(x',z')$ lies in the same orbit as $(x,z)$,
and hence $(x,z) \in O_1 \circ O_2$. 
This concludes the proof that $R_1 \circ R_2 \subseteq O_1 \circ O_2$. 
\end{proof}

\begin{corollary}\label{cor:square}
Let $O \subseteq \; \prec$ be an orbit of pairs in  
$\Aut(\bR;\prec)$. Then $\pp \subseteq O \circ O$.
\end{corollary}
\begin{proof}
If $O = \pp$, then $O \circ O = O = \pp$. 
%If $O = \ppi$ or $O = \eq$ the statement is clear, too. 
If $O = (\dr \; \cap \prec)$ then 
\begin{align*}
O \circ O & = (\dr \circ \dr) \, \cap \prec && \text{(Lemma~\ref{lem:decomp})} \\
& = {\bf 1} \, \cap \prec  && \text{(table in Figure~\ref{table:rcc})} \\
& = \; \prec \, .
\end{align*}
Similarly we may compute that $\pp$ is contained in $(\poprec) \circ (\poprec)$. 
%$O \circ O$ for all other orbits of pairs in $(\bR;\prec)$. 
%if $O =$. 
\end{proof}

The following general observations are useful when working with operations that preserve  binary relations over some set $B$. If $f \colon B^k \to B$ is an operation that preserves $R_1,\dots,R_k,R'_1,\dots,R_k' \subseteq B^2$ then 
%$f$ preserves $R_1,\dots,R_k$
%and $R_1',\dots,R_k'$, then
\begin{align}
& f(R_1 \circ R_1', \dots, R_k \circ R'_k) \nonumber \\
\subseteq \; & f(R_1,\dots,R_k) \circ f(R'_1,\dots,R'_k). \label{eq:comp-pres}
\end{align}
%\end{lemma}
Also note that 
\begin{align}
f(R_1^\smile,\dots,R_k^\smile) = f(R_1,\dots,R_k)^\smile.
\label{eq:smile-pres}
\end{align}

\begin{lemma}\label{lem:pres-prec}
Let $f \in \Pol(\bR) \cap \Can(\bR,\prec)$.
Then $f$ preserves $\prec$
and $\precnsim$. 
\end{lemma}
\begin{proof}
Let $a_1,\dots,a_k,b_1,\dots,b_k$ be elements of $\bR$ such that $a_i \prec b_i$ for all $i \in \{1,\dots,k\}$ and $ f(b_1,\dots,b_k) \preceq f(a_1,\dots,a_k)$. %Let $R_i \in \{\eq,\pp,\ppi,\dr,\po\}$ be such that $(a_i,b_i) \in R_i$. 
For $i \in \{1,\dots,k\}$, let $O_i$ be the orbit of $(a_i,b_i)$ in $\Aut(\bR,\prec)$. 
%By , we have $\pp \in O_i \circ O_i$, and 
We then have 
\begin{align*}
f(\pp,\dots,\pp) & \subseteq f(O_1 \circ O_1,\dots,O_k \circ O_k) && \text{(Corollary~\ref{cor:square})} \\
& \subseteq f(O_1,\dots,O_k) \circ f(O_1,\dots,O_k) && \text{(\ref{eq:comp-pres})} \\
& \subseteq (\succeq \circ \succeq) && \text{(canonicity)} \\
& = \; \succeq
\end{align*}
which is a contradiction to $f$ being a polymorphism of $\bR$. 
Since $\nsim$ is primitively positively definable in
$\bR$, it also follows that every polymorphism
of $\bR$ which is canonical
with respect to $(\bR,\prec)$ preserves $\precnsim$. 
\end{proof}

\begin{lemma}\label{lem:pres-neq}
Let $f \in \Pol(\bR)$ be canonical over $\bR$ or over $(\bR,\prec)$. 
Then $f$ preserves $1 \setminus \eq$. 
% $\pp \cup \ppi \cup \dr \cup \po$.
\end{lemma} 
\begin{proof}
Let $k$ be the arity of $f$. First suppose that $f$ is canonical over $(\bR,\prec)$. 
Let $O_1,\dots,O_k$ be orbits of pairs of $\Aut(\bR;\prec)$ that are distinct from $\eq$. Suppose for contradiction that $f(O_1,\dots,O_k) = \eq$. Then 
\begin{align*}
f(O_1 \circ O_1,\dots,O_k \circ O_k) & \subseteq f(O_1,\dots,O_k) \circ \cdots \circ f(O_1,\dots,O_k) \\ & \subseteq \eq \circ \cdots \circ \eq = \eq. 
\end{align*}
Since $O_i \circ O_i$ contains $\pp$ or $\ppi$ for every $i \in \{1,\dots,k\}$ by Corollary~\ref{cor:square}, it follows that $f(P_1,\dots,P_k) \subseteq \eq$ for some $P_1,\dots,P_k \in \{\pp,\ppi\}$. 
Note that $f(P_1^\smile,\dots,P_k^\smile)= \eq^\smile = \eq$ by~\eqref{eq:smile-pres}.
Since $\pp \subseteq \pp \circ \ppi$ and 
$\pp \subseteq \ppi \circ \pp$ we obtain that
$f(\pp,\pp) \subseteq \eq \circ \eq = \eq$, a contradiction to $f(\pp,\pp) \subseteq \pp$. 
We conclude that $f$ preserves 
%$\pp \cup \ppi \cup \dr \cup \po$. 
$1 \setminus \eq$. 
If $f$ is canonical over $\bR$ instead of $(\bR,\prec)$, then the statement can be shown similarly, using the table in Figure~\ref{table:rcc}  
 instead of Corollary~\ref{cor:square}.
\end{proof}

\subsection{Independent substructures}
To verify the UIP property, we will use 
Theorem~\ref{thm:binary-indep}
and therefore need certain pairs $(\bA_1,\bA_2)$ of independent elementary substructures of $(\bR,\prec)$. 

\begin{lemma}\label{lem:rcc5indep}
There are elementary substructures $\bA_1$ 
and $\bA_2$ of $(\bR,\prec)$ such that 
for all $a_1 \in A_1$ and $a_2 \in A_2$ we have
$(a_1,a_2) \in \dr$ and $a_1 \prec a_2$; in particular, $\bA_1$ and $\bA_2$ are independent. 
\end{lemma}
\begin{proof}
By the homogeneity of $(\bR,\prec)$ it suffices to show that for every structure $(\bB,\prec) \in \Age(\bR,\prec)$ there are embeddings $e_1,e_2 \colon \bB \to (\bR,\prec)$ such that for all $b_1,b_2 \in B$ we have $(e_1(a_1),e_2(a_2)) \in \dr$ and $e_1(b_1) \prec e_2(b_2)$. 
Choose an embedding $f$ 
of $\bB$ into the structure $\bS$ from the definition of $\bR$; so $f(b) \subseteq {\mathbb N}$ for each $b \in B$. Then $f_1 \colon b \mapsto \{2n \colon n \in f(b)\}$ and $f_2 \colon b \mapsto \{2n+1 \colon n \in f(b)\}$
are two embeddings of $\bB$ into $\bS$ 
such that $(f_1(b_1),f_2(b_2)) \in \dr$ for all $b_1,b_2 \in B$. Let $\bB'$ be the substructure
of $\bS$ with domain $B' \coloneqq f_1(B) \cup f_2(B)$. 
For all $b_1,b_2 \in B$ with $b_1 \prec b_2$,
define the linear order $\prec$ on $B'$ by 
%\begin{itemize}\item 
$f_1(b_1) \prec f_1(b_2)$,
%\item 
$f_2(b_1) \prec f_2(b_2)$,
%\item 
$f_1(b_1) \prec f_2(b_2)$. 
%\end{itemize}
Then it is straightforward to check that $\prec$ extends $\pp$, and hence $(\bB',\prec) \in \Age(\bR,\prec)$, which concludes the proof.  
\end{proof}

Independent substructures are used in the proof
of the following lemma that plays an important role when verifying the UIP later. 

\begin{lemma}\label{lem:copy}
Let $\bC$ be a first-order expansion of $\bR$ 
and let 
$\zeta \colon \Pol(\bC) \cap {\Can(\bR,\prec)} \to \proj$ be a clone homomorphism which does
not have the UIP with respect to $\Pol(\bC)$. 
Then for every finite 2-rich subset $F$ of the domain of $\bR$ 
there exist $f \in \Pol(\bC)^{(2)}$ and 
$\alpha_1,\alpha_2 \in \Aut(\bR;\prec)$ 
such that 
\begin{itemize}
\item %$f(\id,v_1)$ and $f(\id,v_2)$ are
$f$ is canonical on $F \times \alpha_1(F)$
and on $F \times \alpha_2(F)$ with respect to $(\bR,\prec)$, 
\item $\zeta(f(\id,\alpha_1)|_{F^2}) \neq \zeta(f(\id,\alpha_2)|_{F^2})$ (recall Definition~\ref{def:rich-behave}), 
\item for all $a,b \in F$, if $(a,b) \in \pp \cup \eq$, then 
$(\alpha_1(a),\alpha_2(b)) \in \pp$, and 
\item  for all $a,b \in F$, if $(a,b) \in \; \nsim$
then $(\alpha_1(a),\alpha_2(b)) \in \; \nsim$. 
\end{itemize}
\end{lemma}
\begin{proof}
Let $\bA_1,\bA_2$ be the two independent 
elementary substructures of $(\bR,\prec)$ from 
Lemma~\ref{lem:rcc5indep}. 
Since $\zeta$ does not have the UIP with respect to
$\Pol(\bC)$, by Theorem~\ref{thm:binary-indep}
there exists $f' \in \Pol(\bC)^{(2)}$ and 
$u_1,u_2 \in \overline{\Aut(\bR,\prec)}$ such that for $i \in \{1,2\}$ the image $\Im(u_i)$ of $u_i$ is contained in $A_i$,  the operation $f'(\id,u_i)$ is canonical with respect to $(\bR,\prec)$, and 
$\zeta(f'(\id,u_1)) \neq \zeta(f'(\id,u_2))$. 
By Lemma~\ref{lem:canon2} there exists a finite 
subset $X$ of the domain of $\bR$ such that for all $g \in \Pol(\bR)^{(2)}$ there exist $\beta_1,\beta_2 \in \Aut(\bR,\prec)$ with $\beta_1(F),\beta_2(F) \subseteq X$ 
such that $g$ is canonical on $\beta_1(F) \times \beta_2(F)$. 
For $i \in \{1,2\}$, 
let $\epsilon_i \in \Aut(\bR,\prec)$ 
be such that $\epsilon_i(X) 
\subseteq \Im(u_i)$. 
%Recall that $\bR$ is the model-complete core of
%the structure $\bD$ from Section~\ref{sect:rcc5-ramsey}, and hence can be viewed as a substructure of $\bD$, which is itself a reduct of the atomless Boolean algebra $\bA$. 
We may view the substructure of $\bR$ induced
by $\epsilon_1(X) \cup \epsilon_2(X)$ as a substructure of $\bS$. 
%, which has an embedding 
%$e$ into $\bR$; moreover, 
By the homogeneity of $\bR$ we may also assume that
the substructure of $\bS$ on $\epsilon_1(X) \cup \epsilon_2(X) \cup \{\epsilon_1(x) \cup \epsilon_2(x) \colon x \in X\}$ has an embedding 
$e$ into $\bR$ such that $e(\epsilon_i(x))=\epsilon_i(x)$ for all $x \in X$
and $i \in \{1,2\}$. 
Define 
$$\epsilon \colon X \to \bR \text{ by } \epsilon(x) \coloneqq e \big (\epsilon_1(x) \cup \epsilon_2(x) \big).$$ 
Let $a,b \in X$. Note that
\begin{enumerate}
\item
 \label{eq:pp}
  if $(a,b) \in \pp \cup \eq$ then 
\begin{align*} 
\epsilon_i(a) = e(\epsilon_i(a)) & \subset e \big (\epsilon_1(a) \cup \epsilon_2(a) \big) \\
& \subseteq e \big (\epsilon_1(b) \cup \epsilon_2(b) \big) = \epsilon(b)
\end{align*}
and hence $(\epsilon_i(a),\epsilon(b)) \in \pp$;  
\item \label{eq:nsim} 
if $(a,b) \in \; \nsim$ then 
$(\epsilon_i(a),\epsilon_i(b)) \in \; \nsim$.
Since $\epsilon_i(a) \in A_i$ it follows that
$\epsilon_i(a)$ is disjoint from $\epsilon_j(b)$ for $j \neq i$. This implies  that $$\epsilon_i(a) \setminus \epsilon(b) = \epsilon_i(a) \setminus  \big (\epsilon_1(b) \cup \epsilon_2(b) \big) = \epsilon_i(a) \setminus \epsilon_i(b) \neq \emptyset$$
and that for $j \neq i$ we have
$$\epsilon_j(b) = \epsilon_j(b) \setminus \epsilon_i(a) \subseteq \epsilon(b) \setminus \epsilon_i(a).$$
%$\epsilon_i(a) \cap \big (\epsilon_1(b) \cup  \epsilon_2(b) \big) = \emptyset$ and hence
%\begin{align}
%\big (\epsilon_i(a),\epsilon(b) \big) = \big (e(\epsilon_i(a)),e(\epsilon_1(b) \cup  \epsilon_2(b))\big ) \in \; \dr
%\label{eq:nsim}
%\end{align} 
%since
%$e$ is an embedding. 
In particular, $\epsilon(b) \setminus \epsilon_i(a) \neq \emptyset$. We obtained that the sets $\epsilon_i(a) \setminus \epsilon(b)$ and
$\epsilon(b) \setminus \epsilon_i(a)$ are both non-empty. 
Therefore $(\epsilon_i(a),\epsilon(b)) \in \; \nsim$. 
\end{enumerate}
We define an order $\prec'$ on 
$\epsilon_1(X) \cup \epsilon_2(X) \cup \epsilon(X)$ by setting $a \prec' b$ if one of the following holds.
\begin{itemize}
\item $a,b \in \epsilon_1(X) \cup \epsilon_2(X)$ and $a \prec b$;
\item $a \in \epsilon_1(X) \cup \epsilon_2(X)$
and $b \in \epsilon(X)$; 
%$b \in \{\epsilon_1(x) \cup \epsilon_2(x) \colon x \in X\}$. 
\item $a,b \in \epsilon(X)$ and $\epsilon^{-1}(a) \prec \epsilon^{-1}(b)$. 
%$a,b \in \{\epsilon_1(x) \cup \epsilon_2(x) \colon x \in X\}$ and $\epsilon(
\end{itemize}
Then it is easy to see that $\prec'$ defines a partial order that extends $\pp$; let $\prec''$ be a linear order that extends $\prec'$. 
By the definition of $(\bR;\prec)$ there exists
an automorphism $\gamma$ of $\bR$ 
that maps $\prec''$ to $\prec$. 
This shows that we may assume that $\epsilon$ preserves $\prec$ (otherwise, replace $\epsilon$ by $\gamma \circ \epsilon$). 

%there exists an embedding $\gamma$ 
%of the substructure of $\bS$ with domain
%$\{ \gamma_1(a) \cup \gamma_2(a) \colon a \in F\}$ into $\bR$. 
By the definition of $X$ there are $\beta_1,\beta_2 \in \Aut(\bR,\prec)$ such that $\beta_1(F) \subseteq X$, $\beta_2(F) \subseteq X$, and 
$f'(\id,\epsilon)$ is canonical
on $\beta_1(F) \times \beta_2(F)$ over $(\bR,\prec)$. 
Since $$\epsilon_i \beta_2(F) \subseteq \epsilon_i(X) \subseteq \Im(u_i)$$
for $i \in \{1,2\}$ 
we have that 
%$f$ is canonical on $\beta_1(C) \times \epsilon_i \beta_2(C)$, for $i \in \{1,2\}$, and
$\zeta(f'|_{\beta_1(F) \times \epsilon_1 \beta_2(F)}) \neq \zeta(f'|_{\beta_1(F) \times \epsilon_2 \beta_2(F)})$. 
Then for some $i \in \{1,2\}$ we
have that $\zeta(f'|_{\beta_1(F) \times \epsilon_i \beta_2(F)}) \neq \zeta(f'|_{\beta_1(F) \times \epsilon \beta_2 (F)})$. 
We claim that $f \coloneqq f'(\beta_1,\id)$, 
$\alpha_1 \coloneqq \epsilon_i \circ \beta_2$ 
and $\alpha_2 \coloneqq \epsilon \circ \beta_2$ satisfy the conclusion of the lemma: 
\begin{itemize}
\item 
$f$ is canonical on $F \times \alpha_i(F)$, for $i \in \{1,2\}$, because $f'(\id,\epsilon)$ is canonical on $\beta_1(F) \times \beta_2(F)$. 
\item 
for all $(a,b) \in F^2 \cap (\pp \cup \eq)$ we have
that $(\beta_2(a),\beta_2(b)) \in X^2 \cap (\pp \cup \eq)$ and by (\ref{eq:pp}) we obtain 
$(\alpha_1(a),\alpha_2(b)) = (\epsilon_i(\beta_2(a)),\epsilon(\beta_2(b))) \in \pp$. 
\item for all $(a,b) \in F^2 \cap \nsim$ we have 
that $(\beta_2(a),\beta_2(b)) \in X^2 \cap \bot$ and by (\ref{eq:nsim}) we obtain 
\begin{align*}
(\alpha_1(a),\alpha_2(b)) & = (\epsilon_i(\beta_2(a)),\epsilon(\beta_2(b))) \in \; \nsim.\qedhere
\end{align*}  
\end{itemize}
\end{proof}

%\subsection{Polymorphisms of $(\bR,\prec)$}
%Note that if $R_1,R_2$ are preserved by $f$,
%then $f$ also preserves $R_1 \circ R_2$.
%If $f \colon R^k \to R$ is canonical over $(\bR,\prec)$ and $O_1,\dots,O_k$ are orbits of pairs in  $\Aut(\bR,\prec)$ then we write $f(O_1,\dots,O_k)$ for the orbit for
%$f(a_1,\dots,a_k)$ for some (equivalently, any) $a_1 \in O_1,\dots,a_k \in O_k$. Note that then $f(O_1 \circ O_1',\dots,O_k \circ O_k') = f(O_1,\dots,O_k) \circ f(O_1',\dots,O_k')$. 

\subsection{Uniformly continuous minor-preserving maps}
\label{sect:rcc5-uch1}
%In this section we show that if the hardness condition from Section~\ref{sect:eta} does not apply, and the hardness condition from 
%Section~\ref{sect:rho} does not apply,
%then there is a canonical polymorphism that
%implies polynomial-time tractability. 
Let $\bB$ be a first-order expansion of $\bR$. 
%and let $\mathscr C \coloneqq \Pol(\bB) \cap \Can(\bR,\prec)$. 
In this section we show that if there exists 
a uniformly continuous minor-preserving map 
from $\Pol(\bB) \cap \Can(\bR)$ to $\proj$,
then a specific minor-preserving map $\eta$ (introduced in Section~\ref{sect:eta}) 
or $\rho$ (introduced in Section~\ref{sect:rho})  is a uniformly continuous minor-preserving map from $\Pol(\bB) \cap \Can(\bR,\prec) \to \proj$. 
For this task, we need some concepts from finite universal algebra that will be recalled in the following. 

%Recall that $\precnsim \; = (\dr \cup \po) \, \cap \, {\prec}$ is preserved by every polymorphism of $\bR$ (Lemma~\ref{lem:pres-prec}). 
%In the following, $0$ stands for $\drprec$
%and $1$ stands for $\poprec$.
%\begin{proposition}\label{prop:cyclic}
%Let ${\mathscr C} \subseteq \Pol(\bR) \cap \Can(\bR,\prec)$ be a clone and 
%let $\rho \colon {\mathscr C} \to {\mathscr O}_{\{0,1\}}$ be given by 
%$$\rho(f) \coloneqq \xi(f)|_{\{0,1\}}.$$

\subsubsection{Subfactors}
\label{sect:subfactors}
%An operation $s \colon A^6 \to A$ is called a 
%\emph{Siggers operation}~\cite{Siggers} if it satisfies
%$$s(x,x,y,y,z,z) = s(y,z,x,z,x,y)$$
%for all $x,y,z \in A$. 
%\end{itemize}
%For $H_1,\dots,H_k \subseteq A$ define 
%$$f(H_1,\dots,H_k)  \coloneqq \{f(a_1,\dots,a_k) \colon a_1 \in H_1,\dots,a_k \in H_k\}.$$
A \emph{subfactor} of a clone ${\mathscr C}$ on $A$ is a clone ${\mathscr C'}$ obtained from ${\mathscr C}$ as follows: pick a subset $S \subseteq A$ which is preserved by all operations of ${\mathscr C}$ and an equivalence relation $E$ on $S$ that is preserved by all operations of ${\mathscr C}$ to $E$;  
the domain of ${\mathscr C}'$ are the equivalence classes of $E$, and the operations of ${\mathscr C}'$ are the actions of the operations of ${\mathscr C}$ on these classes.
We say that the subfactor ${\mathscr C}'$ is \emph{trivial} if
it is isomorphic to $\proj$. A frequent case will be that the subfactor is a clone on two elements; note that $\Pol(\bB)$ has a trivial subfactor if and only if there are two non-empty disjoint $H_1,H_2 \subseteq B$ such that
for every $f \in \Pol(\bB)^{(k)}$ there exists $j \in \{1,\dots,k\}$ such that 
for all $i_1,\dots,i_k \in \{1,2\}$
$$f(H_{i_1},\dots,H_{i_k}) \subseteq H_{i_j}.$$
An operation $f \colon A^k \to A$ is 
 \emph{idempotent} if $f(x,\dots,x) = x$ for every $x \in A$. 
If $E$ is an equivalence relation on a set $S$, and $u \in S$ is an element, then the equivalence class of $u$ with respect to $E$ is denoted by $u/E$.

%The following combines results from~\cite{wonderland,BulatovJeavons}. %Siggers,Cyclic,
%\begin{theorem}\label{thm:fin-dichotomy}
%Let $\bB$ be a structure with a finite domain
%such that all polymorphisms of $\bB$ are idempotent. 
%Then the following are equivalent. 
%\begin{itemize}
%\item $\bB$ has no Siggers polymorphism.
%\item $\bB$ has no cyclic polymorphism. 
%\item $\Pol(\bB)$ has a minor-preserving map to $\proj$. 
%\item $\Pol(\bB)$ has a trivial subfactor. 
%\end{itemize}
%\end{theorem}

\subsubsection{Factoring $\prec$} 
\label{sect:eta}
In the following, $0$ stands for $\{\drprec, \poprec\}$ and $1$ stands for $\{\pp\}$. 
Note that $\prec \; = (\drprec) \cup (\poprec) \cup \pp$. Let $F$ be the equivalence relation on 
$\{\drprec, \poprec, \pp\}$ with the equivalence classes $0$ and $1$. 
Throughout this section, $\xi$ denotes $\xi^{(\bR,\prec)}_2$. 

\begin{lemma}\label{lem:F}
For every $f \in \Pol(\bR) \cap \Can(\bR,\prec)$, 
the operation $\xi(f)$ preserves $F$.
\end{lemma}
\begin{proof}
Let $f \in \Pol(\bR)^{(k)}$ be canonical with respect to $(\bR,\prec)$ and suppose that 
$O_1,\dots,O_k$ are orbits of pairs contained in $\prec$ such that $f(O_1,\dots,O_k) \subseteq \pp$. By Lemma~\ref{lem:pres-prec} it suffices to show that then
$f(O_1',\dots,O_k') \subseteq \pp$ for all orbits of pairs 
$O_i'$ such that $(O_i',O_i) \in F$. 
%that are equivalent to $O_i$. 
We claim that $O_i' \subseteq O_i \circ O_i$. 
If $O_i = \pp$ then $O_i' = \pp$ and the statement is clear since $\pp \circ \pp = \pp$. 
If $O_i = \po$ then $O_i \circ O_i = \; \prec$ (Lemma~\ref{lem:decomp}) which contains 
$\drprec$ and $\poprec$. 
Similarly, if $O_i = \dr$ then $O_i \circ O_i = \; \prec$, and again the claim follows. 
By (\ref{eq:comp-pres}),  we have
\begin{align*}
f(O_1',\dots,O_k') & \subseteq f(O_1 \circ O_1,\dots,O_k \circ O_k) \\
& \subseteq f(O_1,\dots,O_k) \circ f(O_1,\dots,O_k) = \pp. \qedhere
\end{align*}
\end{proof}

\begin{definition}
Let ${\mathscr C} := \Pol(\bR) \cap \Can(\bR,\prec)$.
Then $\eta \colon {\mathscr C} \to {\mathscr O}_{\{0,1\}}$ 
is the map given by 
$$\eta(f) \coloneqq (\xi(f)|_{\{\drprec, \poprec, \pp\}})/F.$$
\end{definition}

An Boolean operation $f \colon \{0,1\}^3 \to \{0,1\}$ is called 
\begin{itemize}
\item the ternary \emph{minority} operation if for all $x,y \in \{0,1\}$ we have $f(x,x,y) = f(x,y,x) = f(y,x,x) = y$,
\item the ternary \emph{majority} operation if 
for all $x,y  \in \{0,1\}$ we have $f(x,x,y) = f(x,y,x) = f(y,x,x) = x$. 
\end{itemize}
Note that the given conditions specify $f$ uniquely.

\begin{proposition}\label{prop:wedge}
Let ${\mathscr C} \subseteq \Pol(\bR) \cap \Can(\bR,\prec)$ be a clone.
% and 
%let $\eta \colon {\mathscr C} \to {\mathscr O}_{\{0,1\}}$ be given by 
%$$\eta(f) \coloneqq (\xi(f)|_{\{\drprec, \poprec, \pp\}})/F.$$
Then $\eta({\mathscr C})$ contains neither the binary maximum, nor the ternary majority, nor the ternary minority operation. Every operation in 
$\eta({\mathscr C})$
is a projection, or $\mathscr C$ contains an operation $f$ such that $\eta(f) = \wedge$ where $\wedge$ denotes the binary minimum operation. 
\end{proposition}
\begin{proof}
First observe that $\eta({\mathscr C})$ is 
an idempotent clone because every $f \in {\mathscr C}$ preserves $\pp$, and preserves 
$\precnsim = (\drprec) \, \cup \, (\poprec)$ by Lemma~\ref{lem:pres-prec}. 
The well-known classification of clones over a two-element set by Post~\cite{Post} implies that if 
$\eta({\mathscr C})$ contains an operation that is not a projection, then it must contain the binary minimum, the binary maximum, the ternary majority, or the ternary minority    operation.
%Moreover, every clone on a Boolean domain that contains a Maltsev operation also contains a minority operation. 
So it suffices to prove that the last three cases are impossible. Note that in these three cases 
${\mathscr C}$ contains a ternary operation $g$ such that $\eta(g)$ is cyclic and $\eta(g)(0,0,1) = 1$ or $g(0,1,1) = 1$. 
Hence,  
\begin{align*}
 g(\precnsim,\precnsim,\pp), g(\precnsim,\pp,\precnsim), g(\pp,\precnsim,\precnsim) & \subseteq \pp \\
\text{ or } \quad g(\precnsim,\pp,\pp), g(\pp,\precnsim,\pp), g(\pp,\pp,\precnsim) & \subseteq  \pp \, .
\end{align*}
Note that
\begin{align*}
\precnsim \circ \precnsim \circ \, \pp \, = \, \precnsim \circ \, \pp \, \circ \precnsim \, = \pp \,  \circ \precnsim \circ  \precnsim  & = \, \prec \\
\text{ and } 
\precnsim \circ  \, \pp  \circ \pp = \pp  \, \circ \precnsim \, \circ \, \pp = \pp \circ  \pp  \, \circ  \precnsim & = \, \prec
\end{align*}
and we obtain $g(\prec,\prec,\prec) \subseteq \pp$ from applying $(\ref{eq:comp-pres})$ twice. 
In particular, we have $g(\dr \; \cap \prec,\dr \; \cap \prec,\dr \; \cap \prec) \subseteq \pp$, which contradicts the fact that $g$ preserves $\precnsim$.  
\end{proof}

If $\mathscr C$ contains an operation $f$ such that $\eta(f) = \wedge$, then ${\mathscr C}$ also contains such an operation which is not only
canonical with respect to $(\bR,\prec)$, but also
with respect to $\bR$. To prove this, we first show
the following lemma.

\begin{lemma}\label{lem:partial-can}
Let $f \in \Pol(\bR)^{(k)}$ be canonical with respect to $(\bR,\prec)$. Let $O_1,\dots,O_k$ be orbits of pairs of distinct elements in $\Aut(\bR)$ such that $f(O_1,\dots,O_k) \subseteq \; \nsim$. Then there exists $O \in \{\dr,\po\}$ such that $f(O_1,\dots,O_k) \subseteq O$. 
\end{lemma}
\begin{proof}
If $O_i \in \{\dr,\po\}$, then let $O_i' \coloneqq (O_i \, \cap \prec)$; otherwise, let $O_i' \coloneqq O_i$. Then $O_1',\dots,O_i'$ are orbits of pairs in $\Aut(\bR,\prec)$.
Let $O \coloneqq \xi(f)(O_1',\dots,O_k')$. By assumption, $O \in \{\dr \; \cap \prec, \po \; \cap \prec\}$. 
We have to show that for all orbits of pairs
$O_1'',\dots,O_k''$ in $\Aut(\bR,\prec)$
such that $O_i'' \subseteq O_i$ for all $i \in \{1,\dots,k\}$ we have that 
$\xi(f)(O_1'',\dots,O_k'') \in \{O,O^\smile\}$. 
Note that $\pp$ is contained in $O'_i \circ (O_i'')^{\smile}$, and (\ref{eq:comp-pres}) implies that $\pp = \xi(f)(\pp,\dots,\pp)$ is contained in $\xi(f)(O_1',\dots,O_k') \circ \xi(f)(O_1'',\dots,O_k'')$ and in $\xi(f)(O_1'',\dots,O_k'') \circ \xi(f)(O_1',\dots,O_k')$. On the other hand, we know that $\pp$ is not contained in $\po \circ \dr$, which implies that if $O \subseteq \po$ then $O^\smile \subseteq \po$ and if $O \subseteq \dr$ then $O^{\smile} \subseteq \dr$. 
\end{proof}

\begin{lemma}\label{lem:eta-wedge-inj}
Let $f \in \Pol(\bR)^{(2)} \cap \Can(\bR,\prec)$ 
be such that $\eta(f) = \wedge$. Then 
$f$ is injective. 
\end{lemma}

\begin{proof}
Let $O_1$ and $O_2$ be orbits of pairs of 
$\Aut(\bR,\prec)$. If both $O_1$ and $O_2$ are distinct from $\eq$, then by Proposition~\ref{lem:pres-neq} we know that $f(O_1,O_2) \neq \eq$. 
Now suppose that one of $O_1$ and $O_2$, say $O_1$, equals $\eq$. If $O_2 = \eq$, there is nothing to be shown; otherwise, either $\pp$ or $\ppi$ is contained in $O_2 \circ O_2$ by Corollary~\ref{cor:square}. In the first case 
we have 
$$f(\eq,\pp) = f(\eq \circ \eq, O_2 \circ O_2) \subseteq \eq \circ \eq = \eq$$
and in the second case, analogously, 
$f(\eq,\ppi) \subseteq \eq$. Note that 
$f(\eq,\ppi) \subseteq \eq$ and 
$f(\eq,\pp) \subseteq \eq$ are equivalent by (\ref{eq:smile-pres}). 
So we have that $f(\eq,\pp) \subseteq \eq$.
Also note that $f(\pp,\bot) \subseteq \, \nsim$
because $\eta(f) = \wedge$. 
Then 
\begin{align*}
f(\pp,\pp) & \subseteq f(\eq \circ \pp, \pp \, \circ \nsim) \\
& \subseteq \eq \, \circ \nsim \; = \; \nsim \, ,
\end{align*}
which contradicts the assumption that 
$f$ preserves $\pp$. 
\end{proof}

\begin{lemma}\label{lem:nsim}
Let $f \in \Pol(\bR)^{(2)} \cap \Can(\bR,\prec)$ be such that $\eta(f) = \wedge$. 
If $O_1,O_2 \in \{\pp,\ppi,\dr,\po\}$ are distinct then $f(O_1,O_2) \subseteq \; \nsim$. 
\end{lemma}
\begin{proof}
First observe that 
if $O_1,O_2 \subseteq \; \prec$ then $\xi(f)(O_1,O_2) \in \{\drprec, \poprec \}$ since $\eta(f) = \wedge$. If $O_1,O_2 \subseteq \; \succ$ then we apply the same argument to
$O_1^\smile$ instead of $O_1$ 
and $O_2^{\smile}$ instead of $O_2$.  
Otherwise, by the injectivity of $f$ (Lemma~\ref{lem:eta-wedge-inj}) we know that $\xi(f)(O_1,O_2) \neq \eq$. 
Assume for contradiction that $\xi(f)(O_1,O_2) = \pp$. For $i \in \{1,2\}$, let $O_i' \coloneqq O_i$ if $O_i \subseteq \; \prec$
and otherwise let $O_i' \coloneqq \po \; \cap \prec$. Then one can check using the composition table (Table~\ref{table:rcc}) and Lemma~\ref{lem:decomp} that 
$O_i' \subseteq O_i \circ \pp$. Therefore, 
\begin{align*}
f(O_1',O_2') &\subseteq f(O_1 \circ \pp, O_2 \circ \pp) \\
& \subseteq  f(O_1,O_2) \circ f(\pp,\pp) = \pp \circ \pp = \pp.
\end{align*}
On the other hand, $O_i \subseteq \; \succ$ for some $i \in \{1,2\}$, and hence $O_i' = \poprec$. Since $\eta(f) = \wedge$ we have $f(O_1',O_2') \subseteq \, \nsim$, a contradiction. 
\end{proof}

\begin{lemma}\label{lem:eta-can}
Let ${\mathscr C} \subseteq \Pol(\bR) \cap \Can(\bR,\prec)$ be a clone that contains
a binary operation $f$ such that $\eta(f) = \wedge$. Then $\mathscr C$ contains
an operation $g$ such that $\eta(g) = \wedge$
and $g \in \Can(\bR)$. 
\end{lemma}

\begin{proof}
We claim that the operation
$$g(x,y) \coloneqq f \big (f(x,y),f(y,x) \big)$$
satisfies the requirements. It is clear that $\eta(g) = \wedge$. To show that $g$ is canonical with respect to $\bR$, let $O_1,O_2$ be two orbits of pairs in $\Aut(\bR)$. We have to show that $g(O_1,O_2)$ is contained in one orbit of pairs in $\Aut(\bR,\prec)$. 
If $O_1 = O_2$ then $g(O_1,O_2) \subseteq O_1$, and we are done, so suppose that $O_1 \neq O_2$. 

\medskip 
{\bf Case 1.} $O_1 \neq \eq$ and $O_2 \neq \eq$. Then $f(O_1,O_2) \subseteq \; \nsim$
 by Lemma~\ref{lem:nsim}.
 %Since $\nsim$ is primitive positive definable in $\bR$, the operation $g$ must preserve $\nsim$
%and it follows that 
%$g(O_1,O_2) \subseteq \; \nsim$.  
Hence, the statement follows from Lemma~\ref{lem:partial-can}. 

%\medskip 
%{\bf Claim.} There exists $R \in \{\pp,\ppi,\nsim,\eq\}$ such that $g(O_1,O_2),g(O_2,O_1) \subseteq R$. 

{\bf Case 2.} 
$O_1 = \eq$ and 
$O_2 \in \{\pp,\ppi\}$. Then the statement follows
directly from the assumption that $f$ is canonical with respect to $(\bR,\prec)$. 

{\bf Case 3.} 
$O_1 = \eq$ and 
$O_2 \in \{\po,\dr\}$. Define $O_2' \coloneqq (O_2 \, {\cap} \, {\prec})$. Note that $O_2 = O_2' \cup (O_2')^\smile$. 
%Let $P \coloneqq \xi(f)(\eq,O_2')$. 
Since $f$ is injective (Lemma~\ref{lem:eta-wedge-inj}), we have z
%{\bf Claim 3.}
%$f(\eq,O_2') \subseteq \nsim$.
$\xi(f)(\eq,O_2') \in \{(\poprec),(\drprec), (\po \, {\cap} \, {\succ}), (\dr \, {\cap} \, {\succ})\}$. 
Using observation (\ref{eq:smile-pres}) we get that $$f(\eq,O_2) 
%= \xi(f)(\eq,O_2') \cup f(\eq, (O_2')^\smile) 
%\subseteq
= f(\eq,O_2') \cup f(\eq,O_2')^\smile \in \{\dr,\po\}.$$ 

{\bf Case 4.} $O_1 \neq \eq$ and $O_2 = \eq$.
The statement can be shown analogously to the cases above. This concludes the proof. 
\end{proof}

\subsubsection{Restricting to $\precnsim$} 
\label{sect:rho}
Recall that $\precnsim \; = (\dr \cup \po) \, \cap \, {\prec}$ is preserved by every polymorphism of $\bR$ (Lemma~\ref{lem:pres-prec}). 
In the following, $0$ stands for $\drprec$
and $1$ stands for $\poprec$.

\begin{definition}
Let ${\mathscr C} := \Pol(\bR) \cap \Can(\bR,\prec)$. 
Let $\rho \colon {\mathscr C} \to {\mathscr O}_{\{0,1\}}$ be the map given by 
$\rho(f) \coloneqq \xi(f)|_{\{0,1\}}$. 
\end{definition}

\begin{proposition}\label{prop:cyclic}
Let ${\mathscr C} \subseteq \Pol(\bR) \cap \Can(\bR,\prec)$ be a clone.
Then either 
every operation in $\rho({\mathscr C})$
is a projection, or $\rho({\mathscr C})$
contains a ternary cyclic operation. 
\end{proposition}
\begin{proof}
%The operation $\rho$ is well-defined since
The statement follows from Post's result~\cite{Post}, because each of the clones generated by min, max, majority, and minority  contains a ternary cyclic operation.   
\end{proof} 

\subsubsection{Canonical pseudo-cyclic polymorphisms}
In this section we prove the goal statement from the beginning of Section~\ref{sect:rcc5-uch1} (Corollary~\ref{cor:cyclic}). 

\begin{theorem}\label{thm:cyclic}
%Suppose that $\Pol(\bB) \cap \Can(\bR,\prec)$ 
Let ${\mathscr C} \subseteq \Pol(\bR) \cap \Can(\bR,\prec)$ be a clone that 
contains 
\begin{itemize}
\item a binary operation $g$ such that $\eta(g) = \wedge$, and  %(see Proposition~\ref{prop:wedge}), and
\item a ternary operation $f$ such that
$\rho(f)$ is cyclic. 
% (see Proposition~\ref{prop:cyclic}). 
\end{itemize}
Then $\xi \big({\mathscr C} \cap \Can(\bR) \big)$ contains a ternary cyclic operation. 
\end{theorem}

\begin{proof}
Let \begin{align*}
h(x,y,z) & \coloneqq g(g(x,y),z) \\
h'(x,y,z) & \coloneqq h \big (h(x,y,z),h(y,z,x),h(z,x,y) \big) \\ 
h''(x,y,z) & \coloneqq f \big (h'(x,y,z),h'(y,z,x),h'(z,x,y) \big)
\end{align*}
We verify that $h''$ is canonical with respect to
$\Pol(\bR)$ and that $\xi(h'')$ is cyclic. 
In order to show this it is enough to show that for all orbits  of pairs $O_1,O_2,O_3$ in $\Aut(\bR)$ 
there exists an orbit of pairs $O$ in $\Aut(\bR)$ such that
$$h''(O_1,O_2,O_3),h''(O_2,O_3,O_1),h''(O_3,O_1,O_2) \subseteq O.$$
Let $O_1,O_2,O_3$ be orbits of pairs in
$\Aut(\bR)$. Note that $h'' \in \Pol(\bR)$; hence, 
if $O_1 = O_2 = O_3$ then 
$$h''(O_1,O_2,O_3), h''(O_2,O_3,O_1), h''(O_3,O_1,O_2) \subseteq O_1$$ and we are done in this case. Suppose that this is not the case, so that in particular $O_i \neq \eq$ for some $i \in \{1,2,3\}$. 
Since $g$ is injective (Lemma~\ref{lem:eta-wedge-inj}), so is $h$.
By Proposition~\ref{lem:eta-can}, we
have that $g,h \in \Can(\bR)$ and 
\begin{align*}
O_1' & \coloneqq \xi(h)(O_1,O_2,O_3) \neq \eq, \\
O_2' & \coloneqq \xi(h)(O_2,O_3,O_1) \neq \eq, \\
O_3' & \coloneqq \xi(h)(O_3,O_1,O_2) \neq \eq.  
\end{align*}
We claim that $O_1'' \coloneqq h'(O_1,O_2,O_3) \subseteq h(O_1',O_2',O_3') \subseteq \; \nsim$. 
If $O'_1 \neq O'_2$ then 
$g(O'_1,O'_2) \subseteq \, \nsim$ by Lemma~\ref{lem:nsim}.
If $O'_1,O'_2 \subseteq \, \nsim$
then $g(O'_1,O'_2) \subseteq \, \nsim$
since $g \in \Pol(\bR)$ preserves $\nsim$. 
In both cases, 
$g(g(O'_1,O'_2),O'_3) \subseteq  g(\nsim,O'_3) \subseteq \, \nsim$. 
Otherwise, $O'_1 = O'_2 \in \{\pp,\ppi\}$
and $O'_3 \neq O'_1$ 
 in which case $$g(g(O'_1,O'_2),O'_3) \subseteq g(O'_1,O'_3) \subseteq \, \nsim$$ which proves the claim. 
 
Similarly, $O_2'' \coloneqq h'(O'_2,O'_3,O'_1) \subseteq \; \nsim$ and $O_3'' \coloneqq h'(O'_3,O'_1,O'_2) \subseteq \, \nsim$. Since $f \in \Pol(\bR)$ preserves $\nsim$ we have that $f(O_1'',O_2'',O_3'') \subseteq \; \nsim$. Lemma~\ref{lem:partial-can} implies
that  there exists $O'''_1 \in \{\dr,\po\}$ such that
$f(O_1'',O_2'',O_3'') \subseteq O_1'''$. 
Similarly, there are $O'''_2,O'''_3 \in \{\dr,\po\}$ 
such that $f(O_2'',O_3'',O_1'') \subseteq O'''_2$
and $f(O_3'',O_1'',O_2'') \subseteq O'''_3$. 
For $i \in \{1,2,3\}$, let $O^*_i \coloneqq O_i'' \, \cap {\prec}$. 
Then 
$$f(O^*_1,O^*_2,O^*_3),f(O^*_2,O^*_3,O^*_1),f(O^*_3,O^*_1,O^*_2)  \subseteq \; \prec$$ 
by Lemma~\ref{lem:pres-prec}. 
Since $f$ is cyclic on $\{\drprec,\poprec\}$ it follows that $f(O_2^*,O_3^*,O_1^*) \subseteq (O'''_1 \cap \prec)$, which
implies that $O'''_2 = O_1'''$. Similarly, 
we obtain $O_3''' = O_2'''$. Since $h''(O_1,O_2,O_3) \subseteq O_1'''$, $h''(O_2,O_3,O_1) \subseteq O_2'''$, and 
$h''(O_3,O_1,O_2) \subseteq O_3'''$ 
this concludes the proof that $h'' \in \Can(\bR)$ and $\xi(h'')$ is cyclic. 
\end{proof}

\begin{corollary}\label{cor:cyclic}
Let $\bB$ be a first-order expansion of $\bR$. 
If there exists 
a uniformly continuous minor-preserving map 
from $\Pol(\bB) \cap \Can(\bR)$ to $\proj$, then 
$\eta$ 
or $\rho$ is a uniformly continuous minor-preserving map from $\Pol(\bB) \cap \Can(\bR,\prec)$ to $\proj$. 
\end{corollary}
\begin{proof}
We prove the contraposition, and suppose that
neither 
$\eta$ nor $\rho$ is a uniformly continuous minor-preserving map from $\Pol(\bB) \cap \Can(\bR,\prec)$ to $\proj$. By Proposition~\ref{prop:wedge}, 
$\Pol(\bB)$ contains 
a binary operation $f$ 
such that $\eta(f) = \wedge$. 
By Proposition~\ref{prop:cyclic}, 
$\mathscr C$ contains 
a ternary operation $g$ such that
$\rho(g)$ is cyclic. 
By Theorem~\ref{thm:cyclic}, 
$\mathscr C$ contains a ternary operation $c$ such that 
$\xi(c)$ is cyclic, and hence $\mathscr C$ does not
have a uniformly continuous minor-preserving map to $\proj$ by Proposition~\ref{prop:can-proj}. 
\end{proof}

\subsection{Verifying the UIP}
\label{sect:rcc5-uip}
In this section we show that if
$\eta$ or $\rho$ is a uniformly continuous map 
from $\Pol(\bC) \cap \Can(\bR,\prec)$ to
$\proj$, then it has the UIP. 

\begin{theorem}\label{thm:eta-uip}
Let $\bC$ be a first-order expansion of 
the structure $\bR$ such that
%$\eta \colon \Pol(\bC) \cap \Can(\bR,\prec) \to {\mathscr O}_{\{0,1\}}$. 
%$\mathscr C \subseteq \Pol(\bR) \cap \Can(\bR,\prec)$ be such that 
%$\eta \big (\Pol(\bC) \cap \Can(\bR,\prec) \big)$
$\eta$ is a uniformly continuous minor-preserving map from 
$\Pol(\bC) \cap \Can(\bR,\prec)$ to $\proj$. 
%is isomorphic to $\proj$. 
Then $\eta$ has the UIP with respect to $\Pol(\bC)$ over $(\bR,\prec)$.  
\end{theorem}

\begin{proof}
We use %Theorem~\ref{thm:binary-indep} via 
Lemma~\ref{lem:copy}. 
Let $\bA_1,\bA_2$ be the two independent 
elementary substructures of $(\bR,\prec)$ from 
Lemma~\ref{lem:rcc5indep}. 
%Let $f \in \Pol(\bC)^{(2)}$ and $u_1,u_2 \in \overline{\Aut(\bR,\prec)}$ be such that for $i \in \{1,2\}$ we have $u_i(R) \subseteq A_i$ and  $f(\id,u_i)$ is canonical with respect to $(\bR,\prec)$. Suppose for contradiction that 
%$\eta(f(\id,u_1)) \neq \eta ( f(\id,u_2) )$.
Note that 
\begin{align*}
\nsim & \subseteq \; \precnsim \, \circ \; \pp \circ \, \pp \\
 \text{ and } \nsim & \subseteq \, \pp \circ \, \pp \, \circ \precnsim 
%\nsim & \subseteq \precnsim \circ \pp \circ \pp \\
\end{align*}
so there are elements $a_{j,k}$ of $\bR$, for $j \in \{1,2\}$ and $k \in \{1,2,3,4\}$, such that
\begin{align*}
 & (a_{1,1},a_{1,2}) \in \pp,
  (a_{1,2},a_{1,3}) \in \pp, \\
 & (a_{1,3},a_{1,4}) \in \; \precnsim,  \; (a_{1,1},a_{1,4}) \in \; \nsim,  \\
 \text{ and } & (a_{2,1},a_{2,2}) \in \; \precnsim, \; (a_{2,2},a_{2,3}) \in \pp, \\
&  (a_{2,3},a_{2,4}) \in \pp,  \; 
 (a_{2,1},a_{2,4}) \in \; \nsim. 
 \end{align*}
 Let $F$ be a finite subset of the domain of $\bR$ and 2-rich with respect to 
 $\Pol(\bB)$ (see Lemma~\ref{lem:rich}); we may assume without loss of generality that $F$ contains $a_{j,k}$ for all $j \in \{1,2\}$ and $k \in \{1,2,3,4\}$. 
 
 If $\eta$ does not have the UIP, then we may apply Lemma~\ref{lem:copy} 
to $\eta$ and $F$;
let 
 $f \in \Pol(\bC)^{(2)}$ and 
 $\alpha_1, \alpha_2 \in \Aut(\bR;\prec)$ 
be the operations as in the statement of 
Lemma~\ref{lem:copy} so that 
$\eta(f(\id,\alpha_1)|_{F^2}) = \pi^2_{\ell}$ and
$\eta(f(\id,\alpha_2)|_{F^2}) = \pi^2_{\ell+1}$ 
where indices are considered modulo two.
Let $b_{j,k} \coloneqq \alpha_1(a_{j,k})$ for
 $k \in \{1,2\}$ and let $b_{j,k} \coloneqq \alpha_2(a_{j,k})$ for $k \in \{3,4\}$. Then $(b_{1,1},b_{1,2}),(b_{2,3},b_{2,4}) \in \; \pp$ and $(b_{1,3},b_{1,4}),(b_{2,1},b_{2,2}) \in \, \precnsim$, because $\alpha_i$ preserves $\pp$ and $\precnsim$.  Moreover, for $j \in \{1,2\}$ we have $(b_{j,2},b_{j,3}) \in \, \pp$ and $(b_{j,1},b_{j,4}) \in \; \precnsim$ by the properties of $\alpha_i$. This implies
 that 
 $$(f(a_{\ell,1},b_{\ell+1,1}),f(a_{\ell,4},b_{\ell+1,4})) \in \; \nsim.$$ On the other hand, we have that 
\begin{itemize}
\item 
$\big (f(a_{\ell,1},b_{\ell+1,1}),f(a_{\ell,2},b_{\ell+1,2}) \big) \in \pp$, because 
$(a_{1,1},a_{1,2}) \in \pp$, % for l=1
$(b_{1,1},b_{1,2}) \in \pp$,  % for l=2
and 
   $\eta(f(\id,\alpha_1)|_{F^2}) = \pi^2_\ell$. 
\item $\big (f(a_{\ell,2},b_{\ell+1,2}),f(a_{\ell,3},b_{\ell+1,3}) \big) \in \pp$,  because $(a_{\ell,2},a_{\ell,3}) \in \pp$
and $(b_{\ell+1,2}$, $b_{\ell+1,3}) \in \pp$. 
\item 
$\big (f(a_{\ell,3},b_{\ell+1,3}),f(a_{\ell,4},b_{\ell+1,4}) \big) \in \pp$, because 
   $\eta(f(\id,\alpha_2)|_{F^2}) = \pi^2_{\ell+1}$,
$(b_{2,3}$, $b_{2,4}) \in \pp$, % for l=1
and $(a_{2,3},a_{2,4}) \in \pp$.  % for l=2
\end{itemize}
Since $\pp \circ \pp \circ \pp = \pp$,
the pair 
$\big (f(a_{\ell,1},b_{\ell+1,1}),f(a_{\ell,4},b_{\ell+1,4}) \big)$
is also in $\pp$, a contradiction. The statement 
now follows from Theorem~\ref{thm:binary-indep}. 
 \end{proof}

The following theorem can be shown similarly 
as the previous theorem. 

\begin{theorem}\label{thm:rho-uip}
Let $\bC$ be a first-order expansion of 
$\bR$ such that
%$\eta \colon \Pol(\bC) \cap \Can(\bR,\prec) \to {\mathscr O}_{\{0,1\}}$. 
%$\mathscr C \subseteq \Pol(\bR) \cap \Can(\bR,\prec)$ be such that 
%the clone $\rho(\Pol(\bC) \cap \Can(\bR,\prec))$
%is isomorphic to $\proj$. 
$\rho$ is a uniformly continuous minor-preserving map from 
$\Pol(\bC) \cap \Can(\bR,\prec)$ to $\proj$.  
Then $\rho$ has the UIP with respect to $\Pol(\bC)$ over $(\bR,\prec)$.  
\end{theorem}

\begin{proof}
We again use Lemma~\ref{lem:copy}. 
% Theorem~\ref{thm:binary-indep};
the proof is very similar to  the proof
of Theorem~\ref{thm:eta-uip}. 
Let $\bA_1,\bA_2$ be the two independent 
elementary substructures of $(\bR,\prec)$ from 
Lemma~\ref{lem:rcc5indep}. 
Let $f \in \Pol(\bC)^{(2)}$ and $u_1,u_2 \in \overline{\Aut(\bR,\prec)}$ be such that for $i 
\in \{1,2\}$ we have $\Im(u_i) \subseteq A_i$ and  $f(\id,u_i)$ is canonical with respect to $(\bR,\prec)$. Suppose for contradiction that 
$\rho(f(\id,u_1)) \neq \rho( f(\id,u_2) )$.
Note that 
\begin{align*}
\pp \subseteq  \; & (\drprec) \circ \pp \circ (\drprec) \\
\pp \subseteq \; & (\poprec) \circ \pp \circ (\poprec) %\\
%& (\poprec) \circ \pp \circ (\drprec) \cap \pp = \emptyset
\end{align*}
so there are elements $a_{j,k}$ of $\bR$, for $j \in \{1,2\}$ and $k \in \{1,2,3,4\}$, such that
\begin{align*}
 & (a_{1,1},a_{1,2}) \in \drprec, \; (a_{1,2},a_{1,3}) \in \pp, \\
 & (a_{1,3},a_{1,4}) \in \drprec,  \; (a_{1,1},a_{1,4}) \in \pp,  \\
 & (a_{2,2},a_{2,2}) \in \poprec, \; (a_{2,2},a_{2,3}) \in \pp, \\
 & (a_{2,3},a_{2,4}) \in \poprec,  \; (a_{2,1},a_{2,4}) \in \pp. 
 \end{align*}
 Let $F$ be a finite subset of the domain of $\bR$ which is 2-rich with respect to 
 $\Pol(\bC)$ (see Lemma~\ref{lem:rich}); we may assume without loss of generality that $F$ contains $a_{j,k}$ for all $j \in \{1,2\}$ and $k \in \{1,2,3,4\}$. 
  Let $\alpha_1$ and $\alpha_2$ be the automorphisms of $(\bR;\prec)$
obtained from applying Lemma~\ref{lem:copy} to $\rho$ and $F$ so that 
$\rho(f(\id,\alpha_1)|_{F^2}) = \pi^2_{\ell}$ and
$\rho(f(\id,\alpha_2)|_{F^2}) = \pi^2_{\ell+1}$ 
where indices are considered modulo two. 

Let $b_{j,k} \coloneqq \alpha_1(a_{j,k})$ for
 $k \in \{1,2\}$ and let $b_{j,k} \coloneqq \alpha_2(a_{j,k})$ for $k \in \{3,4\}$. Then 
 $(b_{1,1},b_{1,2}),(b_{1,3},b_{1,4}) \in \; \drprec$ and $(b_{2,1},b_{2,2}),(b_{2,3},b_{2,4})\in \, \poprec$, because $\alpha_i$ preserves $\drprec$ and $\poprec$.  Moreover, for $j \in \{1,2\}$ we have $(b_{j,2},b_{j,3}),(b_{j,1},b_{j,4})  \in \, \pp$ by the properties of $\alpha_i$. This implies
 that 
 $(f(a_{\ell,1},b_{\ell+1,1}),f(a_{\ell,4},b_{\ell+1,4})) \in \; \pp$. On the other hand, 
\begin{itemize}
\item 
$(f(a_{\ell,1},b_{\ell+1,1}),f(a_{\ell,2},b_{\ell+1,2})) \in \drprec$, because 
   $\eta(f(\id,\alpha_1)|_{F^2}) = \pi^2_\ell$,
$(a_{1,1}$, $a_{1,2}) \in \drprec$, % for l=1
and $(b_{1,1},b_{1,2}) \in \drprec$.  % for l=2
\item $(f(a_{\ell,2},b_{\ell+1,2}),f(a_{\ell,3},b_{\ell+1,3})) \in \pp$,  because $(a_{\ell,2},a_{\ell,3}) \in \pp$
and $(b_{\ell+1,2}$, $b_{\ell+1,3}) \in \pp$. 
\item 
$(f(a_{\ell,3},b_{\ell+1,3}),f(a_{\ell,4},b_{\ell+1,4})) \in \poprec$, because 
   $\eta(f(\id,\alpha_2)|_{F^2}) = \pi^2_{\ell+1}$,
$(b_{2,3},b_{2,4}) \in \poprec$, % for l=1
and $(a_{2,3},a_{2,4}) \in \poprec$.  % for l=2
\end{itemize}
Since $\big ((\drprec) \circ \pp \circ (\poprec) \big) \cap \pp = \emptyset$, we reached a contradiction. The statement 
now follows from Theorem~\ref{thm:binary-indep}. 
\end{proof}

\subsection{Proof of the complexity classification}
\label{sect:rcc5-proof}
In this section we prove Theorem~\ref{thm:rcc5}.
Let $\bC$ be a first-order expansion of $\bR$. 
We verify that $\bC$ satisfies the assumptions
of Corollary~\ref{cor:dicho}. 
The structure $\bR$ is finitely bounded by Lemma~\ref{lem:R-fin-bound},
and $\bR$ has the finitely homogeneous Ramsey expansion $(\bR,\prec)$ by Theorem~\ref{thm:ramsey-exp}.
If $\Pol(\bC) \cap \Can(\bR)$ has a uniformly continuous minor-preserving map to $\proj$, then by Corollary~\ref{cor:cyclic}
either $\eta$ or $\rho$ is a uniformly 
continuous minor-preserving map from
$\Pol(\bC) \cap \Can(\bR,\prec)$ to $\proj$. 
In the first case, Theorem~\ref{thm:eta-uip} 
implies that $\eta$ has the UIP with respect to $\Pol(\bC)$ and in the second case
Theorem~\ref{thm:rho-uip} implies
that $\rho$ has the UIP with respect to $\Pol(\bC)$. The complexity dichotomy follows
from Corollary~\ref{cor:dicho}. 
\qed

\section{Expressibility in Datalog and Extensions}
\label{sect:Datalog}
Let $\bC$ be a finite-signature first-order expansion of $\bR$. We show that if the first condition of Theorem~\ref{thm:rcc5} holds then $\Csp(\bC)$ is in Datalog. %By Theorem~\ref{thm:cyclic} and the results of Section 8.6. we know that in this case the clone $\xi_2(\Pol(\bC) \cap \Can(\bR))$ contains a ternary cyclic operation. 
By Lemma 4.9 in~\cite{BodMot-Unary} there exists a finite structure with a finite binary relational signature so that $\Pol(\bD)=\xi_3^\bR(\Pol(\bC) \cap \Can(\bR))$. 
	We use the following variant of Theorem~\ref{thm:can-tract}, which follows from the results in~\cite{BodMot-Unary}. 
	
\begin{lemma}\label{can_datalog}
	Let $\bD$ be a finite-signature structure so that $\Pol(\bD)=\xi_3^\bR(\Pol(\bC) \cap \Can(\bR))$. 
	If $\Csp(\bD)$ is in Datalog then so is $\Csp(\bC)$.
\end{lemma}

\begin{proof}
	 The statement follows from the discussion after the proof of Theorem 4.10 in~\cite{BodMot-Unary} considering that $\bR$ is binary and 
	 because there exists a set of at most 3-element structures  ${\mathcal F}$ such that 
	 the age of $\bR$ equals the class of all finite structures that do not embed any of the structures from ${\mathcal F}$. 
\end{proof}

	The following sufficient condition for $\Csp(\bD)$ to be in Datalog is due to~\cite{BoundedWidthJournal}.
	
	\begin{theorem}\label{thm:bounded-width}
	Let $\bD$ be a finite structure with finite relational signature. If $\bD$ has for all but finitely many $n \in {\mathbb N}$ a WNU polymorphism of arity $n$. Then $\Csp(\bD)$ can be solved in Datalog. 
	\end{theorem}

	The condition in Theorem~\ref{thm:bounded-width} has many equivalent characterisations in universal algebra. 
	A clone $\mathscr C$ with domain $A$ is called \emph{affine} (see~\cite{FreeseMcKenzie}) if there is an abelian group $(A;+,-)$ such that 
	\begin{itemize}
	\item $(x,y,z) \mapsto x-y+z$ is in ${\mathscr C}$,
	and
	\item for all $a,b,c \in A^n$ we have
	$f(a-b+c) = f(a) - f(b) + f(c)$ for every $f \in {\mathscr C}^{(n)}$. 
	\end{itemize}
		
	%such that there exists a clone ${\mathscr C} \in {\mathcal K}$ and a surjective function $h$ from the domain of ${\mathscr C}$ to the domain of ${\mathscr C'}$ such that 
	%for every $f \in {\mathscr C}^{(k)}$ there exists an operation $f' \in {\mathscr C'}$ with 
	%$f'(h(x_1),\dots,h(x_k)) = h(f(x_1,\dots,x_k))$ for all $x_1,\dots,x_k$ from the domain of ${\mathscr C}$. Elements of $H(S({\mathscr C}))$ are also called \emph{subfactors} of ${\mathscr C}$. 
	% source: Freese-McKenzie
	
	\begin{theorem}[Lemma 4.5 in~\cite{Strong-Zhuk}]\label{thm:wnu-datalog}
	 Let ${\mathscr C}$ be a clone of idempotent operations on a finite set without a WNU term operation of arity $n \geq 3$. Then 
	 %$\HS({\mathscr C})$ contains an 
${\mathscr C}$ has a trivial or affine subfactor with a domain of size $k \geq 2$. 
\end{theorem}

\begin{lemma}\label{no_affine}
Let $\bC$ be a first-order expansion of $\bR$ 
such that $\Pol(\bC)$ does not have a uniformly continuous minor-preserving map to $\proj$. Then 
	none of the subfactors of size $k \geq 2$ of the clone $\mathscr{C}:=\xi_2^\bR(\Pol(\bC) \cap \Can(\bR))$ 
	contains an affine clone. 
	%\red{with domain $\mathbb{Z}^k_p$ ($p$ is prime, $k \geq 1$)} contains 
%	the operation $(x,y,z)\mapsto x-y+z$.
\end{lemma}

\begin{proof}
	Let us assume for contradiction that
	${\mathscr C}$ has an affine subfactor ${\mathscr C}'$ with at least two elements. 
	Let 
	%$h$ be the surjective map from a  ... to the domain of ${\mathscr C}'$ 
	$E$ be the equivalence relation on a 
	subset $S$ of the domain of ${\mathscr C}$ that witnesses this. We may assume that $S$ is chosen to be minimal satisfying these conditions.
	Let $f \in \Pol(\bC) \cap \Can(\bR)$ 
	be 
	such that $(\xi^{\bR}_2(f)|_S)/E$ is given by $(x,y,z) \mapsto x-y+z$ where $+$ and $-$ refer to the underlying abelian group of ${\mathscr C'}$. 
	%Let $f\in \Can(\bR)$ be such that $\xi_2^\bR(f)=\tilde{f}$. 
	By the canonisation lemma (Lemma~\ref{lem:canon}) we know that $f$ interpolates some function $f'$ over $\Aut(\bR,\prec)$ which is canonical over $(\bR,\prec)$. Then clearly $\xi_2^\bR(f')=\xi_2^\bR(f)$ so we may also assume without loss of generality that $f\in \Can(\bR,\prec)$.
		
	\medskip 
{\bf Claim 1.} $S\subseteq \{\pp,\ppi,\po,\dr\}$. 
Suppose for contradiction that this is not the case;
that is, one of the $E$-classes contains $\eq$. 
First consider the case that this class is distinct from $\{\eq\}$. 
By Proposition~\ref{lem:pres-neq} 
we know that $\{\pp,\ppi,\po,\dr\}$ is preserved by $\xi_2^\bR(\Pol(\bB) \cap \Can(\bR,\prec))$; hence, we may intersect each class with 
$\{\pp,\ppi,\po,\dr\}$ and replace $S$ by $S \setminus \{\eq\}$ and obtain a contradiction to the minimality of $S$. 
%\red{Even without $\prec$, but an argument is missing here.
%We need that the $\xi_2^\bR$ images of $\Pol(\bB) \cap \Can(\bR)$ and
%$\Pol(\bB) \cap \Can(\bR)\cap \Can(\bR,\prec)$ are the same.}
% \textcolor{red}{This part is true even without the assumption on $\eta$.} 
%Therefore, by the minimality of $S$ it is enough to show that every $E$-class intersects $\bf 1\setminus \{\eq\}$. \red{Note that this is equivalent to the statement that $\{\eq\}$ is not a $E$-class:  Clearly, $\{\eq\}$ does not intersect $\bf 1 \setminus \{\eq\}$ and hence is not an $E$-class, and conversely if an $E$-class intersects $1 \setminus \Eq$, then by intersecting all classes with $1 \setminus }
 %Let us assume that this is not the case.
We may therefore suppose that $\{\eq\}$ is one of the $E$-classes. 
Arbitrarily choose $O \in S\setminus \{\eq\}$.
%Since 
%$\xi_2^\bR(f)|_S/\tau$ 
%$(\xi^{\bR}_2(f)|_S)/E$ 
%s a Maltsev operation 
By assumption we have $\xi_2^{\bR}(f(\eq,O,O)) =  \xi_2^{\bR}(f)(O,O,\eq) = \eq$. By the table in Figure~\ref{table:rcc} and Lemma~\ref{lem:decomp} we know that $O\subseteq O\circ O$. Therefore, $$f(O,O,O)\subseteq f(\eq\circ O,O\circ O,O\circ \eq)\subseteq \eq\circ \eq=\eq,$$ a contradiction.

\medskip 
{\bf Claim 2.} $S \neq \{\po,\dr\}$. Otherwise, $\xi_2^{\bR}(f)|_{\{\po,\dr\}}$ would be a minority operation and $f(\dr,\po,\po)\subseteq \dr$ and $f(\pp,\pp,\pp)\subseteq \pp$. By Table~\ref{table:rcc} we know that $\dr=\pp\circ \dr$ and $\dr,\po\subseteq \pp\circ \po$. Therefore $f(\dr,\dr,\po)\subseteq \pp\circ \dr=\dr$ which contradicts the assumption that 
	$\xi_2^{\bR}(f)|_{\{\po,\dr\}}$ 
	is a minority operation.

%\medskip 
%\red{Claim 1 and 2 imply that we may assume that $S$ contains $\pp$ or $\ppi$. Assume that $S$ contains $\pp$; the case that $S$ contains $\ppi$ can be treated similarly.}

\medskip 
{\bf Claim 3.} $S \neq \{\pp,\ppi\}$. Otherwise, $\xi_2^{\bR}(f)|_{\{\pp,\ppi\}}$ is a minority operation and 
$$f(\pp,\ppi,\ppi),f(\ppi,\pp,\ppi) \subseteq \pp.$$ On the other hand, by the table in Figure~\ref{table:rcc} we know that 
$$\po \subseteq \pp \circ \ppi \circ \ppi, \ppi \circ \pp \circ \ppi, \ppi \circ \ppi \circ \pp.$$
Thus, $f(\po,\po,\po) \subseteq \pp \circ \pp \circ \pp = \pp$, contradicting that 
$f$ %\in \Pol(\bR)$
 preserves $\po$.

\medskip 
Claim 3 implies that there exists $C_0 \in S \setminus \{\pp,\ppi\}$. 
Pick $C_1 \in S$ such that $C_0$ and $C_1$ are in distinct $E$-classes.
We claim that $C_1 \in \{\pp,\ppi\}$. Suppose for contradiction that $C_1 \in \{\po,\dr\}$. Since $\po \cup \dr$ is preserved by $f$ by Lemma~\ref{lem:partial-can}, 
it follows that $S = \{C_0,C_1\} = \{\po,\dr\}$ by the minimal choice of $S$, in contradiction to Claim 2.
%red{Since $\po \cup \dr$ is preserved by $f$ by Lemma~\ref{lem:partial-can}, 
%$\po$ and $\dr$ must be $E$-equivalent, 
%because otherwise $S = \{C_0,C_1\} = \{\po,\dr\}$ by the minimal choice of $S$, in contradiction to Claim 2. 
%In particular, we must have that $C_1 \notin \{\po,\dr\}$. }
Hence, $C_1 = \pp$ or $C_1 = \ppi$ by Claim 1. 
We consider the case that $C_1 = \pp$; the case that $C_1 = \ppi$ can be treated similarly.
%By the reasoning above we may also assume that if $S$ contains both $\po$ and $\dr$, then they must lie in the same $E$-class.}

Recall from the beginning of Section~\ref{sect:eta} that $C_0 \cap {\prec}$ and $\pp$ are in different $F$-classes. If $P,Q \in \{C_0 \cap {\prec}, \pp\}$, then 
by the assumptions for $f$ we have that 
%$f(P,P,Q)$ is contained in $\prec$. Since 
 $\xi_2(f)(P,P,Q)$ is in the same $E$-class as $Q$. 
 %,because $(\xi_2(f)_S)/E$ is a Maltsev operation. 
Note that both $Q$ and $f(P,P,Q)$ are contained in $\prec$. Hence, if $Q$ is $\pp$ then $f(P,P,Q)$ is contained in $\pp$,
and if $Q$ is $C_0 \, \cap \prec$ then $f(P,P,Q)$ is contained in $\dr \; \cap \prec$ or in $\po \; \cap \prec$. It follows that
$f(P,P,Q)$ and $Q$ are in the same $F$-class.  Likewise, we deduce
that $f(Q,P,P)$ and $Q$ lie in the same $F$-class. We conclude that 
the operation $\eta(f)$ is the ternary minority operation, which is a contradiction to  
Proposition~\ref{prop:wedge}.
\end{proof}

The following result is a dichotomy result that does not rely on any complexity-theoretic assumptions.  

\begin{theorem}\label{thm:data}
Let $\bC$ be a finite-signature first-order expansion of $\bR$. If $\Pol(\bC)$ has a uniformly continuous minor-preserving map to $\proj$, then $\Pol(\bC)$ is not expressible in fixed-point logic plus counting (FPC). Otherwise, $\Csp(\bC)$ is solvable in Datalog.
\end{theorem}

\begin{proof}
If $\Pol(\bC)$ has a uniformly continuous minor-preserving map to $\proj$, then 
 $\bB$ can constructs primitively positively all finite structures (see, e.g.,~\cite{Book}), and 
it follows from work in~\cite{AtseriasBulatovDawar} 
that $\Csp(\bC)$ is
inexpressible in FPC.

Otherwise, Lemmas~\ref{no_affine} implies that none of the subfactors of
${\mathscr C} := \xi_2^{\bR}(\Pol(\bC) \cap \Can(\bR))$ 
%with domain ${\mathbb Z}^k_p$, for $p$ prime and $k \geq 1$, contains the operation $(x,y,z) \mapsto x-y+z$. It follows that ${\mathscr C}$ contains no affine subfactor with a domain of size at least two. 
contains an affine clone with a domain of size at least two. 
Theorem~\ref{thm:wnu-datalog} implies that
$\mathscr C$ has for every $k \geq 3$ a WNU of arity $k$; let $w_k \in \Pol(\bC) \cap \Can(\bR)$ be such that $\xi_2^{\bR}(w_k)$ is a WNU of arity $k$. 
Since $\bR$ is homogeneous in a binary language,  the operation $\xi_3^{\bR}(w_k)$ 
%, \red{which implies that
%$\xi_3^{\bR}(\Pol(\bC) \cap \Can(\bR)$ has for every $k \geq 2$ 
% and
is a WNU polymorphism of arity $k$. 
Let $\bD$ be a finite-signature structure so that $\Pol(\bD)=\xi_3^\bR(\Pol(\bC) \cap \Can(\bR))$.
Then Theorem~\ref{thm:bounded-width} implies that	$\Csp(\bD)$ is in Datalog. 
	We can therefore conclude with Lemma~\ref{can_datalog}  that $\Csp(\bC)$ is in Datalog.
\end{proof}

\section{Future Work}
The techniques to verify the UIP presented here can also be used
to verify the tractability conjecture for the class ${\mathcal C}$ of all
 structures where the number of orbits of 
 tuples with pairwise distinct entries is bounded by 
  $c n^{dn}$
for constants $c,d$ with $d<1$.
The automorphism groups of the structures in ${\mathcal C}$ have been described  in~\cite{BodirskyBodor}; 
 they include all first-order reducts of unary structures, and in particular all finite structures. 
It can be shown that ${\mathcal C}$ is precisely Lachlan's class of \emph{cellular structures}; an explicit description of the automorphism groups of such structures can be found in~\cite{BodorMonadicallyStable}. 

In fact, ${\mathcal C}$ can be further generalised to the class of all \emph{hereditarily cellular structures}, which are precisely the $\omega$-categorical monadically stable structures~\cite{Lachlan-Tree-Decomp};
the growth-rates of the number of orbits is discussed in~\cite{Braunfeld-Monadic-Stab}.  
In the thesis of the second author~\cite{BodorDiss}, 
the tractability conjecture can be verified even for this larger class, again using the unique interpolation property and the results from the present paper. 

We propose two larger classes of structures where the tractability conjecture has not yet been classified:
\begin{itemize}
\item The class of all structures with a first-order interpretation in the structure $({\mathbb N};=)$.  This class includes for example all first-order reducts of the line graph of an infinite clique; these structures are not covered by any of the classes mentioned above. 
\item The class of all structures where the number of orbits of the action of the automorphism group on sets of size $n$ is bounded by some polynomial in $n$; the automorphism groups of such structures have been classified by Falque and Thierry~\cite{FalqueThiery}. This is an interesting class and a challenge to the method of the present paper because there are polynomial-time tractable cases where the canonical polymorphisms have a uniformly continuous minor-preserving map to the projections (some of these cases are described in~\cite{RydvalDescr}). 
%~\cite{tcsps-journal}). 
\end{itemize}
Concerning RCC5 we ask whether 
 every first-order expansion of $\bR$ that can be solved by Datalog can also be solved by a Datalog program with IDBs of maximal arity two. 
 %the 
 %famous \emph{path-consistency procedure}~\cite{Qualitative-Survey}. 
 The same statement is known to be true for finite-domain CSPs~\cite{BoundedWidth}. But we cannot deduce the statement above from the result for finite domains via Lemma~\ref{can_datalog}, because the reduction does not preserve the arity of the IDBs.
 %uses IDBs whose size is twice the maximal obstruction set, 
 
\section*{Acknowledgements}
The authors thank Thomas Quinn-Gregson, all the referees of the conference version~\cite{BodirskyBodorUIP}, and the referees of the present journal version for their helpful comments. 

\bibliographystyle{alpha}
% argument is your BibTeX string definitions and bibliography database(s)
%\bibliography{IEEEabrv}
\bibliography{../../global}

\end{document}